\documentclass[10pt,twoside,a4paper,english,leqno]{article}

\title{A new class of two-layer Green-Naghdi systems with improved frequency dispersion}
\author{Vincent Duch\^ene%
\thanks{IRMAR - UMR6625, CNRS and Univ. Rennes 1, Campus de Beaulieu, F-35042 Rennes cedex, France. ({\tt vincent.duchene@univ-rennes1.fr}). VD is partially supported by the project Dyficolti ANR-13-BS01-0003-01 of the Agence Nationale de la Recherche.}
\and Samer Israwi%
\footnotemark[2]
\and Raafat Talhouk%
\thanks{Math\'ematiques, Facult\'e des sciences I et Ecole doctorale des sciences et technologie, Universit\'e Libanaise, Beyrouth, Liban. ({\tt rtalhouk@ul.edu.lb}, {\tt s\_israwi83@hotmail.com}). SI and RT are partially supported by the Lebanese  University research program (MAA group project).}
}
\date{\today}

\usepackage{babel}
\usepackage[utf8]{inputenc}
\usepackage[T1]{fontenc}
\usepackage{amsmath, amsthm}
\usepackage{amsfonts,amssymb}
\usepackage{float}
\usepackage[pdftex]{graphicx}
\usepackage{fancyhdr}
\usepackage{url}
\usepackage{hyperref} 
\usepackage{stmaryrd}
\usepackage{mathtools} 
\usepackage{titling}

\makeatletter
\let\Title\@title
\let\Author\@author
\makeatother

\fancypagestyle{mystyle}{%
	\fancyfoot{}
	\fancyhead[CO]{\Title} 
	\fancyhead[CE]{\scriptsize V. Duch{\^e}ne, S. Israwi and R. Talhouk} 
	\fancyhead[RE]{\scriptsize\today}
	\fancyhead[LO]{}
	\fancyhead[LE,RO]{\scriptsize\thepage}
	}
\fancypagestyle{plain}{\fancyhead{} } 
\numberwithin{equation}{section}
\newcommand{\RR}{\mathbb{R}}
\newcommand{\NN}{\mathbb{N}}
\renewcommand{\t}{\tilde}
\newcommand{\e}{\textbf{e}}
\newcommand{\m}{\mathfrak{m}}
\newcommand{\N}{\mathcal{N}}
\newcommand{\A}{\mathcal{A}}
\newcommand{\Q}{\mathcal{Q}}
\newcommand{\R}{\mathcal{R}}
\renewcommand{\O}{\mathcal{O}}
\renewcommand{\S}{\mathcal{S}}
\newcommand{\dd}{{\rm d}}
\newcommand{\F}{{\sf F}}

\newcommand{\eqdef}{\stackrel{\rm def}{=}}
\DeclareMathOperator{\Bo}{Bo}
\DeclareMathOperator*{\esssup}{ess\,sup}

\newcommand{\nn}{\nonumber}
\newcommand{\ie}{{\em i.e.}~}
\newcommand{\eg}{{\em e.g.}~}
\newcommand{\id}[1]{\left\vert_{_{#1}}\right.}
\newcommand{\dsp}{\displaystyle}

\newcommand{\W}{W_{\Bo^{-1}}}
\newcommand{\X}{X_{\Bo^{-1}}}
\newcommand{\Y}{Y_{\F^\mu}}
\newcommand{\Z}{Z_{\F^\mu}}

\DeclarePairedDelimiter\norm{\big\lvert}{\big\rvert}
\DeclarePairedDelimiter\Norm{\big\lVert}{\big\rVert}

\newtheorem{Theorem}{Theorem}[section]
\newtheorem{Definition}[Theorem]{Definition}
\newtheorem{Proposition}[Theorem]{Proposition}

\newtheorem{Lemma}[Theorem]{Lemma}
\newtheorem{Remark}[Theorem]{Remark}

\begin{document}
\maketitle

\begin{abstract}
We introduce a new class of Green-Naghdi type models for the propagation of internal waves between two ($1+1$)-dimensional layers of homogeneous, immiscible, ideal, incompressible, irrotational fluids, vertically delimited by a flat bottom and a rigid lid. These models are tailored to improve the frequency dispersion of the original bi-layer Green-Naghdi model, and in particular to manage high-frequency Kelvin-Helmholtz instabilities, while maintaining its precision in the sense of consistency. Our models preserve the Hamiltonian structure, symmetry groups and conserved quantities of the original model. We provide a rigorous justification of a class of our models thanks to consistency, well-posedness and stability results. These results apply in particular to the original Green-Naghdi model as well as to the Saint-Venant (hydrostatic shallow-water) system with surface tension.
\end{abstract}

\section{Introduction}
\subsection{Motivation}
This work is dedicated to the study of a bi-fluidic system which consists in two layers of homogeneous, immiscible, ideal and incompressible fluids under only the external influence of gravity. Such a configuration is commonly used in oceanography, where variations of temperature and salinity induce a density stratification; see~\cite{HelfrichMelville06} and references therein.

A striking property of the above setting, in contrast with the water-wave case (namely only one layer of homogeneous fluid with a free surface) is that the Cauchy problem for the governing equations is ill-posed outside of the analytic framework when surface tension is neglected~\cite{Ebin88,IguchiTanakaTani97,KamotskiLebeau05}. This ill-posedness is caused by the formation of high-frequency (\ie small wavelength) Kelvin-Helmholtz instabilities which are triggered by any non-trivial velocity shear. Recently, Lannes~\cite{Lannes13} showed that a small amount of surface tension is sufficient to durably regularize the high-frequency component of the flow, while the main low-frequency component remains mostly unaffected.

This result explains why, occasionally, surface tension may be harmlessly neglected in asymptotic models, that is simplified models constructed from smallness assumptions on physical properties of the flow. This is typically expected to be the case for models derived in the so-called shallow-water regime, which implies that the main component of the flow is located at low frequencies; and in particular for the two-layer extension of the classical Green-Naghdi model introduced by Miyata~\cite{Miyata85,Miyata87}, Mal'tseva~\cite{Malcprimetseva89} and Choi and Camassa~\cite{ChoiCamassa99}. 
However, as noticed in~\cite{JoChoi02} a linear stability analysis indicates that the bi-fluidic Green-Naghdi system actually overestimates Kelvin-Helmholtz instabilities, in the sense that the threshold on the velocity shear above which high-frequency instabilities are triggered is always smaller for the model than for the full Euler system.

Many attempts have been made in order to ``regularize'' the Green-Naghdi model, that is proposing new models with formally the same precision as the original model, but which are not subject to high-frequency Kelvin-Helmholtz instabilities, even without surface tension~\cite{NguyenDias08,ChoiBarrosJo09,BoonkasameMilewski14,DucheneIsrawiTalhouk,LannesMing}. The strategies adopted in these works rely on change of unknowns and/or Benjamin-Bona-Mahony type tricks; see~\cite[Section~5.2]{Lannes} for a thorough presentation of such methods in the water-wave setting. In this work, we present a new class of modified Green-Naghdi systems obtained through a different, somewhat simpler mean. We find numerous advantages in our method:
\begin{itemize}
\item The original Green-Naghdi model is only lightly modified, and the physical interpretations of variables and identities of the original model are preserved.
\item The method is quite flexible. It allows in particular to construct models which completely suppress large-frequency Kelvin-Helmholtz instabilities; or a model which conforms perfectly with the linear stability analysis of the full Euler system.
\item The rich structure of the original Green-Naghdi system (Hamiltonian formulation, groups of symmetry, conserved quantities) is maintained. This is generally not the case when change of unknowns or BBM tricks are involved; see discussion in~\cite{Christov01,DuranDutykhMitsotakis13}.
\end{itemize}
Our models may be viewed as Green-Naghdi systems with improved frequency dispersion. In particular, one of our models shares with full dispersion models such as in~\cite{Whitham,BonaLannesSaut08,SautXu12a,Obrecht14} the property that their dispersion relation is the same as the one of the full Euler system. Notice however that, consistently with the derivation of the original Green-Naghdi model and contrarily to the aforementioned ones, our models do not rely on an assumption of small amplitude: we only assume that the wavelength is large compared with the depth of the two layers (see the consistency result given in Proposition~\ref{P.consistency}). To the best of our knowledge, there does not exist any other model in the literature with such properties, even when only one layer of fluid is involved. A noteworthy feature of our models is that, by construction, they involve non-local operators (Fourier multipliers). Such operators are common in deep-water models or small-amplitude models, such as the Benjamin-Ono or Whitham equations for instance, but appear to be original in the shallow-water setting. 

Our new class of models is obtained by regularizing the original Green-Naghdi one and not from a direct derivation from the full Euler system. As it is made clear in Section~\ref{S.Hamiltonian}, below, we could also derive our models from Hamilton's principle, but once again our approximate Hamiltonian functional consists only in a harmless perturbation of the original Green-Naghdi one. It would be interesting to derive our models directly from an expansion of the so-called Dirichlet-Neumann operators, and investigate whether the full dispersion model is more precise than the original one when only small nonlinearities are involved.

We also acknowledge that the use of surface tension in view of modeling miscible fluids of different densities, such as water and brine, is arguable. We view surface tension as a simple and effective way of regularizing the flow. Let us point out that the construction of our new models does not rely on any particular structure of the surface tension component, thus could be applied with any additive regularizing term. Our systems may be more evidently applicable to two-layer systems of genuinely immiscible, homogeneous and inviscid fluids. We believe they may also be of interest in the situation of one layer of homogeneous fluid, with or without surface tension.

For the sake of simplicity, our study is restricted to the setting of a flat bottom, rigid lid and one-dimensional horizontal variable. The construction of our models, however, may straightforwardly be extended to the two-dimensional case. We also expect that our strategy can be favorably applied to more general configurations (non-trivial topography, free surface, multi-layer, {\em etc.})
\medskip

In the present work, we motivate our models through the study of Kelvin-Helmholtz instabilities, by linearizing the systems around solutions with constant shear. This formal study is supported by numerical simulations, which demonstrate how the predictions of the modified Green-Naghdi models may vary dramatically depending on their large-frequency dispersion properties, and the significant influence of small surface tension. We also provide a rigorous analysis for a class of our models by proving the well-posedness of the Cauchy problem as well as consistency and stability results, which together offer the full justification of our asymptotic models, in the sense described in~\cite{Lannes}. This includes the original Green-Naghdi model as well as the Saint-Venant (hydrostatic shallow-water) system with surface tension, for one or two layers of fluid; all these results being original as far as we know.

\subsection{The full Euler system}\label{S.fE}
 For the sake of completeness and in order to fix the notations, we briefly recall the governing equations of a two-layer flow in our configuration, that we call {\em full Euler system}. We let the interested reader refer to~\cite{BonaLannesSaut08,Anh09,Duchene14} for more details.

The setting consists in two layers (infinite in the horizontal variable, vertically delimited by a flat rigid lid and a flat bottom) of immiscible, homogeneous, ideal, incompressible and irrotational fluid under only the external influence of gravity. We assume that the interface between the two layers is given as the graph of a function, $\zeta(t,x)$, so that the domain of the two fluids at time $t$ is
\[ \Omega_1^t\eqdef \{(x,y),\ \zeta(t,x)\leq z\leq d_1\},\quad \Omega_2^t\eqdef \{(x,z),\ -d_2\leq z\leq \zeta(t,x)\} .\]
Here and thereafter, the subscript $i=1$ (resp. $i=2$) always refer to the upper (resp. lower) layer. 
The fluids being irrotational, we consider the velocity potentials in each layer, that we denote $\phi_i$. Finally, $P_i$ denotes the pressure inside each layer.

Let $a$ be the maximum amplitude of the deformation of the interface. We denote by $\lambda$ a characteristic horizontal length, say the wavelength of the interface. The celerity of infinitely long and small internal waves is given by
\[c_0 \ = \ \sqrt{g\frac{(\rho_2-\rho_1) d_1 d_2}{\rho_2 d_1+\rho_1 d_2}},\]
where $d_1$ (resp. $d_2$) is the depth of the upper (resp. lower) layer and $\rho_1$ (resp. $\rho_2$) its mass density. $g$ denotes the acceleration of gravity.
Consequently, we introduce the dimensionless variables
\[\begin{array}{cccc}
\t z \ \eqdef\ \dfrac{z}{d_1}, \quad\quad & \t x\ \eqdef \ \dfrac{x}{\lambda}, \quad\quad & \t t\ \eqdef\ \dfrac{c_0}{\lambda}t,
\end{array}\]
the dimensionless unknowns
\[\begin{array}{c}
\t{\zeta}(\t t,\t x)\eqdef \dfrac{\zeta(t,x)}{a}, \quad \t{\phi_i}(\t t,\t x,\t z) \eqdef \dfrac{d_1}{a\lambda c_0}\phi_i(t,x,z), \quad \t{P_i}(\t t,\t x,\t z) \eqdef \dfrac{d_1}{a\rho_2 c_0^2}P_i(t,x,z) \quad (i=1,2),
\end{array}\]
as well as the following dimensionless parameters
\[
\gamma \eqdef \dfrac{\rho_1}{\rho_2}, \quad \epsilon \eqdef \dfrac{a}{d_1}, \quad \mu \eqdef \dfrac{d_1^2}{\lambda^2}, \quad \delta \eqdef \dfrac{d_1}{d_2}, \quad \Bo \eqdef \dfrac{g(\rho_2-\rho_1)\lambda^2}{\sigma}.
\]
The last parameter is the Bond number, and measures the ratio of gravity forces over capillary forces ($\sigma$ is the surface tension coefficient). After applying the above scaling, but withdrawing the tildes for the sake of readability, the system may be written as
\begin{equation}\label{fE}
\left\{ \begin{array}{lr}
\mu \partial_x^2 \phi_i \ + \ \partial_z^2 \phi_i \ = \ 0 & \quad \text{ in }\Omega^t_i \quad (i=1,2),\\
\partial_{z}\phi_1 \ = \ 0 & \mbox{on } \{(x,z),z=1\}, \\
\partial_{z}\phi_2 \ = \ 0 & \mbox{on } \{(x,z),z=-\delta^{-1}\}, \\
\partial_t \zeta = \frac1\mu \sqrt{1+\mu\epsilon^2|\partial_x\zeta|^2}\partial_{n}\phi_1 = \frac1\mu \sqrt{1+\mu\epsilon^2|\partial_x\zeta|^2}\partial_{n}\phi_2 \quad & \mbox{on } \{(x,z),z=\epsilon\zeta(t,x)\},\\
\partial_t \phi_i+\frac{\epsilon}{2} |\nabla_{x,z}^\mu\phi_i|^2=-\frac{P}{\gamma^{2-i}}-\frac{\gamma+\delta}{1-\gamma}z &\quad \text{ in }\Omega^t_i \quad (i=1,2),\\
\llbracket P(t,x) \rrbracket \ = \ -\frac{\gamma+\delta}{\Bo} \frac{k(\epsilon\sqrt\mu\partial_x\zeta)}{{\epsilon\sqrt\mu}}& \mbox{on } \{(x,z),z=\epsilon\zeta(t,x)\},
\end{array}\right.
\end{equation}
where $ k(\partial_x\zeta)\eqdef-\partial_x \Big(\frac1{\sqrt{1+|\partial_x\zeta|^2}}\partial_x\zeta\Big)$,
$\llbracket P(t,x) \rrbracket \eqdef \lim\limits_{\varkappa\to 0} \Big( P(t,x,\epsilon\zeta(t,x)+\varkappa) - P(t,x,\epsilon\zeta(t,x)-\varkappa) \Big) $,\\
$\nabla_{x,z}^\mu\eqdef (\sqrt\mu\partial_x,\partial_z)^\top$ and $\left(\sqrt{1+\mu\epsilon^2|\partial_x\zeta|^2}\partial_{n}\phi_i \right)\id{z=\epsilon\zeta}= -\mu\epsilon(\partial_x\zeta) (\partial_x\phi_i)\id{z=\epsilon\zeta}+(\partial_z\phi_i)\id{z=\epsilon\zeta}$.

We may conveniently rewrite the above system as two evolution equations, thanks to the use of Dirichlet-Neumann operators. In order to do so, we define $\psi(t,x) \eqdef \phi_1(t,x,\epsilon\zeta(t,x))$,
and
\begin{align*}
G^{\mu}\psi & = G^{\mu}[\epsilon\zeta]\psi \eqdef \sqrt{1+\mu|\epsilon\partial_x\zeta|^2}\big(\partial_n \phi_1 \big)\id{z=\epsilon\zeta} = -\mu\epsilon(\partial_x\zeta) (\partial_x\phi_1)\id{z=\epsilon\zeta}+(\partial_z\phi_1)\id{z=\epsilon\zeta},\\
H^{\mu,\delta}\psi & = H^{\mu,\delta}[\epsilon\zeta]\psi\ \eqdef\ \big(\phi_2\big)\id{z=\epsilon\zeta} \ = \ \phi_2(t,z,\epsilon\zeta(t,x)),
\end{align*}
where $\phi_1$ and $\phi_2$ are uniquely defined (up to an additive constant for $\phi_2$) as the solutions of the Laplace problems implied by~\eqref{fE}.
The full system~\eqref{fE} then becomes
\begin{equation}\label{fE-DN}
\left\{ \begin{array}{l}
\displaystyle\partial_{ t}{\zeta} \ -\ \frac{1}{\mu}G^{\mu}\psi\ =\ 0, \\ \\
\displaystyle\partial_{ t}\Big(\partial_x H^{\mu,\delta}\psi-\gamma \partial_x{\psi} \Big)\ + \ (\gamma+\delta)\partial_x{\zeta} \ + \ \frac{\epsilon}{2} \partial_x\Big(|\partial_x H^{\mu,\delta}\psi|^2 -\gamma |\partial_x {\psi}|^2 \Big) \\
\displaystyle\hspace{5.8cm} = \mu\epsilon\partial_x\N^{\mu,\delta}-\frac{\gamma+\delta}{\Bo}\frac{\partial_x \big(k(\epsilon\sqrt\mu\partial_x\zeta)\big)}{{\epsilon\sqrt\mu}} ,
\end{array}
\right.
\end{equation}
where $\N^{\mu,\delta} = \N^{\mu,\delta}[\epsilon\zeta,\psi] \eqdef \dfrac{\big(\frac{1}{\mu}G^{\mu}\psi+\epsilon(\partial_x{\zeta}) (\partial_x H^{\mu,\delta}\psi) \big)^2\ -\ \gamma\big(\frac{1}{\mu}G^{\mu}\psi+\epsilon(\partial_x{\zeta}) (\partial_x{\psi}) \big)^2}{2(1+\mu|\epsilon\partial_x{\zeta}|^2)}$.

\subsection{Our new class of modified Green-Naghdi models}\label{S.Models}

Let us now present our new class of models, which aim at describing the flow in the situation where the shallowness parameter is small:
\[ \mu\ll 1.\]
Our systems use layer-averaged horizontal velocities as unknowns, that is defining
\[ \overline{u}_1(t,x)=\frac1{h_1(t,x)}\int_{\epsilon\zeta}^1 \partial_x\phi_1(t,x,z)\dd z, \quad \overline{u}_2(t,x)=\frac1{h_2(t,x)}\int_{-\delta^{-1}}^{\epsilon\zeta} \partial_x\phi_2(t,x,z)\dd z.\]
Here and thereafter, $h_1=1-\epsilon\zeta$ (resp. $h_2=\delta^{-1}+\epsilon\zeta$) always denotes the depth of the upper (resp. lower) layer and is assumed to be positive. One benefit of such a choice of unknowns is the {\em exact} identities (in contrast with $\O(\mu^2)$ approximations) due to mass conservation (see \eg~\cite[Proposition 3 and (23)]{DucheneIsrawiTalhouk14}):
\begin{equation}\label{eq-hiui}  \partial_t h_i \ + \  \epsilon \partial_x(h_i\overline{u}_i) \ = \ 0 \qquad (i=1,2).\end{equation}
These identities are then supplemented with the following $\O(\mu^2)$ approximations:
\begin{equation}\label{eq-GNi}
\partial_{ t}\Big( \overline{u}_i \ + \ \mu\Q_i[h_i]\overline{u}_i \Big) \ + \ \frac{\gamma+\delta}{1-\gamma}\partial_x\zeta \ + \ \frac{\epsilon}{2} \partial_x\big(|\overline{u}_i|^2\big) \ = \ \mu\epsilon\partial_x\big(\R_i[h_i, \overline{u}_i] \big) -\frac1{\gamma^{2-i}}\partial_x P_i ,
\end{equation}
where $P_2- P_1=-\frac{\gamma+\delta}{\Bo}\partial_x\Big(\frac1{\sqrt{1+\mu\epsilon^2|\partial_x\zeta|^2}}\partial_x\zeta\Big)$ and
\[
\Q_i^\F[h_i]\overline{u}_i\eqdef -\frac13 h_i^{-1}\partial_x \F_i^\mu\big\{h_i^3 \partial_x \F_i^\mu\ \overline{u}_i\big\}, 
 \quad  \R_i^\F[h_i, \overline{u}_i] \ \eqdef \ \frac12 \big( h_i\partial_x \F_i^\mu \overline{u}_i\big)^2 + \frac13 h_i^{-1}\overline{u}_i \ \partial_x\F_i^\mu\big\{ h_i^3\partial_x\F_i^\mu\overline{u}_i\big\},
\]
 where $\F_i^\mu\eqdef \F_i(\sqrt\mu D)$ ($i=1,2$) is a Fourier multiplier:
 \[ \widehat{\F_i^\mu \varphi} =\F_i(\sqrt\mu \xi ) \widehat{\varphi}(\xi).\]
  The choice of the Fourier multipliers is free, although natural properties for our purpose include $\F_i(0)=1$ and $\F_i'(0)=0$ (so that $\F_i^\mu-{\rm Id}$ is formally of size $\O(\mu)$), $\F_i(k)=\F_i(|k|)$ and $0\leq \F_i \leq 1$. A class of Fourier multipliers for which our rigorous results hold is precised in Definition~\ref{D.admissible}, thereafter, and we present three relevant examples below.
 \begin{itemize}
 \item $\F_i^{\rm id}(\sqrt{\mu}D)\equiv 1$ yields the classical two-layer Green-Naghdi model introduced in~\cite{Miyata85,Miyata87,ChoiCamassa99}, supplemented with the surface tension component.
 \item $\F_i^{\rm reg}(\sqrt{\mu}D)=(1+\mu\theta_i |D|^2)^{-1/2}$, with $\theta_i>0$ is an operator of order $-1$, and $\sqrt\mu\partial_x\F_i^\mu$ is a bounded operator in $L^2$, uniformly with respect to $\mu\geq 0$. As a consequence, this choice yields a well-posed system for sufficiently small and regular data, even in absence of surface tension.
 \item $\F_i^{\rm imp}(\sqrt{\mu}D)=\sqrt{\dfrac{3}{\delta_i^{-1}\sqrt{\mu}|D|\tanh(\delta_i^{-1}\sqrt{\mu}|D|)}-\dfrac3{\delta_i^{-2}\mu |D|^2}}$, with convention $\delta_1=1,\delta_2=\delta$. The modified Green-Naghdi system with this choice conforms perfectly with the full Euler system, as far as the linear stability analysis (see Section~\ref{S.KH}) is concerned. In particular, its dispersion relation is the same as the one of the full Euler system. One may thus hope for an improved precision when only weak nonlinearities ($\epsilon \ll 1$ in addition to $\mu\ll1$) are involved. More precisely, we expect that our model is precise, in the sense of consistency, at the order $\O(\mu^2\epsilon)$ instead of $\O(\mu^2)$ for the original Green-Naghdi system.
 \end{itemize}
 
In~\eqref{eq-GNi}, $P_1$ (or equivalently $P_2$) plays the role of a Lagrange multiplier enforcing the constraint resulting from~\eqref{eq-hiui} and $h_1+h_2=1+\delta^{-1}={\rm Cstt}$, namely
\begin{equation}\label{id-hiui}   \partial_x(h_1\overline{u}_1) \ + \ \partial_x(h_2\overline{u}_2) \ = \ 0.\end{equation}
Notice now that in the one-dimensional setting, supplementing the above equation with the conditions at infinity $\overline{u}_i\to 0$ ($| x|\to\infty$), one obtains the identity
\begin{equation}\label{u1u2}
 h_1\overline{u}_1 \ + \ h_2\overline{u}_2 \ = \ 0.
 \end{equation}
In this situation, one may equivalently rewrite~\eqref{eq-hiui}-\eqref{eq-GNi} as two scalar evolution equations:
\begin{equation}\label{eq-GN}
\left\{ \begin{array}{l}
\displaystyle\partial_{ t}{\zeta} \ + \ \partial_x w \ =\ 0, \\ \\
\displaystyle
\partial_{ t} \left( \frac{h_1+\gamma h_2}{h_1 h_2}w +\mu \Q[\epsilon\zeta]w\right)\ + \ (\gamma+\delta)\partial_x{\zeta} \ + \ \frac{\epsilon}{2} \partial_x\Big(\frac{h_1^2 -\gamma h_2^2 }{(h_1 h_2)^2}| w|^2\Big) \\
\hspace{5.9cm} = \ \mu\epsilon\partial_x\big( \R[\epsilon\zeta,w]\big)+\frac{\gamma+\delta}{\Bo}\partial_x ^2\Big(\frac1{\sqrt{1+\mu\epsilon^2|\partial_x\zeta|^2}}\partial_x\zeta\Big)
\end{array}
\right.
\end{equation}
where, by~\eqref{u1u2} and since $h_1,h_2>0$,
\begin{equation}\label{def-w}
 w\eqdef \frac{h_1 h_2}{h_1+\gamma h_2}(\overline{u}_2-\gamma \overline{u}_1) = -h_1\overline{u}_1=h_2\overline{u}_2
 \end{equation}
and
\begin{equation}\label{def-QR}
\dsp\Q^\F[\epsilon\zeta]w\eqdef \Q_2^\F[h_2](h_2^{-1}w)-\gamma\Q_1^\F[h_1](-h_1^{-1}w), \ 
 \R^\F[\epsilon\zeta, w] \eqdef \R_2[h_2,h_2^{-1}w]-\gamma\R_1^\F[h_1,-h_1^{-1}w].
\end{equation}

Our models are all precise in the sense of consistency at the same order as the original Green-Naghdi model, namely $\O(\mu^2)$. Furthermore, we fully justify system~\eqref{eq-GN} for a class of Fourier multipliers including all of the three examples above, in the sense that for sufficiently small and regular initial data, the model admits a unique strong solution which is proved to be close (when $\mu\ll 1$) to the solution of the full Euler system with corresponding initial data. However, only for the latter two examples is the smallness assumption ensuring the stability of the flow, by preventing the appearance of Kelvin-Helmholtz instabilities, consistent with the one of the full Euler system. The precise statements of these results are displayed in Section~\ref{S.justification}.

\subsection{Outline of the paper}\label{S.mr}

Some elementary properties of our models are studied in Section~\ref{S.Preliminaires}. More precisely, we show that all of our models enjoy a Hamiltonian structure, symmetry groups and conserved quantities, consistently with the already known properties of the original Green-Naghdi model and the full Euler system.

In Section~\ref{S.KH}, we recall the linear analysis of Kelvin-Helmholtz instabilities for the full Euler system, and extend the study to our models. In particular, we recover that the classical Green-Naghdi model overestimates Kelvin-Helmholtz instabilities, whereas our modified model with the choice $\F_i=\F_i^{\rm imp}$ recovers perfectly the behavior of the full Euler system.

Section~\ref{S.numerics} is dedicated to numerical illustrations of this phenomenon. We give two examples (with and without surface tension) where the original, improved and regularized Green-Naghdi models predict very different behavior. Roughly speaking, the flows are very similar as long as no instabilities are present, but the threshold above which Kelvin-Helmholtz instabilities are triggered varies dramatically from one model to the other.

Section~\ref{S.justification} is dedicated to the rigorous justification of our models, for a class of admissible Fourier multipliers $\F_i$. We first prove that the Cauchy problem for system~\eqref{eq-GN} with sufficiently regular initial data is well-posed under some hyperbolicity conditions. Roughly speaking, we show that provided a dimensionless parameter which depends on the large frequency behavior of $\F_i$ is sufficiently small, our system is well-posed for sufficiently regular and bounded initial data, on a time interval uniform with respect to compact sets of parameters; see Theorem~\ref{T.WP} for details. 
We then supplement this result with consistency (Proposition~\ref{P.consistency}) and stability (Proposition~\ref{P.stability}) results, which together offer the full justification of our models (Proposition~\ref{P.justification}).

Finally, we present in Section~\ref{S.SV} some improved results in the limiting case $\mu=0$, that is on the so-called Saint-Venant, or shallow-water system (with surface tension). 

Appendix~\ref{S.functional} is dedicated to the detailed presentation of our functional setting and some notations; and Appendix~\ref{S.proof} provides the proof of our main result, Theorem~\ref{T.WP}.

\section{Hamiltonian structure, group of symmetries and conserved quantities}\label{S.Preliminaires}

It is known since the seminal work of Zakharov~\cite{Zakharov68} that the full Euler system (with one layer) admits a Hamiltonian structure. This Hamiltonian structure has been extended to the two-layer case in~\cite{BenjaminBridges97,LuDaiZhang99,CraigGroves00}. We show in Section~\ref{S.Hamiltonian} that our models --- both under the form~\eqref{eq-hiui}-\eqref{eq-GNi} and~\eqref{eq-GN} --- also admit a Hamiltonian structure, so that our models could be derived from Hamilton's principle on an approximate Lagrangian. Such an approach has been worked out in~\cite{MilesSalmon85,CamassaHolmLevermore96,CamassaHolmLevermore97} (see also~\cite{Li02}) in the one-layer case, in~\cite{CraigGuyenneKalisch05} in the regime of small-amplitude long waves or small steepness, in~\cite{BarrosGavrilyukTeshukov07} in the two-layer case with free surface, and lacks for existing regularized Green-Naghdi systems in the literature~\cite{NguyenDias08,ChoiBarrosJo09,BoonkasameMilewski14,LannesMing}. We then enumerate the group of symmetries of the system (Section~\ref{S.symmetries}) that originates from the full Euler system (see~\cite{BenjaminOlver82}), and deduce the related conserved quantities (Section~\ref{S.conserved}).

\subsection{Hamiltonian formulation}\label{S.Hamiltonian}

Let us define the functional
\begin{multline} \label{H-GN4} \mathcal{H}^\F(h_1,h_2,\overline{u}_1,\overline{u}_2)=\frac12 \int_{\RR}  \frac1{\epsilon^2}\frac{\gamma+\delta}{1-\gamma}(h_2^2-\gamma h_1^2-C) + \frac{2}{\epsilon}(h_1+h_2-H)P_1+ \frac{ 2(\gamma +\delta)}{\mu\epsilon^2\Bo}\big(\sqrt{1+\mu|\partial_x h_2|^2}-1\big)\\
+ \gamma h_1|\overline{u}_1|^2+ h_2|\overline{u}_2|^2  +\mu\frac\gamma3 h_1(h_1\partial_x\F_1^\mu \overline{u}_1)^2+\mu\frac13 h_2(h_2\partial_x \F_2^\mu\overline{u}_2)^2
\end{multline}
with $H=1+\delta^{-1}$ (the total height) and $C=c_0+c_1 h_1+c_2 h_2$ (Casimir invariants) where $c_0,c_1,c_2$ are chosen such that the integral is well-defined. It is convenient for our purpose to define
\[ w_i\eqdef h_i\overline{u}_i \quad \text{ and } \quad v_i \eqdef  \A_i^\F[h_i] w_i \eqdef \frac1{h_i} w_i -\frac{\mu}3 h_i^{-1} \partial_x\F_i^\mu\big\{ h_i^3 \partial_x\F_i^\mu \{h_i^{-1}w_i\}\big\} .\]
Under the non-cavitation assumption $h_i>0$, $\A_i^\F[h_i]$ is a symmetric, coercive thus positive definite operator, and therefore we may equivalently write the Hamiltonian functional as
\begin{multline*} \mathcal{H}^\F(h_1,h_2,v_1,v_2)=\frac12 \int_{\RR}  \frac1{\epsilon^2}\frac{\gamma+\delta}{1-\gamma}(h_2^2-\gamma h_1^2-C) + \frac2\epsilon (h_1+h_2-H)P_1+ \frac{ 2(\gamma +\delta)}{\mu\epsilon^2\Bo}\big(\sqrt{1+\mu|\partial_x h_2|^2}-1\big) \\+ \gamma v_1\A_1^\F[h_1]^{-1} v_1+ v_2\A_2^\F[h_2]^{-1} v_2.
\end{multline*}
We deduce the functional derivatives
\begin{align*}
\frac{\delta \mathcal{H}^\F}{\delta v_1}&=\gamma \A_1^\F[h_1]^{-1} v_1 = \gamma h_1\overline{u}_1, \qquad\qquad  \frac{\delta \mathcal{H}^\F}{\delta v_2}=\A_2^\F[h_2]^{-1} v_2 = h_2\overline{u}_2,\\
\frac{\delta \mathcal{H}^\F}{\delta h_1}&=\frac1{\epsilon^2} \frac{\gamma+\delta}{1-\gamma}(-\gamma h_1+c_1)+\frac1\epsilon P_1+\frac\gamma2|\overline{u}_1|^2-\gamma\mu \R_1^\F[h_1,\overline{u}_1],\\
\frac{\delta \mathcal{H}^\F}{\delta h_2}&=\frac1{\epsilon^2}\frac{\gamma+\delta}{1-\gamma} (h_2+c_2)+\frac1\epsilon P_1-\frac{\gamma+\delta}{\epsilon^2\Bo}\partial_x \Big(\frac{\partial_x h_2}{\sqrt{1+\mu|\partial_x h_2|^2}}\Big)+\frac12|\overline{u}_2|^2-\mu \R_2^\F[h_2,\overline{u}_2].
\end{align*}
One can now observe that the system~\eqref{eq-hiui}-\eqref{eq-GNi} enjoys the non-canonical symplectic form
\[ \partial_t U+\epsilon J \frac{\delta \mathcal{H}^\F}{\delta U} \ = \ 0 \quad \text{ with } \quad U\eqdef \begin{pmatrix}
 h_1\\ v_1\\ h_2 \\ v_2
\end{pmatrix} \quad \text{ and } \quad J\eqdef \begin{pmatrix}
& \gamma^{-1} \partial_x & &\\ \gamma^{-1}  \partial_x & &  & \\ & & & \partial_x   \\ & & \partial_x  & 
\end{pmatrix}.
\]

Let us now indicate how the Hamiltonian formulation of our system can be identified with the well-known similar formulation of the full Euler system for internal waves, which we recall below. We base our discussion on the two-equations system~\eqref{eq-GN} as it allows a direct comparison with earlier works on the bi-fluidic full Euler system, such as~\cite{CraigGuyenneKalisch05}. We introduce the functional
\begin{equation}\label{H-Euler}
\mathcal{H}(\zeta,v)=\frac12\int_\RR (\gamma +\delta)\zeta^2+\frac{ 2(\gamma +\delta)}{\mu\epsilon^2\Bo}\big(\sqrt{1+\mu\epsilon^2|\partial_x\zeta|^2}-1\big)- v \mathcal G_1 (\gamma \mathcal G_2+ \mathcal G_1)^{-1} \mathcal G_2 v ,
 \end{equation}
where $\mathcal G_i=\mathcal G_i[\epsilon\zeta]$ are such that $\partial_x (\mathcal G_i \partial_x \varphi)=G_i \varphi$, and $G_i$ are Dirichlet-to-Neumann operators:
 \begin{itemize}
 \item $G_1: \varphi \mapsto \frac1\mu (\partial_n \phi)\id{z=\epsilon\zeta}$, where $\phi$ is the unique solution to \[ \mu\partial_x^2\phi+\partial_z^2\phi=0 \text{ in } \Omega_1^t, \quad \phi(x,\epsilon\zeta)=\varphi\quad \text{and }\quad \partial_z \phi(x,1)=0.\]
 \item $G_2: \varphi \mapsto \frac1\mu(\partial_n \phi)\id{z=\epsilon\zeta}$, where $\phi$ is the unique solution to
 \[ \mu\partial_x^2\phi+\partial_z^2\phi=0 \text{ in } \Omega_2^t, \quad \phi(x,\epsilon\zeta)=\varphi\quad \text{and }\quad \partial_z \phi(x,-\delta^{-1})=0.\]
 \end{itemize}
The operators $G_i$ are well-defined if $h_1(\epsilon\zeta),h_2(\epsilon\zeta)\geq h_0>0$ (see~\cite[Chapter~3]{Lannes}), and consequently $\mathcal G_i\partial_x\varphi$ as well, thanks to the identity~\cite[Prop.~3.35]{Lannes}
\[ G_i[\epsilon\zeta]\psi \ = \ -\partial_x (h_i\overline{u}_i) \quad(i=1,2). \]
Now, recall the construction of the full Euler in Section~\ref{S.fE} and in particular the definitions of $\phi_1(t,x,z),\phi_2(t,x,z)$ and $\psi(t,x)=\phi_1(t,x,\epsilon\zeta(t,x))$. Denoting
\[
 v(t,x)\eqdef \partial_x\left( \phi_2(t,x,\epsilon\zeta(t,x))-\gamma \phi_1(t,x,\epsilon\zeta(t,x))\right),
\]
we deduce the identities $v= \mathcal G_2^{-1}\mathcal G_1\partial_x\psi -\gamma \partial_x\psi= \mathcal G_2^{-1}(\mathcal G_1 +\gamma \mathcal G_2)\partial_x\psi$.
 
 With these definitions, one can check that the full Euler system~\eqref{fE-DN} can be written as
 \[ \partial_t \zeta=-\partial_x \left( \frac{\delta \mathcal{H}}{\delta v} \right) \qquad ; \qquad
 \partial_t v=-\partial_x \left(\frac{\delta \mathcal{H}}{\delta \zeta} \right). \]
 
We now can view the Hamiltonian functional of our system as a $\O(\mu^2)$ approximation of the Hamiltonian functional of the full Euler system. Indeed, one has (see \eg~\cite[Prop.~3.37]{Lannes} when $\F_i^\mu\equiv 1$, but the general case adds only a $\O(\mu^2)$ perturbation):
\[ \mathcal G_i[\epsilon\zeta]^{-1}= -\A_i^\F[h_i]+\O(\mu^2),\quad \A_i^\F[h_i]:w\mapsto \frac1{h_i}w-\mu\frac13 h_i^{-1}\partial_x \F_i^\mu\big\{h_i^3 \partial_x \F_i^\mu\{h_i^{-1}w\}\big\}.\]
Notice that, using the definitions of Section~\ref{S.Models},
\begin{align*}
\A^\F[\epsilon\zeta]w&\eqdef \gamma\A_1^\F[h_1]w+\A_2^\F[h_2]w=\frac{\gamma}{h_1}w+\frac1{h_2}w-\gamma\mu\frac13 h_1^{-1}\partial_x \F_1^\mu\big\{h_1^3 \partial_x \F_1^\mu\{h_1^{-1}w\}\big\}-\mu\frac13 h_2^{-1}\partial_x \F_2^\mu\big\{h_2^3 \partial_x \F_2^\mu\{h_2^{-1}w\}\big\}\\
&=\frac{h_1+\gamma h_2}{h_1h_2}w+\gamma\mu\Q_1^\F[h_1](h_1^{-1}w)+\mu\Q_2^\F[h_2](h_2^{-1}w)=\frac{h_1+\gamma h_2}{h_1h_2}w+\mu \dsp\Q^\F[\epsilon\zeta]w.
\end{align*}
In particular $\A^\F[\epsilon\zeta]$ is a symmetric, coercive positive definite operator, if $h_1(\epsilon\zeta),h_2(\epsilon\zeta)\geq h_0>0$. Plugging the above (truncated) approximation in the full Euler's Hamiltonian functional~\eqref{H-Euler} yields
\begin{equation}\label{H-GN}
 \mathcal{H}^\F(\zeta,v)\eqdef \frac12\int_{\RR} (\gamma +\delta)\zeta^2+\frac{ 2(\gamma +\delta)}{\mu\epsilon^2\Bo}\big(\sqrt{1+\mu\epsilon^2|\partial_x\zeta|^2}-1\big)+ v\A^\F[\epsilon\zeta]^{-1}v, 
\end{equation}
which corresponds to the previously defined Hamiltonian functional~\eqref{H-GN4} when ${h_1+h_2=H=1+\delta^{-1}}$. Indeed, we may introduce 
\begin{equation}\label{wtov}
 w\eqdef \A^\F[\epsilon\zeta]^{-1}v \quad \text{ and } \quad \overline{u}_i\eqdef (-1)^i w/ h_i \quad (i=1,2),
 \end{equation}
so that
\begin{align*}
 \mathcal{H}^\F(\zeta,v) &= \frac12\int_{\RR} (\gamma +\delta)\zeta^2+\frac{ 2(\gamma +\delta)}{\mu\epsilon^2\Bo}\big(\sqrt{1+\mu\epsilon^2|\partial_x\zeta|^2}-1\big)+ w\A^\F[\epsilon\zeta]w \\
&=\frac12\int_\RR (\gamma +\delta)\zeta^2+\frac{ 2(\gamma +\delta)}{\mu\epsilon^2\Bo}\big(\sqrt{1+\mu\epsilon^2|\partial_x\zeta|^2}-1\big)+ \gamma h_1|\overline{u}_1|^2+ h_2|\overline{u}_2|^2 \\
&\hspace{5cm} +\mu\frac\gamma3 h_1(h_1\partial_x\F_1^\mu \overline{u}_1)^2+\mu\frac13 h_2(h_2\partial_x \F_2^\mu\overline{u}_2)^2.
\end{align*}
Computing the functional derivatives, we get
\begin{align*}
\frac{\delta \mathcal{H}^\F}{\delta v} &=\A^\F[\epsilon\zeta]^{-1}v=w \; \qquad \text{ and }\\
\frac{\delta \mathcal{H}^\F}{\delta \zeta} &=(\gamma +\delta)\zeta-\frac{ \gamma +\delta}{\Bo}\partial_x\left(\frac{(\partial_x\zeta)}{(1+\mu\epsilon^2|\partial_x\zeta|^2)^{1/2}}\right)+ \frac\epsilon2\frac{h_1^2-\gamma h_2^2}{h_1^2h_2^2}|w|^2-\mu\epsilon \R^\F[\epsilon\zeta,w],
\end{align*}
and one recognizes the Hamiltonian structure of system~\eqref{eq-GN}:
\[ \partial_t \zeta =-\partial_x \left(\frac{\delta \mathcal{H}^\F}{\delta v } \right)\qquad ; \qquad \partial_t v =-\partial_x \left(\frac{\delta \mathcal{H}^\F}{\delta \zeta }\right).\]

\subsection{Symmetry groups}\label{S.symmetries}

Based on the work of~\cite{BenjaminOlver82}, one may list symmetry groups of our systems. Most of the symmetry groups of the full Euler system have no equivalent for the Green-Naghdi model, because they involve variations on the vertical variable, which is not accessible anymore. The most physical symmetries, however, remain. We list them below.

Assume $(\zeta,\overline{u}_i,P_i)$ (with $i=1,2$) satisfies~\eqref{eq-hiui}-\eqref{eq-GNi}. Then for any $\varkappa\in\RR$, $(\zeta^\varkappa, \overline{u}_i^\varkappa,P_i^\varkappa)$ also satisfies~\eqref{eq-hiui}-\eqref{eq-GNi}, where
\begin{enumerate}
\item Horizontal translation
\[(\zeta^\varkappa, \overline{u}_i^\varkappa,P_i^\varkappa)(t,x)\eqdef \Big(\zeta(t,x-\varkappa), \overline{u}_i(t,x-\varkappa) , P_i(t,x-\varkappa)\Big)\]
\item Time translation
\[(\zeta^\varkappa, \overline{u}_i^\varkappa,P_i^\varkappa)(t,x)\eqdef \Big(\zeta(t-\varkappa,x), \overline{u}_i(t-\varkappa,x) , P_i(t-\varkappa,x)\Big)\]
\item Variation of base-level for potential pressure
\[(\zeta^\varkappa, \overline{u}_i^\varkappa,P_i^\varkappa)(t,x)\eqdef \Big(\zeta(t,x), \overline{u}_i(t,x) , P_i(t,x)+\varkappa\Big)\]
\item Horizontal Galilean boost
\[(\zeta^\varkappa, \overline{u}_i^\varkappa,P_i^\varkappa)(t,x)\eqdef \Big(\zeta(t,x-\varkappa t), \overline{u}_i(t,x-\varkappa t)+\varkappa,P_i(t,x-\varkappa t)\Big)\] 
\end{enumerate}
It is interesting to notice that when working with formulation~\eqref{eq-GN}. i.,ii.,iii. induce symmetry groups as well (although iii. is trivial), but not iv.. Indeed, because the Galilean boost breaks the conditions $\overline{u}_i\to 0$ at infinity, the identity $w= -h_1\overline{u}_1=h_2\overline{u}_2$ cannot be deduced anymore from the constraint~\eqref{id-hiui} coming from the rigid-lid assumption. Thus the set of admissible solutions to~\eqref{eq-GN} is too restrictive to allow Galilean invariance.

\subsection{Conserved quantities}\label{S.conserved}

The first obviously conserved quantity of our systems, given by~\eqref{eq-hiui}, is the (excess of) mass:
\[\frac{\dd}{\dd t} \mathcal Z=0, \quad \mathcal Z(t)\eqdef \int_\RR \zeta(t,x)\dd x.\]
Equations~\eqref{eq-GNi} yield other conserved quantities: the ``horizontal velocity mass''
\[\frac{\dd}{\dd t} \mathcal V_i=0, \quad \mathcal V_i\eqdef \int_\RR \overline u_i +\mu\Q_i^\F[\epsilon\zeta]\overline u_i \dd x \qquad (i=1,2).\]
Choi and Camassa~\cite{ChoiCamassa99} observed a similar conservation law of the original Green-Naghdi model, and related this result to the irrotationality assumption of the full Euler system. Indeed, one has
\[ \mathcal V_i\approx \int_\RR \partial_x \big(\phi_i(t,x,\epsilon\zeta(t,x))\big) \dd x,\]
where the velocity potentials $\phi_i$ ($i=1,2$) have been defined in Section~\ref{S.fE} and the approximation is meant with precision $\O(\mu^2)$. Thus one has by construction $\frac{\dd}{\dd t} \mathcal V_i=\O(\mu^2)$, and it turns out that this approximately conserved quantity is actually exactly conserved by the Green-Naghdi flow (see also~\cite{GavrilyukKalischKhorsand}).
Of course, the linear combination (recall~\eqref{def-w} and~\eqref{def-QR}) 
\[ \mathcal V\eqdef \mathcal V_2-\gamma\mathcal V_1=\int_\RR \frac{h_1+\gamma h_2}{h_1 h_2}w+\mu\Q^\F[\epsilon\zeta]w\ \dd x\]
is a conserved quantity of system~\eqref{eq-GN}. The above conserved quantities may be interpreted as Casimir invariants of the Hamiltonian system.

After long but straightforward manipulations, one may check 
\[ \frac{\dd}{\dd t} \mathcal M=-\int_{-\infty}^\infty h_1\partial_x P_1+h_2\partial_x P_2 = \left[ h_1 P_1+h_2 P_2\right]^{+\infty}_{-\infty}, \quad \mathcal M\eqdef \int_{-\infty}^\infty \gamma h_1\overline u_1 +h_2 \overline u_2 \dd x.\]
The total horizontal momentum is in general {\em not} conserved. This somewhat unintuitive result is a consequence of the rigid-lid assumption and has been thoroughly studied in~\cite{CamassaChenFalquiEtAl12,CamassaChenFalquiEtAl13}.

One has the conservation of total energy:
\begin{multline*} \frac{\dd}{\dd t} \mathcal H^\F=0, \quad \mathcal H^\F \ = \ \int_{\RR} (\gamma +\delta)\zeta^2+\frac{ 2(\gamma +\delta)}{\mu\epsilon^2\Bo}\big(\sqrt{1+\mu\epsilon^2|\partial_x\zeta|^2}-1\big)\\+ \gamma h_1|\overline u_1|^2+ h_2|\overline u_2|^2 +\mu\frac\gamma3 h_1(h_1\partial_x \F_1^\mu \overline{u}_1)^2+\mu\frac13 h_2(h_2\partial_x \F_2^\mu \overline{u}_2)^2 \dd x.\end{multline*}
The conservation of energy may be deduced from the Hamiltonian structure of the system and more specifically from the invariance of the Hamiltonian with respect to time translation; see \eg~\cite{Shepherd90}. 
Similarly, the invariance with respect to horizontal space translation yields the conservation of the horizontal impulse:
\[ \frac{\dd}{\dd t}\mathcal I = 0, \qquad \mathcal I(\zeta,v)=\int_{\RR}\zeta v\ \dd x.\] 
This conserved quantity of the bi-fluidic Green-Naghdi model seems to have been unnoticed until now. In the one-layer case free surface, that is when $\gamma=0$, the horizontal impulse is related to the momentum (which is conserved in this situation) through the horizontal velocity mass:
\[ \delta^{-1}\mathcal V+\epsilon \mathcal I=\int_\RR h_2 v\ \dd x=\int_\RR h_2\overline{u}_2+\mu h_2\mathcal Q_2^\F[\epsilon\zeta]\overline{u}_2\ \dd x=\int_\RR h_2\overline{u}_2\ \dd x=\mathcal M.\]
In this situation, the symmetry with respect to Galilean boost yields an additional conserved quantity, which is the counterpart of the ``horizontal coordinate of mass centroid times mass'' for the full Euler system as defined in~\cite{BenjaminOlver82}, namely
 \[ \mathcal C(t)\eqdef \int_\RR C(t,x)\dd x, \qquad \text{ with } \quad C(t,x)=\zeta x- t h_2 v \quad \text{ or, equivalently, } \quad \t C(t,x)=\zeta x- t w .\]
The conservation of $\mathcal C$ can be deduced from the above, as~\eqref{eq-hiui} yields
\[ \frac{\dd}{\dd t} \mathcal C =\int_\RR x\partial_t\zeta - w -t\partial_t w\ \dd x=\int_\RR -x\partial_x w - w \ \dd x-t \frac{\dd}{\dd t}\int_\RR w\dd x= -t \frac{\dd}{\dd t} \mathcal M =0.\]

\begin{Remark}[Traveling waves]
The Hamiltonian structure and conserved quantities of our system provide a natural ground for the study of traveling wave solutions. For instance, one easily checks that critical points of the functional $\mathcal{H}^\F(\zeta,v)-c\mathcal{I}(\zeta,v)$ yield traveling wave solutions to system~\eqref{eq-GN}. However, as soon as the surface tension component or non-local operators are present, explicit formula such as the one provided in~\cite{Miyata85,Miyata87} for the original Green-Naghdi model seem hopeless. We thus postpone to a further study the analysis of traveling wave solutions to our system.
\end{Remark}

\section{Kelvin-Helmholtz instabilities}\label{S.KH}
In this section, we formally investigate the conditions for the appearance of Kelvin-Helmholtz instabilities for the full Euler system as well as for our Green-Naghdi models, based on the linear analysis of infinitely small disturbances of steady states.
We say that the system suffers from Kelvin-Helmholtz instabilities when arbitrarily high frequency unstable modes are predicted by this linear theory.

\subsection{The full Euler system} \label{S.KH-fE}
We linearize~\eqref{fE} around the constant shear solution: $( \zeta=0+\varkappa \tilde\zeta,\phi_1=\underline{u}_1 x+\varkappa\tilde\phi_1,\phi_2=\underline{u}_2 x+\varkappa\tilde\phi_2)$ where $\underline{u}_1$ and $\underline{u}_2$ are constants, and $\varkappa\ll1$. Notice that by~\eqref{id-hiui}, one has necessarily $\underline{u}_1+\delta^{-1} \underline{u}_2=0$, and therefore $\underline{u}_1=\frac{- \underline{v}}{\gamma+\delta}$ and $\underline{u}_2=\frac{\delta \underline{v}}{\gamma+\delta}$, where $\underline{v}\eqdef \underline{u}_2-\gamma \underline{u}_1$. When withdrawing $\O(\varkappa^2)$ terms, one obtains the following linear system (see \eg~\cite{LannesMing}):
\begin{equation}\label{shear-fE}
\left\{ \begin{array}{l}
\displaystyle\partial_{ t}\t \zeta\ + \ c(D) \partial_x\t\zeta+b(D)\partial_x\t v \ =\ 0, \\ \\
\displaystyle\partial_{ t}\t v \ + \ a(D)\partial_x\t \zeta+c(D)\partial_x\t v \ =\ 0 
\end{array}
\right.
\end{equation}
where $\t v\eqdef\partial_x\big((\t \phi_2-\gamma\t\phi_1)\id{z=\epsilon\zeta}\big)$, and
\[
c(k)\eqdef \frac{\delta\tanh(\sqrt{\mu}|k|)-\gamma \tanh(\sqrt{\mu}\delta^{-1}|k|)}{\tanh(\sqrt{\mu}|k|)+\gamma \tanh(\sqrt{\mu}\delta^{-1}|k|)}\frac{\underline{\epsilon v}}{\gamma+\delta},\qquad
b(k)\eqdef \frac1{\sqrt{\mu}|k|}\frac{\tanh(\sqrt{\mu}|k|)\tanh(\sqrt{\mu}\delta^{-1}|k|)}{\tanh(\sqrt{\mu}|k|)+\gamma \tanh(\sqrt{\mu}\delta^{-1}|k|)}\]
and
\[
a(k)\eqdef (\gamma+\delta)(1+\frac{|k|^2}{\Bo})-\frac{\sqrt{\mu}|k|\gamma }{\tanh(\sqrt{\mu}|k|)+\gamma \tanh(\sqrt{\mu}\delta^{-1}|k|)} \frac{(\delta+1)^2}{(\delta+\gamma)^2}|\epsilon\underline{v}|^2.\]
Since $b(k)>0$, the mode with wavenumber $k$ is stable (namely the planewave solutions $ e^{i(kx-\omega_\pm(k)t)}$ satisfy $\omega_\pm(k)\in\RR$) if and only if $ a(k)>0$. For small values of $k$, this yields the necessary condition
\[\gamma \epsilon^2|\underline{v}|^2 \frac{\delta(\delta+1)^2}{(\delta+\gamma)^3} <\gamma+\delta.\]
For large values of $k$, one approximates $\tanh(\sqrt{\mu}|k|)+\gamma \tanh(\sqrt{\mu}\delta^{-1}|k|)\approx 1+\gamma$, and deduce
\[\min_{|k|}\big\{a(k)\}\approx(\gamma+\delta)-\frac{\gamma^2 \mu \Bo\ (\delta+1)^4}{4(1+\gamma)^2(\gamma+\delta)^5}\epsilon^4|\underline{v}|^4.
\]
The full Euler system is therefore stable for each wavenumber provided
\[ \Upsilon |\underline{v}|^2\eqdef \gamma\big(1+\sqrt{\mu\Bo}\big)\epsilon^2|\underline{v}|^2 \quad \text{ is sufficiently small.}\]

\subsection{Our class of Green-Naghdi systems}\label{S.KH-GN}

When linearizing~\eqref{eq-GN} around the constant shear solution, $U\eqdef (\zeta,w)^\top \eqdef (0+\varkappa \t\zeta,\underline w+\varkappa \t w)^\top $, where $\underline w$ is constant, one obtains the following system:
\begin{equation}\label{shear}
\left\{ \begin{array}{l}
\dsp\partial_{ t}\t \zeta\ + \ \partial_x\t w \ =\ 0, \\ \\
\dsp \overline{b}^\F(D)\partial_{ t}\t w \ -\ \overline{c}^\F(D)\partial_t \t\zeta \ + \ \overline{a}^\F(D)\partial_x\t \zeta\ +\ \overline{c}^\F(D)\partial_x\t w \ =\ 0, 
\end{array}
\right.
\end{equation}
with
\[
\overline{b}^\F(k)\eqdef\gamma+\delta \ + \ \mu \frac{|\F_2^\mu|^2+\gamma\delta|\F_1^\mu|^2}{3\delta}|k|^2 ,\qquad \overline{c}^\F(k)= \epsilon\underline w (\delta^2-\gamma)\ + \ \mu\epsilon\underline w \frac{ |\F_2^\mu|^2-\gamma|\F_1^\mu|^2}{3}|k|^2,\]
and
\[\overline{a}^\F(k)\eqdef (\gamma+\delta)\ -\ \epsilon^2 |\underline w|^2 (\delta^3+\gamma) \ - \ \mu\epsilon^2|\underline w|^2 \frac{\delta|\F_2^\mu|^2+\gamma|\F_1^\mu|^2}{3} |k|^2 \ +\frac{\gamma+\delta}{\Bo}|k|^2\]
where (with a slight abuse of notations) $\F_i^\mu=\F_i(\sqrt\mu k)$.

The stability criterion is more easily seen when rewriting system~\eqref{shear} with unknown
\footnote{\label{f-unknown}This change of unknown is not without signification. It consists in writing the system with the ``original'' variables of the full Euler system: $\zeta,v=\partial_x\big((\phi_2-\gamma\phi_1)\id{z=\epsilon\zeta}\big)$, or more precisely $\O(\mu^2)$ approximations of these variables. 
It is interesting to note that the nonlinear hyperbolicity condition in Theorem~\ref{T.WP} matches the naive sufficient condition for stability $\overline{a}^\F(k)>0$ coming from~\eqref{shear} but not the sharp condition $a^\F(k)>0$ (in particular when $\gamma\to0$); see also Remark~\ref{R.problems}. 
}
\[ \t v =(\gamma+\delta)\t w-\epsilon \underline w(\delta^2-\gamma)\t\zeta \ - \ \mu\big\{ \frac{|\F_2^\mu|^2+\gamma\delta|\F_1^\mu|^2}{3\delta}\partial_x^2 \t w - \epsilon \underline w\frac{|\F_2^\mu|^2-\gamma|\F_1^\mu|^2}{3} \partial_x^2 \t \zeta \big\} .\]
Indeed, one obtains in that case
\begin{equation}\label{shear-2}
\left\{ \begin{array}{l}
\displaystyle\partial_{ t}\t \zeta\ + \ c^\F(D) \partial_x\t\zeta+b^\F(D)\partial_x\t v \ =\ 0, \\ \\
\displaystyle\partial_{ t}\t v \ + \ a^\F(D)\partial_x\t \zeta+c^\F(D)\partial_x\t v \ =\ 0 
\end{array}
\right.
\end{equation}
with
\[c^\F(k)=\epsilon\underline w\frac{\frac{\delta^2-\gamma}{\gamma+\delta}+\mu\frac{|\F_2^\mu|^2-\gamma|\F_1^\mu|^2}{3(\gamma+\delta)}|k|^2}{1+\mu\frac{|\F_2^\mu|^2+\gamma\delta|\F_1^\mu|^2}{3\delta(\gamma+\delta)}|k|^2}, \qquad b^\F(k)=\frac{\frac1{\gamma+\delta}}{1+\mu\frac{|\F_2^\mu|^2+\gamma\delta|\F_1^\mu|^2}{3\delta(\gamma+\delta)}|k|^2} \]
and
\[ a^\F(k)=(\gamma+\delta)+\frac{\gamma+\delta}{\Bo}|k|^2-|\epsilon\underline w|^2\frac{\gamma(\delta+1)^2}{\delta(\gamma+\delta)}\frac{(\delta^2+\frac13\mu|k|^2|\F_2^\mu|^2)(1+\frac13\mu|k|^2|\F_1^\mu|^2)}{1+\mu\frac{|\F_2^\mu|^2+\gamma\delta|\F_1^\mu|^2}{3\delta(\gamma+\delta)}|k|^2}.\]
Analogously with the full Euler system, the mode with wavenumber $k$ is stable if and only if $ a^\F(k)>0$.
Let us quickly discuss the three examples introduced in Section~\ref{S.Models}.
\begin{itemize}
\item In the case of the original Green-Naghdi system, $\F_i^{\rm id}(\sqrt{\mu}D)\equiv 1$, the condition to ensure that all modes are stable is
\[ \Upsilon_{\rm GN}|\underline w|^2\eqdef \gamma (1+\mu\Bo) \epsilon^2|\underline w|^2 \quad \text{ is sufficiently small.}\]
This is more stringent than the corresponding condition on the full Euler system in the oceanographic context, where one expects $\mu\Bo \gg 1$, as noticed in~\cite{JoChoi02,LannesMing}.

\item If $\F_i^{\rm imp}(\sqrt{\mu}D)=\sqrt{\dfrac{3}{\delta_i^{-1}\sqrt{\mu}|D|\tanh(\delta_i^{-1}\sqrt{\mu}|D|)}-\dfrac3{\delta_i^{-2}\mu |D|^2}}$ (with convention $\delta_1=1,\delta_2=\delta$), then the linearized system~\eqref{shear-2} is {\em exactly}~\eqref{shear-fE} (recall that by~\eqref{id-hiui}, $\underline{w}=\frac{1}{\gamma+\delta}\underline{v}$):
\[ a^\F(k)=a(k) \quad ; \quad b^\F(k)=b(k) \quad ; \quad c^\F(k)=c(k) .\]
In particular, the stability criterion of this Green-Naghdi model corresponds to the one of the full Euler system. As previously mentioned, this also shows that the model has the same dispersion relation as the full Euler system, as this corresponds to setting $\underline w=0$. 

\item In the case $\F_i^{\rm reg}(\sqrt{\mu}D)=\dfrac{1}{\sqrt{1+\mu\theta_i |D|^2}}$, one remarks that
\[ (\gamma+\delta)>|\epsilon\underline w|^2\frac{\gamma(\delta+1)^2}{\delta(\gamma+\delta)}\big(\delta^2+\frac1{3\theta_2}\big)\big(1+\frac1{3\theta_1}\big),\]
is a sufficient condition to ensure that all modes are stable, and {\em does not require the presence of surface tension}. A natural choice is $\theta_i=\frac{1}{15\delta_i^2}$ with convention $\delta_1=1,\delta_2=\delta$, for the Taylor expansion of the dispersion relation around $\mu=0$ to fit the one of the full Euler system, at augmented order $\O(\mu^3)$, instead of the $\O(\mu^2)$ precision of the original Green-Naghdi system. 
\end{itemize}

In Figure~\ref{F.instability}, we plot the instability curves corresponding to $a^\F(k)$ for the three above examples. More precisely, for fixed $k\in\RR$, we plot the value of $\epsilon^2|\underline{w}|^2$ above which $a^\F(k)>0$, and thus instabilities are triggered. One clearly sees a great discrepancy for large wavenumbers. In particular the minimum of the curve, which corresponds to the domain where all wavenumbers are stable, not only varies for each model, but also is obtained at different values of $k$.
\begin{figure}[htp]
\includegraphics[width=1\textwidth]{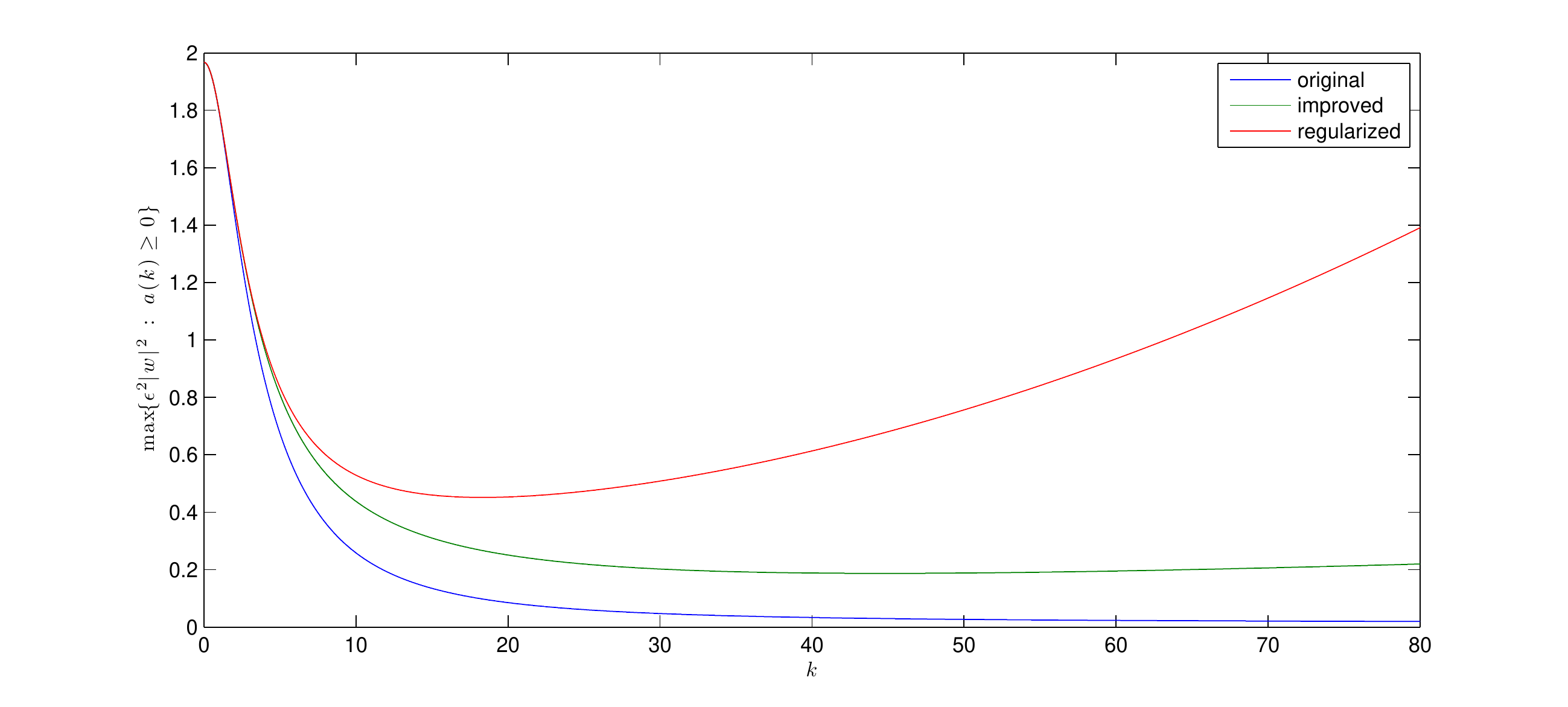}
\caption{Instability curves of the modified Green-Naghdi models with $\F_i=1$ (original), ${\F_i=\F_i^{\rm reg}}$ (regularized) and $\F_i=\F_i^{\rm imp}$ (improved). The last one coincides with the full Euler system counterpart. The dimensionless parameters are $\gamma=0.95$, $\delta=0.5$, $\epsilon=0.5$, $\mu=0.1$, $\Bo^{-1}=5\times 10^{-4}$.}
\label{F.instability}
\end{figure}

\section{Numerical illustrations}\label{S.numerics}

We numerically compute several of our Green-Naghdi systems, with and without surface tension, in order to observe how the different frequency dispersion may affect the appearance of Kelvin-Helmholtz instabilities.

As in Figure~\ref{F.instability}, we focus on the three aforementioned examples: $\F_i=1$ (original), $\F_i=\F_i^{\rm reg}$ with $\theta_i=\frac{1}{15\delta_i^2}$ (regularized) and $\F_i=\F_i^{\rm imp}$ (improved). Values for the dimensionless parameters are as in Figure~\ref{F.instability} above. The initial data is $\zeta(0,x)=-e^{-4| x|^2}$ and $w(0,x)= 0$.

Figures~\ref{F.with-surf-t2} and~\ref{F.with-surf-t3} represent the predicted flow at time $t=2$ and $t=3$, in the situation with surface tension. Figure~\ref{F.without-surf-t2} represents the predicted flow at time $t=2$ in the situation without surface tension. Each time, the left panel plots the flux, $w(t,x)$ (or rather $1+w$ for the sake of readability) as well as the interface deformation, $\zeta(t,x)$; while the right panel plots the spatial Fourier transform of the interface deformation, $\widehat{\zeta}(t,k)$. The dashed line represents the initial data, and the three colored lines the predictions of each model.

\begin{figure}[htbp]
\hfill\includegraphics[width=.95\textwidth]{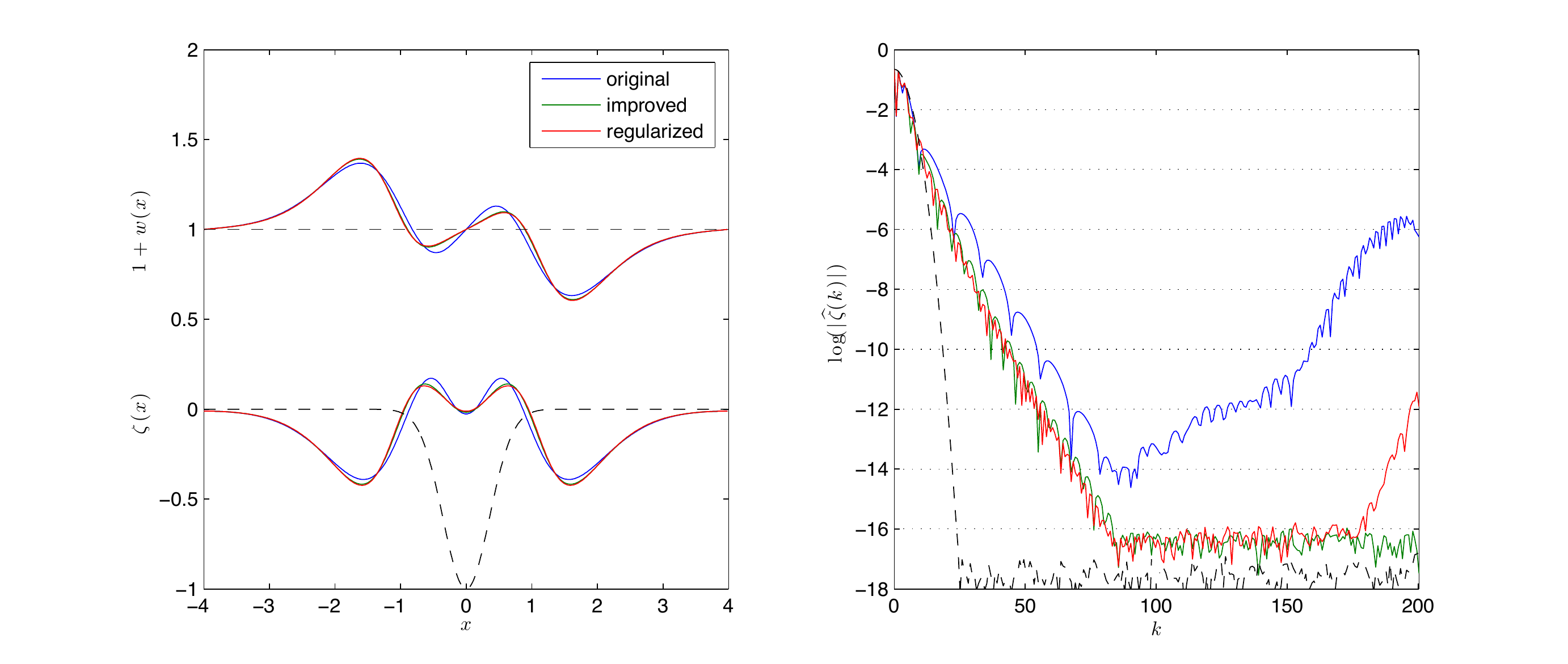}\hfill
\caption{Prediction of the Green-Naghdi models, with surface tension, at time $t=2$.}
\label{F.with-surf-t2}
\end{figure}
\begin{figure}[htbp]
\hfill\includegraphics[width=.95\textwidth]{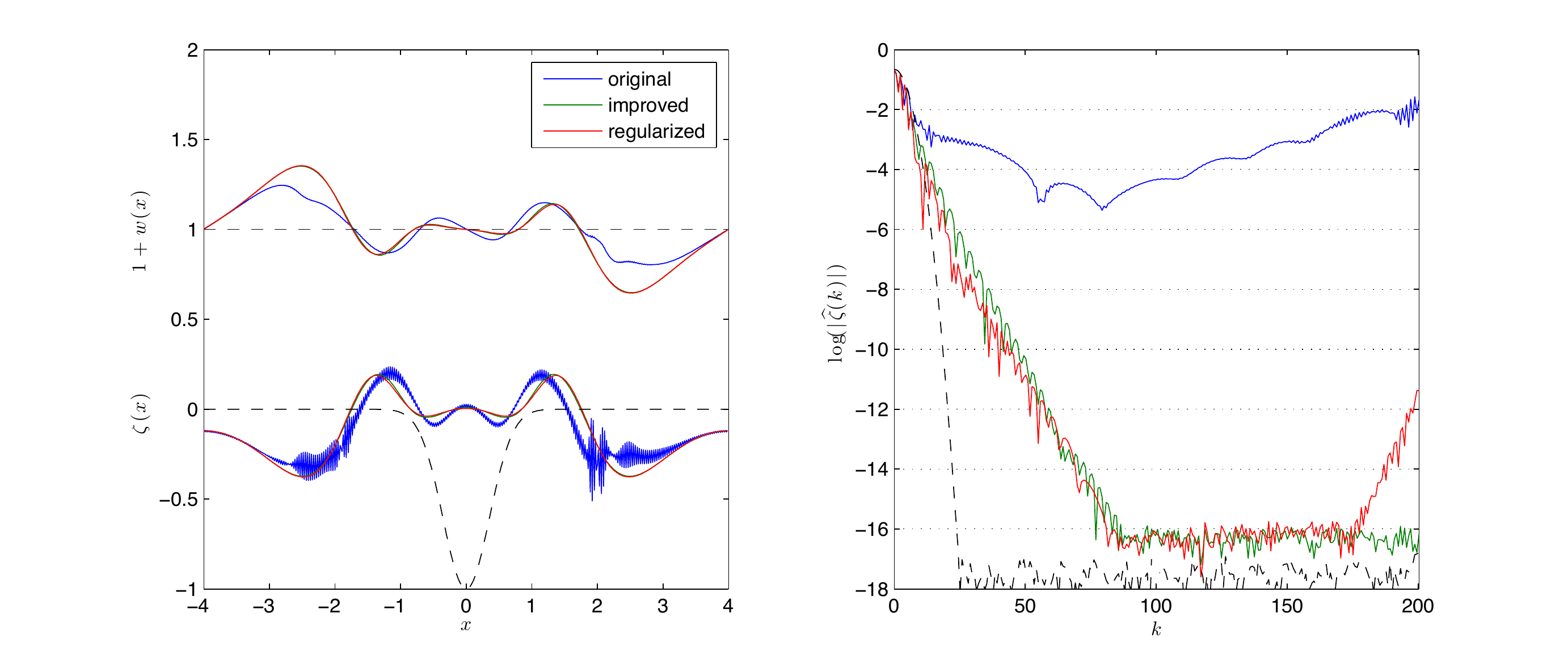}\hfill
\caption{Prediction of the Green-Naghdi models, with surface tension, at time $t=3$.}
\label{F.with-surf-t3}
\end{figure}
\begin{figure}[htbp]
\hfill\includegraphics[width=.95\textwidth]{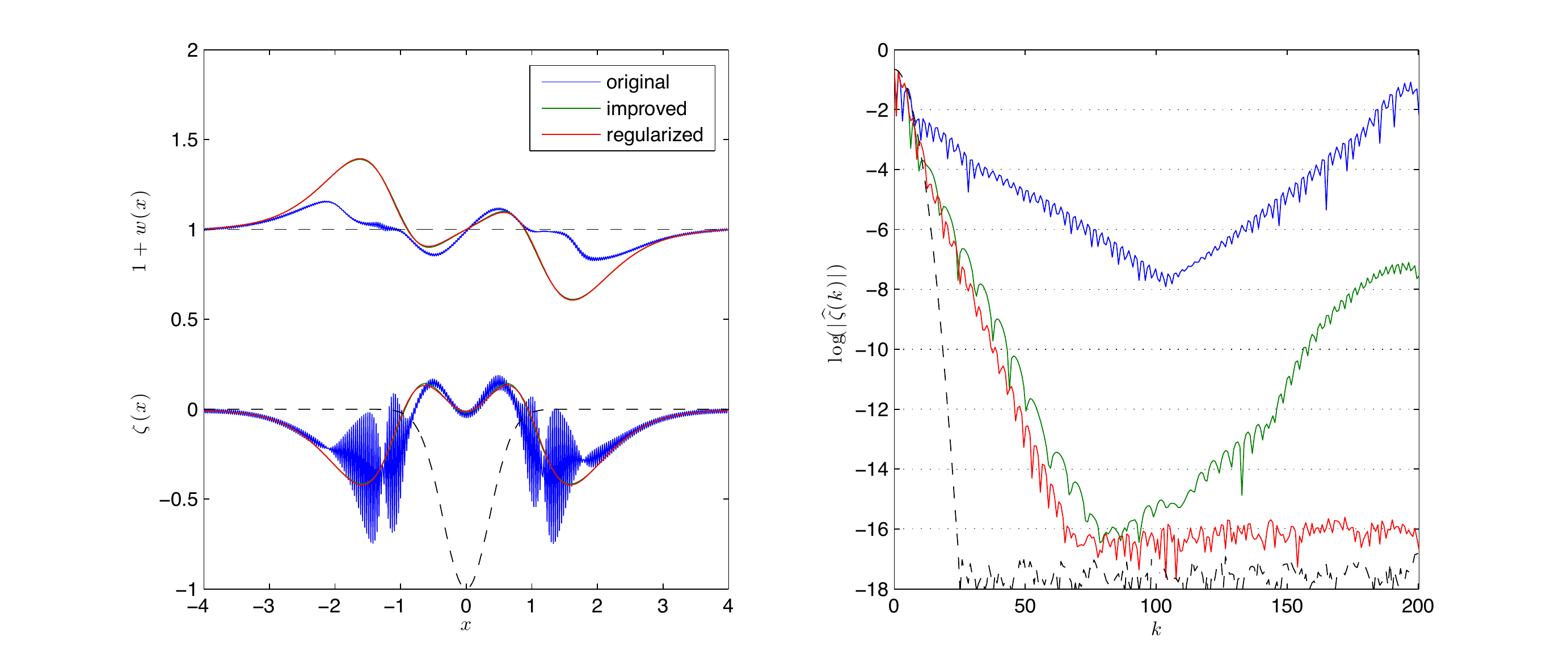}\hfill
\caption{Prediction of the Green-Naghdi models, without surface tension, at time $t=2$.}
\label{F.without-surf-t2}
\end{figure}

\paragraph{Discussion} In the situation with surface tension, we see that at time $t=2$ (Figure~\ref{F.with-surf-t2}), the predictions of the three models are similar. Only the original model shows small but clear discrepancy, and most importantly early signs of instabilities. This situation is clearer when looking at the Fourier transform, right panel. We see the existence of a strong large frequency component which has grown from machine precision noise. As expected, modes with higher wavenumbers grow faster. 
Notice the regularized model also exhibits a non-trivial (although very small) high-frequency component. This component is however stable with respect to time: it is not produced by Kelvin-Helmholtz instabilities, but rather by numerical errors. It does not appear when surface tension is absent (Figure~\ref{F.without-surf-t2}).

At later time $t=3$ (Figure~\ref{F.with-surf-t3}), the Kelvin-Helmholtz instabilities have completely destroyed the flow of the original model. For such irregular data, spectral methods are completely inappropriate and we do not claim that our numerical scheme computes any relevant approximate solution. Meanwhile, the flows predicted by the regularized and improved models remain smooth and are very similar. When running the numerical simulation for much larger time, our computations indicate that the flow of the regularized and improved models remains smooth for all time.

When surface tension is neglected from the models, we see (Figure~\ref{F.without-surf-t2}) that at time $t=2$, Kelvin-Helmholtz instabilities have already destroyed the flow for the original model, and the improved model shows early signs of instabilities in its high-frequency component. The flow predicted by the regularized model, however, remains smooth and is very similar to the flow with surface tension.

\paragraph{Numerical scheme} Let us now briefly present our numerical scheme. It is very natural in our context to use spectral methods~\cite{Trefethen} as for the space discretization, since Fourier multipliers are dealt similarly as regular differential operators. Such methods yield an exponential accuracy with respect to the spatial mesh size for smooth data. In our simulations, we used $2^9=512$ equally distributed points (with periodic boundary conditions) on $x\in [-4;4]$. As for the time evolution, we use the Matlab solver \texttt{ode45}, which is based on the fourth and fifth order Runge-Kutta-Merson method~\cite{ShampineReichelt97}, with a relative tolerance of $10^{-10}$ and absolute tolerance of $10^{-12}$. It is convenient to solve the system written in terms of $\zeta$ and $v=\frac{h_1+\gamma h_2}{h_1 h_2}w +\mu \Q^\F[\epsilon\zeta]w$; see~\eqref{wtov}, although this requires to solve at each time step $w$ as a function of $\zeta$ and $v$.

In Table~\ref{T.precision} we display the numerical variations, between time $t=1$ and initial time $t=0$, of the conserved quantities (discussed in Section~\ref{S.conserved}) as a very rough mean to appreciate the precision of the numerical scheme. One sees that the agreement is excellent, except when the horizontal impulse is concerned. The latter shows a great sensibility to the presence of large frequency components, indicating that such component, as expected, affects the precision of the numerical scheme. It is remarkable that the other conserved quantities do not suffer from such a loss of precision.

\begin{table}[htp] \begin{small}
\begin{tabular}{c|ccc||ccc}
 & \multicolumn{3}{c||}{With surface tension} & \multicolumn{3}{c}{Without surface tension} \\
 & original & regularized & improved & original & regularized & improved \\ 
\hline 
Mass $\mathcal Z$ & \footnotesize 1.5321\ $10^{-14}$ & \footnotesize 1.5654\ $10^{-14}$ & \footnotesize 1.5876\ $10^{-14}$ &\footnotesize 1.299\ $10^{-14}$ & \footnotesize 1.5321\ $10^{-14}$ & \footnotesize 1.4988\ $10^{-14}$ \\
Velocity $\mathcal V$ & \footnotesize -6.3135\ $10^{-16}$ & \footnotesize -7.3647\ $10^{-17}$ & \footnotesize 9.1148\ $10^{-16}$ & 
\footnotesize -1.4962e-15\ $10^{-15}$ & \footnotesize 3.7576\ $10^{-16}$ &\footnotesize 5.3084\ $10^{-16}$ \\
Impulse $\mathcal I$ & \footnotesize 5.4362\ $10^{-10}$ & \footnotesize 6.6559\ $10^{-17}$ & \footnotesize -2.1199\ $10^{-17}$ & \footnotesize -1.3741\ $10^{-2}$ & \footnotesize -4.4256\ $10^{-17}$ & \footnotesize 2.2034\ $10^{-14}$\\
Energy $\mathcal E$ & \footnotesize -3.5282\ $10^{-12}$ & \footnotesize -7.3064\ $10^{-13}$ & \footnotesize -4.8171\ $10^{-12}$ & \footnotesize -3.3628\ $10^{-12}$ & \footnotesize -5.1278\ $10^{-12}$ & \footnotesize -4.7436\ $10^{-12}$ \\
\hline 
\end{tabular} \end{small}
\caption{Difference between conserved quantities at time $t=2$ and time $t=0$.}
\label{T.precision}
\end{table}

\section{Full justification}\label{S.justification}
This section is dedicated to the proof of the main results of this work, namely the rigorous justification of the class of modified Green-Naghdi systems introduced in Section~\ref{S.Models}, \ie
\begin{equation}\label{GN-w}
\left\{ \begin{array}{l}
\displaystyle\partial_{ t}{\zeta} \ + \ \partial_x w \ =\ 0, \\ \\
\partial_{ t} \left( \frac{h_1+\gamma h_2}{h_1 h_2}w +\mu \Q^\F[\epsilon\zeta]w\right)\ + \ (\gamma+\delta)\partial_x{\zeta} \ + \ \frac{\epsilon}{2} \partial_x\Big(\frac{h_1^2 -\gamma h_2^2 }{(h_1 h_2)^2}| w|^2\Big) \\
\hspace{5.9cm} = \ \mu\epsilon\partial_x\big( \R^\F[\epsilon\zeta]w\big)+\frac{\gamma+\delta}{\Bo}\partial_x ^2\Big(\frac1{\sqrt{1+\mu\epsilon^2|\partial_x\zeta|^2}}\partial_x\zeta\Big) ,
\end{array}
\right.
\end{equation}
where
\begin{align*}
 \Q^\F[\epsilon\zeta]w&=-\frac13 h_2^{-1}\partial_x \F_2^\mu\big\{h_2^3 \partial_x \F_2^\mu\{h_2^{-1}w\}\big\}-\frac\gamma3 h_1^{-1}\partial_x \F_1^\mu\big\{h_1^3 \partial_x \F_1^\mu\{h_1^{-1}w\}\big\},\\
 \R^\F[\epsilon\zeta,w]&=\frac13 w h_2^{-2}\partial_x \F_2^\mu\big\{h_2^3\partial_x \F_2^\mu \{h_2^{-1}w\}\big\}-\frac\gamma3 w h_1^{-2}\partial_x \F_1^\mu\big\{h_1^3\partial_x \F_1^\mu \{h_1^{-1}w\}\big\} \\
& \quad +\frac12 \big(h_2\partial_x \F_2^\mu\{h_2^{-1}w\}\big)^2-\frac\gamma2 \big(h_1\partial_x \F_1^\mu\{h_1^{-1}w\}\big)^2.
\end{align*}
Here and thereafter, we always denote $h_1=h_1(\epsilon\zeta)=1-\epsilon\zeta$ and $h_2=h_2(\epsilon\zeta)=\delta^{-1}+\epsilon\zeta$.
Let us also recall that $\F_i^\mu$ ($i=1,2$) denotes a Fourier multiplier:
\[ \F_i^\mu=\F_i(\sqrt\mu D) \quad \text{\ie} \quad \widehat{\F_i^\mu f}(\xi)=\F_i(\sqrt\mu \xi)\widehat f(\xi).\]
In order to allow for the functional analysis detailed in Section~\ref{S.product}, we restrict ourselves to admissible Fourier multipliers $\F_i^\mu$, in the following sense.
\begin{Definition}\label{D.admissible}
The operator $\F_i^\mu= \F_i(\sqrt\mu D)$ ($i=1,2$) is {\em admissible} if it satisfies:
\begin{enumerate}
\item $\F_i:\RR\to\RR^+$ is even and positive;
\item $\F_i$ is of twice differentiable, $\F_i(0)=1$, $\F_i'(0)=0$ and $\sup_{k\in\RR}|\F_i''(k)|\leq C_\F<\infty$;
\item $k\mapsto |k|\F_i(k)$ is sub-additive, namely for any $k,l\in \RR, |k+l|\F_i(k+l)\leq |k|\F_i(k)+|l|\F_i(l)$.
\end{enumerate}
In that case, one can define appropriate pairs $K_{\F_i}\in\RR^+$ and $\varsigma\in[0,1]$ such that
\begin{equation}\label{def-Ki-sigma}
\forall k\in\RR,\quad \F_i(k)\leq K_{\F_i} |k|^{-\varsigma}.
\end{equation}
\end{Definition}
\begin{Proposition} \label{P.examples-are-subadditive}
The three examples provided in Section~\ref{S.Models}, namely $\F_i^{\rm id},\F_i^{\rm reg},\F_i^{\rm imp}$ are admissible, and satisfy~\eqref{def-Ki-sigma} with (respectively) $\varsigma=0,1,1/2$.
\end{Proposition}
\begin{proof}
A sufficient condition for $\F_i^\mu$ to be admissible is to satisfy, in addition to i. and ii.,
\begin{itemize}
\item[iii'.] $k\mapsto k\F_i(k)$ is non-decreasing on $\RR^+$ and $k\mapsto \F_i(k)$ is non-increasing on $\RR^+$.
\end{itemize} 
Indeed, for $k+l\geq k\geq l\geq 0$, that $\F_i(k)$ is non-increasing on $\RR^+$yields $0\leq \F_i(k+l)\leq \F_i(k)\leq \F_i(l)$, and therefore $(k+l)\F_i(k+l)\leq k \F_i(k)+ l\F_i(k)\leq k \F_i(k)+ l\F_i(l)$. This shows $k\in\RR^+\mapsto k \F_i(k)=|k| \F_i(k)$ is sub-additive. Since $\F_i$ is even, one shows in the same way that $k\in\RR^-\mapsto |k| \F_i(k)$ is sub-additive. 
Finally, for $k\leq 0\leq l$, that $\F_i$ is even and $k\in\RR^+\mapsto k\F_i$ is non-decreasing yields $|k+l|\F_i(k+l)=|k+l|\F_i(|k+l|)\leq (|k|+|l|)\F_i(|k|+|l|) $, and the sub-additivity in $\RR^+$ allows to conclude.

That property iii' holds is immediate for the first two examples, only the last one requires clarifications.
One easily checks that $k\in\RR^+\mapsto k\F_i^{\rm imp}(k)$ is non-decreasing, so that we focus on the proof that $k\mapsto \F_i^{\rm imp}(k)$ is non-increasing for $k\in\RR^+$. To this aim, it suffices to show that  
\[\forall k>0, \quad f(k)\eqdef -k^2-\frac12 k\sinh(2k)+\cosh(2k)-1<0.\] This follows from $f(0)=f'(0)=f''(0)=0$ and
\[f'''(k)=2\sinh(2k)-4k\cosh(2k)=2\cosh(2k)(\tanh(2k)-2k)< 0 \quad (k>0).\]
Thus property iii' holds, and $k\mapsto |k|\F_i^{\rm imp}(k)$ is sub-additive. 
\end{proof}

For reasons explained in Appendix~\ref{S.proof}, our energy space involves both space and time derivatives of the unknowns. For $U=(\zeta,w)^\top$, we define
\[
E^0(U)\eqdef \norm{\zeta}_{\X^0}^2+\norm{w}_{\Y^0}^2, \quad E^N(U)\eqdef \sum_{|\alpha|=0}^NE^0(\partial^\alpha U) \eqdef \norm{\zeta}_{\X^N}^2+\norm{w}_{\Y^N}^2,
\]
where $N$ always denotes an integer and $\alpha$ a multi-index. The functional setting and in particular the definitions of functional spaces $\X^N$ and $\Y^N$ are given in Appendix~\ref{S.functional}.

In addition to $\gamma,\mu,\epsilon,\delta,\Bo^{-1}\geq 0$, it is convenient to introduce the dimensionless parameters
\begin{equation} \label{def-Upsilon}
\Upsilon_\F \eqdef \epsilon^2\big (1+(\gamma K_{\F_1}+K_{\F_2})(\mu\Bo)^{1-\varsigma}\big) <\infty
\end{equation}
where $\varsigma$, $K_{\F_1}$, $K_{\F_2}$ are specified in Definition~\ref{D.admissible},~\eqref{def-Ki-sigma}; and
\[
\m\eqdef \max\{\epsilon,\gamma,\delta,\delta^{-1},\mu,\Bo^{-1}\} <\infty.
\]

\begin{Theorem}[Well-posedness]\label{T.WP}
Let $U^0\eqdef(\zeta^0,w^0)^\top\in \X^N\times \Y^N$ with $N\geq 4$, satisfying
\begin{align*}
h_1^0\eqdef 1-\epsilon\zeta^0\geq h_0>0,\qquad h_2^0\eqdef\delta^{-1}+\epsilon\zeta^0\geq h_0>0 ; \\
 (\gamma+\delta)-\epsilon^2\max_{x\in\RR}\left\{((h_2^0)^{-3} +\gamma (h_1^0)^{-3}) |w^0|^2\right\}\geq k_0>0.
\end{align*}
One can define $K=C(\m ,h_0^{-1},k_0^{-1},\epsilon \norm{\zeta}_{H_x^{{3}}})$ such that if $\Upsilon_\F \norm{w^0}_{\Z^1}^2 \leq K^{-1}$, there exists $T>0$ and a unique $U\eqdef (\zeta,w)^\top\in C^0_{\rm w}([0,T);\X^N\times\Y^N)$ solution to~\eqref{GN-w} and $U\id{t=0}=U^0$. Moreover, there exists ${\bf C}_0=C(\m,h_0^{-1},k_0^{-1},K,E^N(U^0))$ such that
\[ T^{-1}\leq {\bf C}_0 \times \Big(\epsilon+\Upsilon_\F^{1/2} \norm{w^0}_{\Z^1}+\Upsilon_\F \norm{w^0}_{\Z^2}^2 \Big)
\quad \text{ and }\quad \sup_{t\in [0,T)} E^N(U)\ \leq {\bf C}_0 \times E^N(U^0).\]
\end{Theorem}
The proof of this result is postponed to Appendix~\ref{S.proof}. Let us for now discuss a few implications.
\begin{Remark}[Initial data]
Since our functional spaces involve time derivatives, it is not {\em a priori} clear how to define $\norm{\zeta^0}_{\X^N}$ and $\norm{w^0}_{\Y^N}$. As it is manifest from the proof, the definition of $(\partial^\alpha U^0)\id{t=0}$ for sufficiently regular $\zeta^0(x),w^0(x)$ is given by system~\eqref{GN-w} itself. More precisely, for $\alpha=(0,\alpha_2)$, then the definition is clear. We then define $(\partial^\alpha U^0)\id{t=0}$ for $\alpha=(\alpha_1,\alpha_2)$ with $\alpha_1>0$ by finite induction on $\alpha_1$, through the identities obtained from differentiating~\eqref{GN-w}  $|\alpha_1|-1$ times with respect to time. These identities given in Lemma~\ref{L.linearize}, and are uniquely solved by Lemma~\ref{L.invertible}.
\end{Remark}
\begin{Remark}[Domain of hyperbolicity and time of existence] \label{R.problems}
Hypotheses on the initial data ensure that the flow lies in the ``domain of hyperbolicity'' of the system; see Lemma~\ref{L.hyperbolicity}. They may be seen as the nonlinear version of the stability criterion presented in Section~\ref{S.KH-GN}, as they provide sufficient conditions for Kelvin-Helmholtz instabilities not to appear. However, remark that our ``Kelvin-Helmholtz instability parameter'', $\Upsilon_\F$, is not multiplied by $\gamma$, in contrast with $\Upsilon$ and $\Upsilon_{\rm GN}$ in Section~\ref{S.KH}, as well as the nonlinear criterion on the full Euler system given by Lannes~\cite[(5.1)]{Lannes13}. The latter results imply that the large-frequency Kelvin-Helmholtz instabilities disappear in the limit $\gamma\to0$, so that surface tension is not necessary for the well-posedness of the system when $\gamma=0$. We do not recover such property with our rigorous analysis, although numerical simulations indicate that our models are well-posed when $\gamma=0$ and $\Bo=\infty$, as long as the non-vanishing depth condition is satisfied.

A second setback is that the time of existence involves $\Upsilon_\F^{1/2} \norm{w^0}_{\Z^1}$, and not only $\Upsilon_\F \norm{w^0}_{\Z^2}^2 $. In practice, this means that when $\Upsilon_\F \ll 1$, and in particular when $\Upsilon_\F \leq \epsilon\ll 1$, then the time of existence of our result is significantly smaller than the one in~\cite[Theorem 6]{Lannes}.

However, let us note that our conclusions, in particular with the choice $\F_i=\F_i^{\rm imp}$ where $\varsigma=1/2$, are in complete agreement with aforementioned results in the oceanographic setting of internal waves, where one expects large values of $\epsilon$ and $\gamma\approx 1$.

We believe that the above limitations originate from the choice of unknowns used when quasilinearizing the equation. This was quickly discussed in footnote~\ref{f-unknown} in Section~\ref{S.KH-GN} as for the occurrence of $\gamma$. The restriction on the time of existence originates from estimates~\eqref{r-est} and~\eqref{rdiff-est} in Lemma~\ref{L.linearize}, and more precisely the lack of an analogue of~\cite[Lemma~7]{Lannes13} thanks to which ``good unknowns'' were constructed. We show in Section~\ref{S.SV} how the techniques used in this work, applied to the Saint-Venant system (that is setting $\mu=0$) written with different unknowns, yields sharp results.
\end{Remark}
\begin{Remark}[Regularized systems]
In the case $\varsigma=1$, one sees that Theorem~\ref{T.WP} does not depend on $\Bo$ (through $\Upsilon_\F$). In particular, the results hold true even when surface tension is neglected, \ie $\Bo^{-1}=0$, and we recover in that case the ``quasilinear timescale'' $T^{-1}\lesssim\epsilon$. Our strategy relying on the use of space-time energy is not needed in that case, as classical energy methods can be applied to prove the well-posedness for initial data in Sobolev spaces: $(\zeta^0,w^0)^\top\in H^s\times H^s$, $s>3/2$.
\end{Remark}

We show now that the above well-posedness analysis can be supplemented with consistency and stability results, which together provide the full justification of our models,~\eqref{GN-w}.
Such a program was completed for similar models in the one-layer setting in~\cite{BonaColinLannes05,Alvarez-SamaniegoLannes08,Alvarez-SamaniegoLannes08a} (see \cite{Lannes} for a detailed account). In the two-layer case with rigid upper lid, the consistency of many models were derived in \cite{BonaLannesSaut08}. The consistency result below builds upon \cite{Duchene14} and \cite{DucheneIsrawiTalhouk14}.

\begin{Proposition}[Consistency]\label{P.consistency}
Let $U\eqdef(\zeta,\psi)^\top$ be a solution of the full Euler system~\eqref{fE-DN} such that
such that there exists $C_0,T>0$ with 
\[ \big\Vert \zeta \big\Vert_{L^\infty([0,T);H_x^{s+5})} + \big\Vert \partial_t\zeta \big\Vert_{L^\infty([0,T);H_x^{s+4})} + \big\Vert \partial_x\psi \big\Vert_{L^\infty([0,T);H_x^{s+\frac{11}2})} + \big\Vert \partial_t\partial_x\psi \big\Vert_{L^\infty([0,T);H_x^{s+\frac92})} \leq C_0 ,\]
for given $s\geq t_0+1/2$, $t_0>1/2$. Moreover, assume that there exists $h_{0}>0$ such that
\begin{equation}\label{C.depth-mr}
\forall (t,x)\in [0,T)\times\RR, \qquad h_1(t,x)=1-\epsilon\zeta(t,x)\geq h_0>0,\quad h_2(t,x)=\delta^{-1}+\epsilon\zeta(t,x)\geq h_0>0.
\end{equation}
 Define $w$ by $\partial_x w=-\frac1\mu G^{\mu}[\epsilon\zeta]\psi =-\partial_t\zeta$. Then $(\zeta,w)^\top$ satisfies exactly the first equation of~\eqref{GN-w}, and the second up to a remainder, $r$, bounded as
\[ \big\Vert r \big\Vert_{L^\infty([0,T);H_x^s)} \ \leq \ \mu^2 \times C(\m,h_{0}^{-1},C_0,C_\F). \]
\end{Proposition}
\begin{proof}
The Proposition has been stated and proved, in the case of the original Green-Naghdi system, $\F_i^{\rm id}\equiv 1$, in~\cite[Proposition~2.4]{Duchene14}. By triangular inequality, there only remains to estimate
\[ r_\Q\eqdef \mu \norm{\partial_t \big(\Q^{\F^{\rm id}}[\epsilon\zeta] w-\Q^\F[\epsilon\zeta]w\big)}_{H_x^s}\quad \text{ and } \quad 
 r_\R\eqdef \mu \epsilon\norm{\partial_x \big(\R^{\F^{\rm id}}[\epsilon\zeta,w]-\R^\F[\epsilon\zeta,w]\big)}_{H_x^s}.\]
Notice that by Definition~\ref{D.admissible}, one has
\[\big\vert \F_i(\sqrt{\mu}k)-1 \big\vert \ \leq \ \frac12 C_\F\mu |k|^2.\]
It follows, since $H_x^s$ is an algebra for $s>1/2$ and by Lemma~\ref{L.f/h}, that
\begin{align*}r_\Q&\leq \sum_{i=1}^2 \frac\mu3\norm{ \partial_t\big( h_i^{-1}\partial_x \big(\F_i^\mu-{\rm Id}\big)\big\{h_i^3 \partial_x \F_i^\mu\{h_i^{-1}w\}\big\}\big)}_{H_x^s}+\frac\mu3\norm{ \partial_t\big( h_i^{-1}\partial_x \big\{h_i^3 \partial_x \big(\F_i^\mu-{\rm Id}\big) \{h_i^{-1}w\}\big\}\big)}_{H_x^s},\\
&\leq \mu^2C_\F\ C(\m,h_0^{-1},\norm{\partial_t \zeta}_{L^\infty([0,T);H_x^{s+4})},\norm{\partial_t w}_{L^\infty([0,T);H_x^{s+4})}).\end{align*}
Similarly, one has
\[ r_\R\leq \mu^2 C_\F\ C(\m,h_0^{-1},\norm{ \zeta}_{L^\infty([0,T);H_x^{s+5})}, \norm{w}_{L^\infty([0,T);H_x^{s+5})}) .\]
Now, recall that by definition, $\partial_x w=-\partial_t\zeta\in L^\infty([0,T);H_x^{s+4})$, and $w\in L^\infty([0,T);L^2)$ by identity~\eqref{id-hiui} and the uniform control of $\overline{u}_i$; see \eg~\cite[Proposition~4]{DucheneIsrawiTalhouk14}. Thus we control $w\in L^\infty([0,T);H_x^{s+5})$, and similarly $\partial_t w\in L^\infty([0,T);H_x^{s+4})$, and the Proposition follows.
\end{proof}

\begin{Proposition}[Stability]\label{P.stability} 
 Let $N\geq 4$ and $U_i =(\zeta_{i},w_{i})^\top \in L^\infty([0,T);\X^N\times \Y^N)$ solution to~\eqref{GN-w} with remainder terms $(0,r_i)^\top$. Assume $\zeta_i$ satisfy~\eqref{C.depth-mr},\eqref{C.hyp0},\eqref{C.hyperbolicity} with $h_0,k_0,K>0$. Set $2\leq n\leq N-1$ and assume that $\partial^\alpha r_i\in L^1([0,T);(\Y^0)^\star)$ for any $|\alpha|\leq n$.

Then there exists $0<T^\star\leq T$ such that for all $t \in [0,T^\star)$
\[
	E^n(U_1-U_2)^{1/2}\leq {\bf C}_0 E^n(U_1\id{t=0}-U_2\id{t=0})^{1/2}e^{\lambda t}+{\bf C}_0\int^{t}_{0} e^{\epsilon\lambda( t-t')} f_n(t')dt'.
\]
with 
\[ \lambda \ = \ {\bf C}_0 \times \Big(\epsilon+\Upsilon_\F \Norm{w_1}_{L^\infty([0,T];\Z^2)}^2 \Big) ,\ \quad f_n(t)=\sum_{|\alpha|\leq n}\norm{\partial^\alpha r_1-\partial^\alpha r_2}_{(\Y^0)^\star},\]
and 
${\bf C}_0=C(\m,h_0^{-1},k_0^{-1},K,\Norm{U_1}_{L^\infty([0,T];\X^4\times\Y^4)},\Norm{U_2}_{L^\infty([0,T];\X^4\times\Y^4)})$. Moreover, one has
\[ (T^\star)^{-1}\leq {\bf C}_0 (\epsilon +\Upsilon_\F^{1/2} \Norm{w_1}_{L^\infty([0,T];\Z^1)}^2+\Upsilon_\F \Norm{w_1}_{L^\infty([0,T];\Z^2)}^2).\]
\end{Proposition}
\begin{proof}
By Lemma~\ref{L.linearize}, for any $|\alpha|\leq n\leq N-1$, $U_i^{(\alpha)}\eqdef (\partial^\alpha \zeta_i,\partial^\alpha w_i)^\top$ satisfies~\eqref{linearized} with remainder terms $\t r^{(\alpha)}_i\eqdef r^{(\alpha)}_i+ \partial^\alpha r_i \in L^1([0,T);(\Y^0)^\star)$, and
\[\norm{\t r^{(\alpha)}_1-\t r^{(\alpha)}_2}_{(\Y^0)^\star} \leq \norm{\partial^\alpha r_1-\partial^\alpha r_2}_{(\Y^0)^\star}\\
+C_0\times (\epsilon +\Upsilon_\F^{1/2} \norm{w_1}_{\Z^1}+\Upsilon_\F \norm{w_1}_{\Z^1}^2)\times E^{|\alpha|}(U_1-U_2)^{1/2},
\]
with $C_0=C(\m,h_0^{-1},E^N(U_1),E^N(U_2))$. By Lemma~\ref{L.energy-estimate-diff}, one has
\[ E^0( U_1^{(\alpha)}- U_2^{(\alpha)})^{1/2} \ \leq \ {\bf C}_0\ E^0(U_1^{(\alpha)}\id{t=0}-U_2^{(\alpha)}\id{t=0})^{1/2}e^{\lambda t} \ + \ {\bf C}_0\int_0^t f^{(\alpha)}(t') e^{\lambda (t-t')}\dd t', \]
with ${\bf C}_0,\lambda$ as in the statement and
\[ f^{(\alpha)}(t)=\norm{\t r_1^{(\alpha)}-\t r_2^{(\alpha)}}_{(\Y^0)^\star}+ \epsilon\norm{U_2}_{(W^{3,\infty}_x)^2} \norm{U_1-U_2}_{\X^2\times\Y^2}.\]
Since $n\geq 2$ and $N\geq 4$, one can restrict $T^\star$ as in the statement and augment ${\bf C}_0$ if necessary so that the estimate holds.
\end{proof}

The following
is now a straightforward consequence of Theorem~\ref{T.WP}, Propositions~\ref{P.consistency} and~\ref{P.stability}.
\begin{Proposition}[Full justification]\label{P.justification}
Let $U^0\equiv(\zeta^0,w^0)^\top\in \X^N\times \Y^N$ with $N$ sufficiently large, and satisfying the hypotheses of Theorem~\ref{T.WP}. Define $\psi^0$ with $\partial_x w^0=-\frac1\mu G^{\mu}[\epsilon\zeta^0]\psi^0$ and assume that $(\zeta^0,\psi^0)^\top$ satisfies the hypotheses of Theorem 5 in~\cite{Lannes13}. Then there exists $C,T>0$ such that
\begin{itemize}
\item There exists a unique solution $U\equiv (\zeta,\psi)^\top$ to the full Euler system~\eqref{fE-DN}, defined on $[0,T]$ and with initial data $(\zeta^0,\psi^0)^\top$ (provided by Theorem 5 in~\cite{Lannes13});
\item There exists a unique solution $U_\F\equiv (\zeta_\F,w_\F)^\top$ to our modified Green-Naghdi model~\eqref{GN-w}, defined on $[0,T]$ and with initial data $(\zeta^0,w^0)^\top$ (provided by Theorem~\ref{T.WP});
\item Defining $\partial_x w=-\frac1\mu G^{\mu}[\epsilon\zeta]\psi =-\partial_t\zeta$, one has for any $t\in[0,T]$,
\[ \Norm{(\zeta,w)-(\zeta_\F,w_\F) }_{L^\infty([0,t];\X^0\times\Y^0)}\leq C \ \mu^2\ t.\]
\end{itemize}
\end{Proposition}

\section{The Saint-Venant system}\label{S.SV}
The Saint-Venant system (with surface tension) is obtained from our Green-Naghdi models~\eqref{GN-w} by setting $\mu=0$. The results of Section~\ref{S.justification} thus apply as a particular case. However, it is possible to obtain sharper results by considering the system obtained after the following change of variable $\overline{v}\eqdef \frac{h_1+\gamma h_2}{h_1h_2}w=\overline{u}_2-\gamma\overline{u}_1$:
\begin{equation}\label{SV}
\left\{ \begin{array}{l}
\displaystyle\partial_{ t}{\zeta} \ + \ \partial_x \big(H(\epsilon\zeta)\overline{v}\big) \ =\ 0, \\ \\
\partial_{ t} \overline{v} \ + \ (\gamma+\delta)\partial_x{\zeta} \ + \ \frac{\epsilon}{2} \partial_x\Big(H'(\epsilon\zeta)| \overline{v}|^2\Big) \ - \ \frac{\gamma+\delta}{\Bo}\partial_x ^3\zeta \ = \ 0,
\end{array}
\right.
\end{equation}
where we denote $H(\epsilon\zeta)= \frac{h_1 h_2}{h_1+\gamma h_2}$.

In the following, we quickly review the steps of the method developed in Appendix~\ref{S.proof},
providing results without proof.

The analogue of Lemma~\ref{L.linearize} is the following:
\begin{Lemma}\label{L.linearize-SV}
Let $U=(\zeta,\overline{v})^\top\in \X^N\times H^N$ with $N\geq 2$, solution to~\eqref{SV} and satisfying
\begin{equation}\label{C.depth-SV}
h_1(\epsilon\zeta)=1-\epsilon\zeta\geq h_0>0,\qquad h_2(\epsilon\zeta)=\delta^{-1}+\epsilon\zeta\geq h_0>0.
\end{equation}
For any $\alpha=(\alpha_1,\alpha_2)$ such that $|\alpha|\leq N$, denote $U^{(\alpha)}\eqdef (\partial^\alpha \zeta,\partial^\alpha \overline{v})^\top$ and $\overline{v}^{\langle\check\alpha\rangle}\eqdef (\partial^{\alpha-\e_1} \overline{v},\partial^{\alpha-\e_2} \overline{v})^\top$ (if $\alpha_j=0$, then $\partial^{\alpha-\e_j} \overline{v}=0$ by convention). Then $U^{(\alpha)}$ satisfies:
\[
\left\{ \begin{array}{l}
\displaystyle\partial_{ t}{\zeta}^{(\alpha)} \ + \ \partial_x \big(H(\epsilon\zeta)\overline{v}^{(\alpha)} \big) \ + \ \epsilon H'(\epsilon\zeta)\overline{v} \partial_x \zeta^{(\alpha)} \ + \ \partial_x\check{\mathfrak{d}}_{\alpha}[\epsilon\zeta] \overline{v}^{\langle \check\alpha\rangle} \ =\ r_1^{(\alpha)}, \\ \\
\partial_{ t} \overline{v}^{(\alpha)} \ + \ \partial_x\mathfrak{a}_{\rm SV}[\epsilon\zeta,\epsilon \overline{v}]\zeta^{(\alpha)} \ + \ \epsilon H'(\epsilon\zeta)\overline{v} \partial_x \overline{v}^{(\alpha)} \ = \ r_2^{(\alpha)},
\end{array}
\right.
\]
with $\mathfrak{\check d}_{\alpha}[\epsilon\zeta] \overline{v}^{\langle \check\alpha\rangle} \eqdef \sum_{j\in\{1,2\}}\alpha_j H'(\epsilon\zeta) (\epsilon\partial^{\e_j}\zeta)(\partial^{\alpha-\e_j} \overline{v})$ and
\[ \mathfrak{a}_{\rm SV}[\epsilon\zeta,\epsilon \overline{v}]\bullet\eqdef \Big((\gamma+\delta)+\frac{\epsilon^2}{2}H''(\epsilon\zeta) |\overline{v}|^2\Big)\bullet-\frac{\gamma+\delta}\Bo \partial_x^2\bullet ;
\]
and $r^{(\alpha)}[\epsilon\zeta,\epsilon\overline{v}]=(r^{(\alpha)}_1[\epsilon\zeta,\epsilon \overline{v}],r^{(\alpha)}_2[\epsilon\zeta,\epsilon \overline{v}])^\top\in \X^0\times L^2$ satisfies
\begin{align*}
\norm{r^{(\alpha)}[\epsilon\zeta,\epsilon\overline{v}] }_{\X^0\times L^2} & \leq \epsilon\ C(\m,h_0^{-1}, \norm{\zeta}_{\X^N},\norm{\overline{v}}_{H^N}) \times \big(\norm{\zeta}_{\X^N}+\norm{\overline{v}}_{H^N}\big) .
\end{align*}
\end{Lemma}

One has immediately the following analogue of Lemma~\ref{L.hyperbolicity}.
\begin{Lemma}\label{L.hyperbolicity-SV} Let $(\zeta, \overline{v})^\top\in L^\infty\times L^\infty$ be such that $\epsilon\zeta$ satisfies~\eqref{C.depth-SV} with $h_0>0$, and
\begin{equation}\label{C.hyp-SV} (\gamma+\delta)+\frac{\epsilon^2}{2} H''(\epsilon\zeta)|\overline{v}|^2=(\gamma+\delta)-\gamma \epsilon^2 \frac{(h_1+h_2)^2}{(h_1+\gamma h_2)^3}|\overline{v}|^2\geq k_0>0.\end{equation}
 Then there exists $K_0,K_1=C(\m,h_0^{-1},k_0^{-1},\epsilon \norm{\zeta}_{L^\infty})$ such that
 \begin{align*}
 \forall f,g \in \X^0,& \qquad &\norm{\big\langle\mathfrak a_{\rm SV}[\epsilon\zeta,\epsilon \overline{v}] f,g\big\rangle_{(\X^0)^\star}}&\leq K_1 \norm{f}_{\X^0}\norm{g}_{\X^0},\\
 \forall f,g \in L^2,& \qquad &\norm{\big(H(\epsilon\zeta) f,g\big)_{L^2}}&\leq K_1\norm{f}_{L^2}\norm{g}_{L^2},\\
 \forall f \in \X^0,& \qquad &\big\langle\mathfrak a_{\rm SV}[\epsilon\zeta,\epsilon \overline{v}] f,f\big\rangle_{(\X^0)^\star}&\geq \frac1{K_0}\norm{f}_{\X^0}^2,\\
 \forall f \in L^2,& \qquad &\big(H(\epsilon\zeta) f,f\big)_{L^2}&\geq \frac1{K_0}\norm{f}_{L^2}^2.
 \end{align*}
 \end{Lemma}
 
 A priori energy estimates are obtained by adding the $L^2$ inner product of the first equation with $\mathfrak{a}_{\rm SV}\zeta^{(\alpha)}$, and the one of second one with $H(\epsilon\zeta)\overline{v}^{(\alpha)}+\check{\mathfrak{d}}_{\alpha}[\epsilon\zeta] \overline{v}^{\langle \check\alpha\rangle} $, and following the proof of Lemmata~\ref{L.energy-estimate} and~\ref{L.energy-estimate-diff}.
Applying the strategy of Section~\ref{S.WP-conclusion}, one then obtains the following analogue of Theorem~\ref{T.WP}.
\begin{Theorem}\label{T.WP-SV}
Let $N\geq 2$ and $U^0\eqdef(\zeta^0,\overline{v}^0)^\top\in \X^N\times H^N$, satisfying~\eqref{C.depth-SV},\eqref{C.hyp-SV} with $h_0,k_0>0$. Then there exists $T>0$ and a unique solution $U\eqdef (\zeta,\overline{v})^\top\in C^0_{\rm w}([0,T);\X^N\times H^N)$ satisfying~\eqref{SV}. Moreover, there exists ${\bf C}_0=C(\m,h_0^{-1},k_0^{-1},\norm{U^0}_{\X^N\times H^N})$ such that
\[ T^{-1}\leq {\bf C}_0 \times\epsilon
\quad \text{ and }\quad \sup_{t\in [0,T)} \left( \norm{\zeta}_{\X^N}+\norm{\overline{v}}_{H^N}\right) \ \leq\ {\bf C}_0\left(\norm{\zeta^0}_{\X^N}+\norm{\overline{v}^0}_{H^N}\right).\]
\end{Theorem}
\begin{Remark}
Theorem~\ref{T.WP-SV} is valid uniformly with respect to the parameter $\Bo^{-1}$, and the result holds in particular in the case without surface tension: $\Bo^{-1}=0$. This case is however straightforward as the Saint-Venant system is then a quasilinear system, and the result was stated in particular in~\cite{GuyenneLannesSaut10}. Assumption~\eqref{C.hyp-SV} corresponds exactly to the hyperbolicity condition provided therein.

Notice that~\eqref{C.hyp-SV} is automatically satisfied in the limit $\gamma\to0$: Kelvin-Helmholtz instabilities disappear in the water-wave setting, and~\eqref{C.depth-SV} suffices to ensure the stability of the flow.
\end{Remark}
\begin{Remark}
While completing our work, we have been informed of related results covering the water-wave setting ($\gamma=0,\delta=1$) of our Theorem~\ref{T.WP-SV}. Saut, Wang and Xu in \cite{SautWangXu} generalize the result to Boussinesq systems in dimension $d=2$, while Chiron and Benzoni-Gavage \cite{ChironBenzoni-Gavage} treat more general Euler-Korteweg systems.
\end{Remark}

\appendix

\section{Notations and functional setting}\label{S.functional}

The notation $a\lesssim b$ means that 
 $a\leq C_0\ b$, where $C_0$ is a nonnegative constant whose exact expression is of no importance.
 We denote by $C(\lambda_1, \lambda_2,\dots)$ a nonnegative constant depending on the parameters
 $\lambda_1$, $\lambda_2$,\dots and whose dependence on the $\lambda_j$ is always assumed to be nondecreasing.

In this paper, we sometimes work with norms involving derivatives in both space and time variables. We find it convenient to use the following sometimes non-standard notations.
\begin{itemize}
\item For $1\leq p<\infty$, we denote $L^p_x=L^p=L^p(\RR)$ the standard Lebesgue spaces associated with the norm
\[\vert f \vert_{L^p}=\left(\int_{\RR}\vert f(x)\vert^p dx\right)^{\frac1p}<\infty.\] 
The real inner product of any functions $f_1$ and $f_2$ in the Hilbert space $L^2(\RR)$ is denoted by
\[
 \big(f_1,f_2\big)_{L^2}=\int_{\RR}f_1(x)f_2(x) \dd x.
 \]
 The space $L^\infty=L^\infty_x=L^\infty(\RR)$ consists of all essentially bounded, Lebesgue-measurable functions
 $f$ with the norm
\[
\norm{f}_{L^\infty}= {\rm ess\,sup}_{x\in\RR} | f(x)|<\infty.
\]
\item We acknowledge the fact that only space derivatives are involved by the use of a subscript.
For $k\in\NN$, we denote by $W^{k,\infty}_x(\RR)=\{f, \mbox{ s.t. } \forall 0\leq j\leq k,\ \partial_x^j f\in L^{\infty}(\RR)\}$ endowed with its canonical norm.\\
For any real constant $s\in\RR$, $H^s_x=H^s_x(\RR)$ denotes the Sobolev space of all tempered distributions $f$ with the norm $\vert f\vert_{H^s_x}=\vert \Lambda^s f\vert_{L^2} < \infty$, where $\Lambda$ is the pseudo-differential operator $\Lambda=(1-\partial_x^2)^{1/2}$. We denote $H^\infty_x=\cap_{N\in\NN} H^N_x$.
\item In absence of subscript, the derivatives are with respect to space and time, and thus apply to functions defined on $ (t,x)\in[0,T)\times \RR$. Thus for $N\in\NN$, $W^{N,\infty}$ is the space of functions endowed with the following norm:
\[ \norm{f}_{W^{N,\infty}} \ = \ \sum_{|\alpha|\leq N} \norm{\partial^\alpha f}_{L^\infty_x},\]
where we use the standard multi-index notation: $\alpha\in \NN^2$, $\partial^{(\alpha_1,\alpha_2)}=\partial_t^{\alpha_1}\partial_x^{\alpha_2}$ and $|\alpha|=\alpha_1+\alpha_2$. In particular, $\partial^{\e_1}=\partial^{(1,0)}=\partial_t$ and $\partial^{\e_2}=\partial^{(0,1)}=\partial_x$. 

Similarly, $H^{N}$ is the space of functions endowed with
\[ \norm{f}_{H^{N}}^2 \ = \ \sum_{|\alpha|\leq N} \norm{\partial^\alpha f}_{L^2_x}^2 .\]
We denote $H^\infty=\cap_{N\in\NN} H^N$.
\item Given $\mu,\gamma ,\Bo^{-1}\geq 0$ and $\F_i$ ($i=1,2$) admissible functions (in the sense of Definition~\ref{D.admissible}), we define $\X^0$, $\Y^0$, $\W^0$, $\Z^0$ as the completion of the Schwartz space, $\S(\RR)$, for the following norms:
\begin{align*}
\norm{f}_{\X^0}^2 &\eqdef \norm{f}_{L^2}^2+\frac1\Bo \norm{\partial_x f}_{L^2}^2,& \quad
\norm{f}_{\Y^0}^2 &\eqdef \norm{f}_{L^2}^2+\mu\gamma\norm{\partial_x\F^\mu_1 f}_{L^2}^2+\mu\norm{\partial_x\F^\mu_2 f}_{L^2}^2,\\
\norm{f}_{\W^0} &\eqdef \norm{\widehat{f}}_{L^1}+\frac1\Bo \norm{\widehat{\partial_x f}}_{L^1},&\quad
\norm{f}_{\Z^0} &\eqdef \norm{\widehat{f}}_{L^1}+\sqrt{\mu\gamma}\norm{\widehat{\partial_x\F^\mu_1 f}}_{L^1}+\sqrt{\mu}\norm{\widehat{\partial_x\F^\mu_2 f}}_{L^1}.
\end{align*}
For $N\in\NN$ we define consistently with above the norms controlling space and time derivatives:
\begin{align*}
\norm{f}_{\X^N}^2 &\eqdef \sum_{|\alpha|\leq N} \norm{\partial^\alpha f}_{\X^0}^2, &\quad 
\norm{f}_{\Y^N}^2 &\eqdef \sum_{|\alpha|\leq N} \norm{\partial^\alpha f}_{\Y^0}^2, \\
\norm{f}_{\W^N} &\eqdef \sum_{|\alpha|\leq N} \norm{\partial^\alpha f}_{\W^0}, &\quad 
\norm{f}_{\Z^N} &\eqdef \sum_{|\alpha|\leq N} \norm{\partial^\alpha f}_{\Z^0}.
\end{align*}

\item Denoting $X$ any of the previously defined functional spaces, we denote by $X^{\star}$ its topological dual, endowed with the norm $\norm{\varphi}_{X^\star}=\sup( |\varphi(f)|,\ \norm{f}_{X}\leq 1)$; and by $\langle\cdot,\cdot\rangle_{(X)^{\star}}$ the $(X^{\star}-X)$ duality brackets. 
\item For any function $u=u(t,x)$ defined on $ [0,T)\times \RR$ with $T>0$, and any of the previously defined functional spaces, $X$, we denote $L^\infty([0,T);X)$ the space of functions such that $u(t,\cdot)$ is controlled in $X$, uniformly for $t\in[0,T)$, and denote the associated norm
 \[\Norm{u}_{L^\infty([0,T);X)} \ = \ \esssup_{t\in[0,T)}\norm{u(t,\cdot)}_{X} \ < \ \infty.\]
For $k\in\NN$, $C^k([0,T);X)$ denotes the space of $X$-valued continuous functions on $[0,T)$ with continuous derivatives up to the order $k$. Finally, $C^0_{\rm w}([0,T);X)$ is the space of continuous functions with values in $X$, given the weak topology.
\end{itemize}

\section{Proof of the Theorem \ref{T.WP}}\label{S.proof}

This section is dedicated to the proof of our main result, Theorem \ref{T.WP}. The proof relies on energy estimates which are also used in the proof of Proposition~\ref{P.stability}.

Our strategy is similar to the one used for the full Euler system with surface tension by Lannes~\cite{Lannes,Lannes13}, and originates from an idea of Rousset and Tzvetkov~\cite{RoussetTzvetkov10,RoussetTzvetkov11}.
The main difference with respect to the traditional methods for quasilinear systems is that we treat time derivatives in the same way as space derivatives. In particular, the main tool of the analysis is the control of a {\em space-time} energy. The reason for such a strategy is that 
\begin{itemize}
\item the two unknowns, $\zeta$ and $w$, are controlled in different functional spaces, one being continuously embedded in the other but to the price of a non-uniform constant (see Lemma~\ref{L.embedding}), and the inclusion being strict;
\item the most singular term of the system, namely the one which involves the operator of highest order, comes from the surface tension component, and couples the two unknowns (it appears as an off-diagonal component of the quasilinearized system).
\end{itemize}
This is why one cannot use standard energy methods in Sobolev-based functional spaces, as commutator estimates fail to control all coupling terms; see also the discussion in~\cite{SautWangXu}.

More precisely, our strategy is as follows. In Lemma~\ref{L.linearize} below, we ``quasilinearize'' the system. We differentiate several times the equations with respect to space and time, and extract the leading order components. The quasilinear system we consider is the complete system of all the equations satisfied by the original unknowns and their space-time derivatives up to sufficiently high order. Thus only $L^2$-type estimates on the aforementioned linear ``block'' systems will be required. In Section~\ref{S.preliminary}, we study the operators involved in the block systems, and exhibit sufficient conditions for the existence of a coercive symmetrizer of the system in Lemma~\ref{L.hyperbolicity}. This yields a priori energy estimates in Section~\ref{S.energy}. Finally, in Section~\ref{S.WP-conclusion}, we explain how to deduce from these energy estimates the well-posedness of the linear block systems (Lemma~\ref{L.WP-linear}), and in turn the well-posedness of the nonlinear system (Theorem~\ref{T.WP}).

\subsection{Technical results}\label{S.product}
In this section, we provide tools (injections and product estimates) similar to the classical ones concerning Sobolev spaces, for the functional spaces $\X^N,\Y^N,\Z^N$, as defined in Appendix~\ref{S.functional}. 

Let us fix $\mu,\gamma ,\Bo^{-1}\geq 0$ and $\F_i$ ($i=1,2$) admissible functions (in the sense of Definition~\ref{D.admissible}). In particular there exists $K_{\F_i}>0$ and $\varsigma\in[0,1]$ such that
\begin{equation}\label{C.Fi}
|\F_i(\xi)|^2\leq \min\big\{1,K_{\F_i}|\xi|^{-2\varsigma}\big\}.
\end{equation}

The following standard injections will be frequently used, sometimes without notice:
\begin{equation}\label{injections} 
\norm{f}_{L^\infty} \leq \norm{\widehat f}_{L^1}\lesssim \norm{f}_{H_x^{t_0}} \ (t_0>1/2); \quad \text{thus}\quad \norm{f}_{\Z^N}\lesssim \norm{f}_{\Y^{N+1}} \text{ and }\norm{f}_{\Z^N}\lesssim \norm{f}_{H^{N+t_0+1}}.
\end{equation}

One immediately sees that the space $\X^0$ is continuously embedded in $\Y^0$; the following Lemma precises the norm of the inclusion map.
\begin{Lemma}\label{L.embedding}
If $\F_i$ satisfies~\eqref{C.Fi}, then
\[
\forall f\in \X^0(\RR), \qquad \norm{f}_{\Y^0}^2\leq \left(1+(\gamma K_{\F_1}+K_{\F_2}) (\mu\Bo )^{1-\varsigma} \right) \norm{f}_{\X^0}^2.
\]
\end{Lemma}
\begin{proof}The inequality is a simple consequence of Parseval's identity and Young's inequality:
\[\mu\norm{\partial_x\F_i(\sqrt\mu D)f}_{L^2}^2\leq K_{\F_i} \mu^{1-\varsigma}\int |\xi|^{2-2\varsigma}|\widehat f(\xi)|^2\ \dd\xi\leq K_{\F_i}(\mu\Bo)^{1-\varsigma} \int (1+\frac1\Bo|\xi|^2)|\widehat f(\xi)|^2\ \dd\xi,\]
where we used $\Bo^{\varsigma}|\xi|^{2-2\varsigma}\leq \varsigma \Bo+ (1-\varsigma) |\xi|^2$. 
\end{proof}

Sobolev spaces $H^N_x$ and $W^{N,\infty}_x$ enjoy straightforward product estimates, which are immediately extended to $\X^N$ and $\W^N$:
\[
\norm{fg}_{\X^0}\lesssim \norm{g}_{\W^0}\norm{f}_{\X^0} \lesssim \norm{g}_{\X^1}\norm{f}_{\X^0} \ ,\qquad \norm{fg}_{\W^0} \lesssim \norm{f}_{\W^0} \norm{g}_{\W^0} ;
\]
and therefore 
\[ \forall N\geq 1, \quad \norm{fg}_{\X^N}\lesssim \norm{ f}_{\X^N}\norm{g}_{\X^N}.\]

The following Lemma shows that spaces $\Y^N$ and $\Z^N$ enjoys similar estimates, thanks to the sub-additivity property of admissible functions (recall Definition~\ref{D.admissible}).
\begin{Lemma}\label{L.produit} Let $\F_i:\RR\to\RR^+$ $(i=1,2)$ be admissible functions.
Then for any $1\leq p,q,\t p,\t q, r\leq \infty$ satisfying $1+\frac1r=\frac1p+\frac1q=\frac1{\t p}+\frac1{\t q}$, one has
\begin{equation}\label{prod-pqr}
\norm{\widehat{\partial_x\F_i^\mu\{fg\}}}_{L^r}\leq \norm{\widehat f}_{L^p} \norm{\widehat{\partial_x \F^\mu_i g}}_{L^q} +\norm{\widehat g}_{L^{\t p}} \norm{\widehat{\partial_x \F^\mu_i f}}_{L^{\t q}}.
 \end{equation}
It follows in particular:
\begin{align}
\label{prod-Y-0} \norm{fg}_{\Y^0}&\lesssim \norm{g}_{\Z^0}\norm{f}_{\Y^0} \lesssim \norm{g}_{\Y^1}\norm{f}_{\Y^0} \ ,\\
\forall N\geq 1, \quad \label{prod-Y-N} \norm{fg}_{\Y^N}&\lesssim \norm{ f}_{\Y^N}\norm{g}_{\Y^N}\ , \\
\label{prod-Z-0} \norm{fg}_{\Z^0} &\lesssim \norm{f}_{\Z^0} \norm{g}_{\Z^0} .
\end{align}
\end{Lemma}
\begin{proof}
From the sub-additivity, one has $\sqrt\mu|\xi|\F_i(\sqrt\mu\xi)\leq \sqrt\mu|\eta|\F_i(\sqrt\mu\eta)+\sqrt\mu|\xi-\eta|\F_i(\sqrt\mu(\xi-\eta))$. Thus
\begin{align*}\norm{\widehat{\partial_x\F_i^\mu\{fg\}}}_{L^r}^r&= \int_\RR \big(|\xi|\F_i(\sqrt\mu\xi)\big)^r |\widehat f\star \widehat g|^r(\xi)\dd \xi=\int_\RR \dd\xi \left|\int _\RR\dd\eta |\xi|\F_i(\sqrt\mu\xi) \widehat f(\eta)\overline{\widehat g(\xi-\eta) }\right|^r\\
&\leq\int \dd\xi \left|\int \dd\eta |\eta|\F_i(\sqrt\mu\eta) |\widehat f|(\eta)|\widehat g|(\xi-\eta)+|\xi-\eta|\F_i(\sqrt\mu(\xi-\eta)) |\widehat f|(\eta)||\widehat g|(\xi-\eta) \right|^r\\
&\leq\int \dd\xi \left| (|\widehat{\partial_x \F_i^\mu f}|\star|\widehat g|)(\xi)+ (|\widehat f|\star|\widehat{\partial_x \F_i^\mu g}|)(\xi) \right|^r ,
\end{align*} 
where we used that $|\widehat{\partial_x \F_i^\mu f}|(\xi)=|i\xi \F_i(\sqrt\mu\xi) \widehat{f}(\xi)|=|\xi|\F_i(\sqrt\mu\xi) |\widehat{f}(\xi)|$ since $\F_i(\sqrt\mu\xi)\geq 0 $.
Estimate~\eqref{prod-pqr} follows from Young's inequality for convolutions.

Estimate~\eqref{prod-Y-0} is deduced with $r=p=\t q=2$ and $\t p=q=1$, and using~\eqref{injections}.

Estimate~\eqref{prod-Y-N} follows from the above result and triangular inequality,
\[ \norm{fg}_{\Y^N}\leq \sum_{|\alpha+\beta|\leq N}C_{\alpha,\beta,N} \norm{(\partial^\alpha f)(\partial^\beta g)}_{\Y^0}\lesssim \sum_{|\alpha|\leq N-1,|\beta|\leq N} \norm{\partial^\alpha f}_{Y^1_{\F_\mu}}\norm{\partial^\beta f}_{Y^0_{\F_\mu}}.\]

Estimate~\eqref{prod-Z-0} follows from~\eqref{prod-pqr} with $p=\t p=q=\t q=r=1$.
\end{proof}

We now provide Schauder-type estimates in our functional spaces.
\begin{Lemma} \label{L.f/h}
Let $H\in C^\infty(-\delta^{-1},1)$ and $\epsilon\zeta\in L^\infty$ such that 
\[ h_1(\epsilon\zeta)=1-\epsilon\zeta\geq h_0>0, \quad h_2(\epsilon\zeta)=\delta^{-1}+\epsilon\zeta\geq h_0>0.\]
Then, denoting $H_{n,h_0}\eqdef\norm{H}_{C^n([-\delta^{-1}+h_0,1-h_0])}$ and fixing $t_0>1/2$, one has
\begin{itemize}
\item For any $s\geq 0$, if $\zeta\in H^s$ and $f\in H^s_x$, then one has with $n\in\NN,n\geq \max\{s,t_0\}$:
\[\norm{H(\epsilon\zeta)f}_{H^s_x} \leq C(h_0^{-1},H_{n,h_0},\norm{\epsilon\zeta}_{H_x^{\max\{s,t_0\}}})\norm{f}_{H^s_x}.
\]
\item For any $\widehat f\in L^1$, one has
\[
 \norm{\widehat{H(\epsilon\zeta)f}}_{L^1}
\leq C(h_0^{-1},H_{1,h_0},\norm{\epsilon\zeta}_{H_x^{t_0}})\norm{\widehat f}_{L^1}.
\]
\item For any $N\in \NN$, if $\zeta\in H_x^{t_0+1+N}$ and $f\in \Z^N$, then one has 
\[ \norm{H(\epsilon\zeta)f}_{\Z^N}\leq C(h_0^{-1},H_{2+N,h_0},\norm{\epsilon\zeta}_{H_x^{t_0+1+N}})\norm{f}_{\Z^N}.\]
\item For any $N\in \NN$, if $\zeta\in H_x^{t_0+1+N}$ and $f\in \Y^N$, then one has 
\[ \norm{H(\epsilon\zeta)f}_{\Y^N}\leq C(h_0^{-1},H_{2+N,h_0},\norm{\epsilon\zeta}_{H_x^{t_0+1+N}})\norm{f}_{\Y^N}.\]
\end{itemize}
\end{Lemma}
\begin{proof}
In each case, we decompose $H(\epsilon\zeta)f=H(0)f+\big(H(\epsilon\zeta)-H(0)\big)f=H(0)f+G_{h_0}(\epsilon\zeta)f$ where $G_{h_0}$ is such that $G_{h_0}\in C^\infty(\RR)$, $G_{h_0}(x)=H(x)-H(0)$ for $x\in [-\delta^{-1}+h_0,1-h_0]$ and $G_{h_0}(x)=0$ for $x\in \RR\setminus [-\delta^{-1},1]$. It is clear that, since $\min\{h_1(\epsilon\zeta),h_2(\epsilon\zeta)\}\geq h_0>0$, one can construct such a $G_{h_0}$ satisfying additionally: for any $n\in \NN$, $\norm{G_{h_0}}_{C^n}=C(h_0^{-1},H_{n,h_0})$.

The first estimate is a direct consequence of a classical Schauder-type estimates in Sobolev spaces; see \eg~\cite{Tao06}. As for the second, one has
\[ \norm{\widehat{H(\epsilon\zeta)f}}_{L^1}\leq \norm{\widehat{H(0)f}}_{L^1}+\norm{\widehat{G_{h_0}(\epsilon\zeta)}\star\widehat{f}}_{L^1}\leq H(0) \norm{\widehat{f}}_{L^1}+\norm{\widehat{G_{h_0}(\epsilon\zeta)}}_{L^1}\norm{\widehat{f}}_{L^1}.\]
The second estimate now follows from~\eqref{injections} and applying the above result:
\[ \norm{\widehat{G_{h_0}(\epsilon\zeta)}}_{L^1}\lesssim \norm{ G_{h_0}(\epsilon\zeta) }_{H_x^{t_0}}\leq C(h_0^{-1},H_{1,h_0},\norm{\epsilon\zeta}_{H_x^{t_0}}).\]

Using that $\Z^0$ is an algebra, one has
\[ \norm{H(\epsilon\zeta)f}_{\Z^0}\leq \norm{H(0)f}_{\Z^0}+\norm{G_{h_0}(\epsilon\zeta)f}_{\Z^0}\leq H(0) \norm{f}_{\Z^0}+\norm{G_{h_0}(\epsilon\zeta)}_{\Z^0}\norm{f}_{\Z^0}.\]
Since $\norm{u}_{\Z^0}\leq \norm{u}_{H_x^{t_0+1}}$ for any $u\in H_x^{t_0+1}$, one deduces the third estimate for $N=0$ as above. The case $N\geq 1$ is obtained by induction, differentiating $N$ times $H(\epsilon\zeta)f$ and applying Leibniz's rule.

The last estimate is obtained identically since by Lemma~\ref{L.produit}
\[ \norm{H(\epsilon\zeta)f}_{\Y^0}\leq \norm{H(0)f}_{\Y^0}+\norm{G_{h_0}(\epsilon\zeta)f}_{\Y^0}\leq H(0) \norm{f}_{\Y^0}+\norm{G_{h_0}(\epsilon\zeta)}_{\Z^0}\norm{f}_{\Y^0}.\]
The proof is now complete.
\end{proof}

\subsection{Quasilinearization of the system}\label{S.quasilinearization}
The following Lemma introduces the quasilinear block systems which are central in our analysis.
\begin{Lemma}\label{L.linearize}
Let $U=(\zeta,w)^\top \in \X^N\times \Y^N$ with $N\geq4$, solution to~\eqref{GN-w} and satisfying
\begin{equation}\label{C.depth}
h_1(\epsilon\zeta)=1-\epsilon\zeta\geq h_0>0,\qquad h_2(\epsilon\zeta)=\delta^{-1}+\epsilon\zeta\geq h_0>0.
\end{equation}
For any $\alpha=(\alpha_1,\alpha_2)$ such that $|\alpha|\leq N$, denote $U^{(\alpha)}\eqdef (\partial^\alpha \zeta,\partial^\alpha w)^\top$ and $\zeta^{\langle\check\alpha\rangle}\eqdef (\partial^{\alpha-\e_1} \zeta,\partial^{\alpha-\e_2} \zeta)^\top$ (if $\alpha_j=0$, then $\partial^{\alpha-\e_j} \zeta=0$ by convention). Then one can define $r^{(\alpha)}\in (\Y^0)^{\star}$ such that
\[
\left\{ \begin{array}{l}
\displaystyle\partial_{ t}\zeta^{(\alpha)} \ + \ \partial_x w^{(\alpha)} \ =\ 0, \\ \\
\mathfrak b[\epsilon\zeta] \partial_{ t} w^{(\alpha)} \ + \ \partial_x \mathfrak a[\epsilon\zeta,\epsilon w] {\zeta}^{(\alpha)} \ + \ \partial_x \check{\mathfrak{a}}_\alpha[\epsilon\zeta] \zeta^{\langle\check\alpha\rangle} \ + \ \mathfrak c[\epsilon\zeta,\epsilon w] \partial_x w^{(\alpha)}\ = \ r^{(\alpha)} ,
\end{array}
\right.
\]
where 
\begin{align*}
\mathfrak{a}[\epsilon\zeta,\epsilon w]\bullet &\eqdef\Big( (\gamma+\delta)-\epsilon^2\dfrac{h_1^3 +\gamma h_2^3 }{(h_1 h_2)^3} |w|^2 \Big)\times\bullet-\mu\epsilon^2 \big(\dd_1\R_2^\F[h_2,w]+\gamma\dd_1\R_1^\F[h_1,w]\big)\bullet \\
&\hspace{7cm} -\frac{\gamma+\delta}\Bo \partial_x\left( \frac{ \partial_x\bullet}{(1+\mu\epsilon^2|\partial_x\zeta|^2)^{3/2}}\right)\\
\mathfrak{b}[\epsilon\zeta]\bullet &\eqdef \frac{h_1+\gamma h_2}{h_1 h_2}\bullet \ +\ \mu(\Q_2^\F[h_2]+\gamma \Q_1^\F[h_1])\bullet ,\\
\mathfrak{c}[\epsilon\zeta,\epsilon w]\bullet &\eqdef 2\epsilon \frac{h_1^2 -\gamma h_2^2 }{(h_1 h_2)^2}w\times \bullet -\mu \epsilon \big(\dd \Q_2^\F[h_2](w)-\gamma \dd \Q_1^\F[h_1](w)\big) \bullet \\
&\hspace{5cm} - \mu\epsilon \big(\dd_2\R_2^\F[h_2,w]-\gamma\dd_2\R_1^\F[h_1,w]\big)\bullet ,
\end{align*}
with $\Q_i^\F,\dd \Q_i^\F,\dd_1 \R_i^\F,\dd_2 \R_i^\F$ defined in~\eqref{def-Qi},\eqref{def-dQi},\eqref{def-d1Ri},\eqref{def-d2Ri} below; and 
\[\check{\mathfrak{a}}_\alpha[\epsilon\zeta]\zeta^{\langle\check\alpha\rangle} \eqdef \frac{\gamma+\delta}\Bo \partial_x\left(\sum_{j\in\{1,2\}}3\alpha_j\mu\epsilon^2\frac{(\partial_x\partial^{\e_j}\zeta)(\partial_x\zeta)(\partial_x\zeta^{\langle\check\alpha\rangle} _j)}{(1+\mu\epsilon^2|\partial_x\zeta|^2)^{5/2}} \right) \quad \text{  if $|\alpha|=N$, and $0$ otherwise}
.\]
Moreover, $r^{(\alpha)}=r^{(\alpha)}[\epsilon\zeta,\epsilon w]$ satisfies
\begin{equation}\label{r-est}
\norm{r^{(\alpha)} }_{(\Y^0)^{\star}}\leq C(\m,h_0^{-1}, E^{N}(U)) \times E^{|\alpha|}(U)^{1/2} \times (\epsilon +\Upsilon_\F^{1/2}\norm{w}_{\Z^1}+\Upsilon_\F\norm{w}_{\Z^1}^2) ,
\end{equation}
and
\begin{multline}\label{rdiff-est}
\norm{ r^{(\alpha)}[\epsilon\zeta_1,\epsilon w_1]-r^{(\alpha)}[\epsilon\zeta_2,\epsilon w_2] }_{(\Y^0)^{\star}}
\leq C(\m,h_0^{-1}, E^{N}(U_1), E^{N}(U_2))\times E^{|\alpha|}(U_1-U_2)^{1/2} \\ \times (\epsilon +\Upsilon_\F^{1/2}\norm{w_1}_{\Z^1}+\Upsilon_\F\norm{w_1}_{\Z^1}^2) .
\end{multline}
\end{Lemma}
\begin{proof}
The proof simply consists in differentiating $\alpha$ times the Green-Naghdi system~\eqref{GN-w}. The higher order terms contribute to $ \mathfrak{a}$, $ \mathfrak{b}$, $ \mathfrak{c}$ and $ \check{\mathfrak{a}}$, while lower order terms contribute to $r^{(\alpha)}$. In the following, we explain how the estimates concerning $r^{(\alpha)}$ are obtained.
\medskip

\noindent {\em Contribution from the first order terms,} $\partial_t\big(\frac{h_1+\gamma h_2}{h_1 h_2}w\big)+\frac{\epsilon}{2} \partial_x\big(\frac{h_1^2 -\gamma h_2^2 }{(h_1 h_2)^2} |w|^2\big)$.

Applying Leibniz's rule, one finds
\begin{equation}\label{r1-def}
\partial^\alpha\left( \frac{\epsilon}{2} \partial_x\Big(\frac{h_1^2 -\gamma h_2^2 }{(h_1 h_2)^2}| w|^2\Big) \right)
=-\epsilon^2\dfrac{h_1^3 +\gamma h_2^3 }{(h_1 h_2)^3} |w|^2 \partial_x \partial^\alpha {\zeta} \ + \ \epsilon \dfrac{h_1^2 -\gamma h_2^2 }{(h_1 h_2)^2} w\partial_x \partial^\alpha w+\epsilon r_1^{(\alpha)},
\end{equation}
with 
\[ r_1^{(\alpha)}=\sum_{n=0}^{|\alpha|+1}\sum_{\beta_i,\beta'_j} \epsilon^n C^{(\beta_i,\beta'_j)} G^{(n)}(\epsilon\zeta) \left(\prod_{i=1}^n \partial^{\beta_i}\zeta\right)\left(\prod_{j=1}^2 \partial^{\beta'_i}w\right) \eqdef \sum_{n=0}^{|\alpha|+1}\sum_{\beta_i,\beta'_j} \epsilon^n C^{(\beta_i,\beta'_j)} r_1^{(\beta_i,\beta'_j)},\]
where $(\beta_i,\beta'_j)$ is any $n+2$-tuple of multi-index satisfying 
\[1\leq |\beta_1|\leq \dots \leq |\beta_n|\leq |\alpha|,\quad \ 0\leq |\beta'_1|\leq |\beta'_2|\leq|\alpha|\quad \text{ and } \quad \sum_{i=1}^n \beta_i+\sum_{j=1}^2\beta'_j=\alpha+(0,1),\]
$C^{(\beta_i,\beta'_j)} $ is a constant and $G^{(n)}$ the $n$-th derivative of $G(X)=\frac{h_1^2(X) -\gamma h_2^2(X) }{(h_1 h_2)^2(X)}=\frac{(1-X)^2-\gamma(\delta^{-1}+X)^2}{(1-X)^2(\delta^{-1}+X)^2}$.

We estimate each of these terms as follows :
\begin{itemize}
\item if $|\beta_n|=|\alpha|$, then $0\leq |\beta_1|,\dots,|\beta_{n-1}|,|\beta'_1|,|\beta'_2|\leq 1$, and
\[\norm{ r_1^{(\beta_i,\beta'_j)}}_{L^2}\leq \norm{G^{(n)}(\epsilon\zeta)}_{L^\infty}\norm{\partial^{\beta_n}\zeta}_{L^2}\left(\prod_{i=1}^{n-1} \norm{\partial^{\beta_i}\zeta}_{L^\infty}\right)\left(\prod_{j=1}^2 \norm{\partial^{\beta'_i}w}_{L^\infty}\right);\]
\item otherwise $0\leq |\beta_1|,\dots,|\beta_{n}|,|\beta'_1|\leq |\alpha|-1$, and 
\[\norm{ r_1^{(\beta_i,\beta'_j)}}_{L^2}\leq \norm{G^{(n)}(\epsilon\zeta)}_{L^\infty}\left(\prod_{i=1}^n \norm{\partial^{\beta_i}\zeta}_{L^\infty}\right)\left( \norm{\partial^{\beta'_1}w}_{L^\infty}\norm{\partial^{\beta'_2}\zeta}_{L^2}\right).\]
\end{itemize}
One has $\norm{G^{(n)}(\epsilon\zeta)}_{L^\infty}\leq C(\m,h_0^{-1})$ since $\epsilon\zeta$ satisfies~\eqref{C.depth}; and by Sobolev embedding, 
\[\norm{\partial^\beta u}_{L^\infty}\leq \norm{\partial^{\beta}u}_{H^1_x}\leq \min\big\{\norm{u}_{\X^{1+|\beta|}}, \norm{u}_{\Y^{1+|\beta|}}\big\}.\]
 We deduce immediately that for $N\geq 2$ and $|\alpha|\leq N$,
\begin{equation}\label{r1-est}
\norm{r_1^{(\alpha)}}_{L^2}\leq C\big(\m,h_0^{-1},E^{N}(U)\big) \times E^{|\alpha|}(U)^{1/2}.
\end{equation}
\medskip

For the second contribution, one deduces from the first equation of~\eqref{GN-w}, $\partial_t\zeta=-\partial_xw$, that
\begin{equation} \label{r2-def}
\partial^\alpha \partial_t \left(\frac{h_1+\gamma h_2}{h_1 h_2}w\right) = \frac{h_1+\gamma h_2}{h_1 h_2}\partial_t \partial^\alpha w+\epsilon w \frac{h_1^2-\gamma h_2^2}{h_1h_2}\partial_x \partial^\alpha w +\epsilon r_2^{(\alpha)},
\end{equation}
where $r_2^{(\alpha)}$ is estimated as above:
\begin{equation}\label{r2-est}
\norm{r_2^{(\alpha)}}_{L^2}\leq C\big(\m,h_0^{-1},E^{N}(U)\big) \times E^{|\alpha|}(U)^{1/2}.
\end{equation}
\bigskip

\noindent {\em Contribution from the dispersive terms,} $\mu \partial_{ t} \big(\Q^\F[\epsilon\zeta]w\big)- \mu\epsilon\partial_x\big( \R^\F[\epsilon\zeta]w\big)$.

Define (with a slight abuse of notation with respect to Section~\ref{S.Models})
\begin{equation}\label{def-Qi}
\Q_i^\F[h_i]w\eqdef -\frac13 h_i^{-1}\partial_x \F_i^\mu\big\{h_i^3 \partial_x \F_i^\mu\{h_i^{-1}w\}\big\},
\end{equation}
so that $\Q^\F[\epsilon\zeta]w=\Q_2^\F[h_2]w+\gamma \Q_1^\F[h_1]w$. Differentiating $\alpha+\e_1$ times and using $\partial_t \zeta=-\partial_x w$ yields
\begin{equation}\label{r3i-def}
\partial^\alpha \partial_t \Q_i^\F[h_i]w= \Q_i^\F[h_i]\partial^\alpha \partial_t w- (-1)^i\dd \Q_i^\F[h_i](w)(\epsilon \partial^\alpha \partial_x w )+ r_{3,i}^{(\alpha)}
\end{equation}
where
\begin{multline} \label{def-dQi} \dd \Q_i^\F[h_i](w)\bullet =\frac13 h_i^{-2}\Big( \partial_x \F_i^\mu\big\{h_i^3 \partial_x \F_i^\mu\{h_i^{-1}w\}\big\}\Big)\times \bullet \\-h_i^{-1}\partial_x \F_i^\mu\big\{h_i^2 \partial_x \F_i^\mu\{h_i^{-1}w\} \times \bullet\big\}
+\frac13 h_i^{-1}\partial_x \F_i^\mu\big\{h_i^3 \partial_x \F_i^\mu\{h_i^{-2}w\times\bullet \}\big\};
\end{multline}
and
\[ r_{3,i}^{(\alpha)}=\sum_{\beta_j} C^{(\beta_j)} (\partial^{\beta_1} h_i^{-1})\partial_x \F_i^\mu\big\{(\partial^{\beta_2} h_i^3) \partial_x \F_i^\mu\{(\partial^{\beta_3} h_i^{-1})(\partial^{\beta_4} w)\}\big\}\eqdef \sum_{\beta_j} C^{(\beta_j)} r_{3,i}^{(\beta_j)},\]
where $C^{(\beta_j)} $ is a constant and $(\beta_j)$ is any $4$-tuple of multi-index satisfying 
\[0\leq |\beta_1|,|\beta_2|,|\beta_3|, |\beta_4|\leq |\alpha| \quad \text{ and } \quad \sum_{j=1}^4 \beta_j=\alpha+(1,0),\]

We estimate each of these terms by assuming that $U=(\zeta,w)^\top \in \S(\RR)\times\S(\RR)$, so that for any $f\in \S(\RR)$, the following identities are immediately valid:
\begin{align*} \langle r_{3,i}^{(\beta_j)},f \rangle_{(\Y^0)^{\star}}&= \Big( (\partial^{\beta_1} h_i^{-1})\partial_x \F_i^\mu\big\{(\partial^{\beta_2} h_i^3) \partial_x \F_i^\mu\{(\partial^{\beta_3} h_i^{-1})(\partial^{\beta_4} w)\}\big\}\ , \ f \Big)_{L^2}\\
&=- \Big((\partial^{\beta_2} h_i^3) \partial_x \F_i^\mu\{(\partial^{\beta_3} h_i^{-1})(\partial^{\beta_4} w)\}\ , \ \partial_x \F_i^\mu\{ (\partial^{\beta_1} h_i^{-1}) f\} \Big)_{L^2}.
\end{align*}
The estimates hold as well for $U=(\zeta,w)^\top \in \X^N\times\Y^N$ ($N\geq4$) and $f\in \Y^0$ by density of $\S(\RR)$ in $\Y^0$ and $\X^0$ using standard continuity arguments.
\begin{itemize}
\item if $|\beta_1|=|\alpha|$, then $0\leq |\beta_2|,|\beta_3|,|\beta_4|\leq 1$ and
\[
\left\vert \gamma^{2-i}\mu\langle r_{3,i}^{(\beta_j)},f \rangle_{(\Y^0)^{\star}}\right\vert \leq \gamma^{2-i}\mu \norm{ \partial^{\beta_1} h_i^{-1}}_{L^2} \norm{\partial_x \F_i^\mu\big\{(\partial^{\beta_2} h_i^3) \partial_x \F_i^\mu\{(\partial^{\beta_3} h_i^{-1})(\partial^{\beta_4} w)\}\big\}}_{L^\infty} \norm{f }_{L^2} .
\]
Notice that for $|\beta_1|=|\alpha|\geq 1$
 there exists $j\in\{1,2\}$ such that $\e_j\leq \beta_1$ and
\[ \norm{ \partial^{\beta_1} h_i^{-1}}_{L^2} = \norm{ \partial^{\beta_1-\e_j} \big( h_i^{-2} \epsilon\partial^{\e_j} \zeta\big) }_{L^2} \leq C(\m,h_0^{-1},\norm{\epsilon\zeta}_{W^{|\alpha|-1,\infty}}) \times \epsilon\norm{\zeta}_{H^{|\alpha|}} .\]
Now, using several times Lemma~\ref{L.f/h}, and since $\partial^{\beta_2} h_i^3=3 h_i^{2} \epsilon \partial^{\beta_2}\zeta $ if $|\beta_2|=1$ or $\partial^{\beta_2} h_i^3=h_i^3$ if $|\beta_2|=0$ (and similarly for $ \partial^{\beta_3} h_i^{-1}$), one has
\begin{align*}
\sqrt{\gamma^{2-i}\mu}\norm{\partial_x \F_i^\mu\big\{(\partial^{\beta_2} h_i^3) \partial_x \F_i^\mu\{(\partial^{\beta_3} h_i^{-1})(\partial^{\beta_4} w)\}\big\}}_{L^\infty}&\leq \norm{(\partial^{\beta_2} h_i^3) \partial_x \F_i^\mu\{(\partial^{\beta_3} h_i^{-1})(\partial^{\beta_4} w)\}}_{\Z^0}\\
&  \leq  C(\m,h_0^{-1},\norm{\epsilon\zeta}_{H^3} )\norm{ (\partial^{\beta_3} h_i^{-1})(\partial^{\beta_4} w)}_{\Z^1}\\
&  \leq  C(\m,h_0^{-1},\norm{\epsilon\zeta}_{H^4})\norm{w}_{\Z^2} .
\end{align*}
Therefore, since $\max\{4,|\alpha|\}\leq N$,
\[\left\vert \gamma^{2-i}\mu\langle r_{3,i}^{(\beta_j)},f \rangle_{(\Y^0)^{\star}}\right\vert \leq C(\m,h_0^{-1},\norm{\epsilon\zeta}_{H^N}) \times \epsilon\sqrt{\gamma^{2-i}\mu} \norm{w}_{\Z^2} \norm{\zeta}_{\X^{|\alpha|}} \norm{f}_{L^2}. \]
\item if $|\beta_2|=|\alpha|$, then $0\leq |\beta_1|,|\beta_3|,|\beta_4|\leq 1$ and
\[\left\vert\gamma^{2-i} \mu\langle r_{3,i}^{(\beta_j)},f \rangle_{(\Y^0)^{\star}}\right\vert \leq \gamma^{2-i}\mu\norm{\partial^{\beta_2} h_i^3}_{L^2} \norm{\partial_x \F_i^\mu\{(\partial^{\beta_3} h_i^{-1})(\partial^{\beta_4} w)\}}_{L^\infty} \norm{\partial_x \F_i^\mu\big\{ (\partial^{\beta_1} h_i^{-1})f\big\} }_{L^2} .\]
One has as above $\norm{\partial^{\beta_2} h_i^3}_{L^2} \leq \epsilon \norm{\zeta}_{H^{|\alpha|}} C(\m,\norm{\epsilon\zeta}_{W^{|\alpha|-1}})$.
By Lemma~\ref{L.f/h}, one has
\[
\sqrt{\gamma^{2-i}\mu}\norm{\partial_x \F_i^\mu\big\{ (\partial^{\beta_1} h_i^{-1})f\big\} }_{L^2}\leq \norm{ (\partial^{\beta_1} h_i^{-1})f }_{\Y^0}\leq C(\m,h_0^{-1},\norm{\epsilon\zeta}_{H^3}) \norm{f}_{\Y^0}.\]
The last term is treated identically and one obtains eventually
\[\left\vert\gamma^{2-i} \mu\langle r_{3,i}^{(\beta_j)},f \rangle_{(\Y^0)^{\star}}\right\vert \leq C(\m,h_0^{-1},\norm{\epsilon\zeta}_{H^N}) \times \epsilon \norm{\zeta}_{\X^{|\alpha|}} \norm{w}_{\Z^1} \norm{f}_{\Y^0}.\]
\item if $|\beta_4|=|\alpha|$, then $0\leq |\beta_1|,|\beta_2|,|\beta_3|\leq 1$ and
\[\left\vert\gamma^{2-i} \mu \langle r_{3,i}^{(\beta_j)},f \rangle_{(\Y^0)^{\star}}\right\vert \leq \gamma^{2-i}\mu \norm{\partial^{\beta_2} h_i^3}_{L^\infty} \norm{\partial_x \F_i^\mu\{(\partial^{\beta_3} h_i^{-1})(\partial^{\beta_4} w)\}}_{L^2} \norm{\partial_x \F_i^\mu\big\{ (\partial^{\beta_1} h_i^{-1})f\big\} }_{L^2} .\]
Reasoning as above and since $|\beta_1|+|\beta_2|+|\beta_3|=1$, one obtains
\[\left\vert \gamma^{2-i} \mu \langle r_{3,i}^{(\beta_j)},f \rangle_{(\Y^0)^{\star}}\right\vert \leq C(\m,h_0^{-1},\norm{\epsilon \zeta}_{H^3}) \times \epsilon\norm{\zeta}_{\Z^1} \norm{w}_{\Y^{|\alpha|}} \norm{f}_{\Y^0}. \]
\item if $|\beta_3|=|\alpha|$, one obtains as above
\[\left\vert \gamma^{2-i}\mu\langle r_{3,i}^{(\beta_j)},f \rangle_{(\Y^0)^{\star}}\right\vert \leq C(\m,h_0^{-1},\norm{\epsilon\zeta}_{H^3},\norm{\epsilon\zeta}_{W^{|\alpha|-1,\infty}}) \times \epsilon \norm{w}_{\Z^1} \norm{\zeta}_{\Y^{|\alpha|}} \norm{f}_{\Y^0}. \]
By Lemma~\ref{L.embedding} and~\eqref{def-Upsilon}, it follows
\[\left\vert \gamma^{2-i}\mu\langle r_{3,i}^{(\beta_j)},f \rangle_{(\Y^0)^{\star}}\right\vert \leq C(\m,h_0^{-1},\norm{\epsilon\zeta}_{H^N}) \times \Upsilon_\F^{1/2} \norm{w}_{\Z^1} \norm{\zeta}_{\X^{|\alpha|}} \norm{f}_{\Y^0}. \]
\item otherwise, one has $0\leq |\beta_1|,|\beta_2|,|\beta_3|,|\beta_4|\leq |\alpha|-1$. 

If $|\beta_1|\leq |\alpha|-2$, we integrate by parts and estimate
\[\left\vert\gamma^{2-i} \mu\langle r_{3,i}^{(\beta_j)},f \rangle_{(\Y^0)^{\star}}\right\vert \leq \norm{\partial^{\beta_2} h_i^3}_{L^\infty} \norm{\partial_x \F_i^\mu\{(\partial^{\beta_3} h_i^{-1})(\partial^{\beta_4} w)\}}_{\Y^0} \norm{\partial_x \F_i^\mu\big\{ (\partial^{\beta_1} h_i^{-1})f\big\} }_{\Y^0} .\]
If $|\beta_1|= |\alpha|-1$, then we estimate
\[\left\vert \gamma^{2-i} \mu\left(r_{3,i}^{(\beta_j)},f\right)_{L^2}\right\vert \leq \mu \norm{ \partial^{\beta_1} h_i^{-1}}_{L^2} \norm{\partial_x \F_i^\mu\big\{(\partial^{\beta_2} h_i^3) \partial_x \F_i^\mu\{(\partial^{\beta_3} h_i^{-1})(\partial^{\beta_4} w)\}\big\}}_{L^\infty} \norm{f }_{L^2} .\]
In both cases, we find
 \[\left\vert \gamma^{2-i} \mu\langle r_{3,i}^{(\beta_j)},f \rangle_{(\Y^0)^{\star}}\right\vert \leq C(\m,h_0^{-1},\norm{\epsilon\zeta}_{H^{N}},\norm{w}_{\Y^{N}}) \times \epsilon \norm{\zeta}_{\X^{|\alpha|}}\norm{f}_{\Y^0}.\]
\end{itemize}

Plugging these estimates into~\eqref{r3i-def}, we proved
\begin{equation}\label{r3-def}
\mu\partial^\alpha \partial_t\big(\Q^\F[\epsilon\zeta]w\big)= \mu(\Q_2^\F[h_2]+\gamma \Q_1^\F[h_1])\partial^\alpha \partial_t w-\mu\big(\dd \Q_2^\F[h_2](w)-\gamma \dd \Q_1^\F[h_1](w)\big)(\epsilon \partial^\alpha \partial_x w ) + r_{3}^{(\alpha)},
\end{equation}
with
\begin{equation}\label{r3-est} 
 \norm{ r_{3}^{(\alpha)} }_{(\Y^0)^{\star}}\leq C(\m,h_0^{-1},E^{N}(U))\times E^{|\alpha|}(U)^{1/2}\times (\epsilon + \Upsilon_\F^{1/2} \norm{w}_{\Z^1}) .
 \end{equation}
\medskip

The other contribution is treated similarly. We define $\R^\F[\epsilon\zeta,w]=\R_2^\F[h_2,w]-\gamma \R_1^\F[h_1,w]$ with
\begin{equation}\label{def-Ri}
 \R_i^\F[h_i, w]\eqdef \frac13 w h_i^{-2}\partial_x \F_i^\mu\big\{h_i^3\partial_x \F_i^\mu \{h_i^{-1}w\}\big\}+\frac12 \big(h_i\partial_x \F_i^\mu\{h_i^{-1}w\}\big)^2,
 \end{equation}
\begin{align}\label{def-d1Ri}
\dd_1\R_i^\F[h_i,w]\bullet &\eqdef -\frac23 w h_i^{-3} \partial_x \F_i^\mu\big\{h_i^3\partial_x \F_i^\mu (h_i^{-1}w)\big\} \times \bullet\\
&\quad + w h_i^{-2}\partial_x \F_i^\mu\big\{h_i^2 \partial_x \F_i^\mu \{h_i^{-1}w\}\times\bullet\big\}
-\frac13 w h_i^{-2}\partial_x \F_i^\mu\big\{h_i^3\partial_x \F_i^\mu \{h_i^{-2}w \bullet \}\big\}\nn\\
&\quad +\big(h_i\partial_x \F_i^\mu\{h_i^{-1}w\}\big)\times\Big(\big(\partial_x \F_i^\mu\{h_i^{-1}w\}\big)\times\bullet-\big(h_i\partial_x \F_i^\mu\{h_i^{-2}w \times \bullet\}\big)\Big),\nn
\end{align}
\begin{multline}\label{def-d2Ri}
\dd_2\R_i^\F[h_i,w]\bullet \eqdef \frac13 \Big( h_i^{-2}\partial_x \F_i^\mu\big\{h_i^3\partial_x \F_i^\mu \{h_i^{-1}w\}\big\}\Big)\times\bullet+\frac13 w h_i^{-2}\partial_x \F_i^\mu\big\{h_i^3\partial_x \F_i^\mu \{h_i^{-1}\times\bullet \}\big\}\\
+\big(h_i\partial_x \F_i^\mu\{h_i^{-1}w\}\big)\times \big(h_i\partial_x \F_i^\mu\{h_i^{-1}\times \bullet \}\big).
\end{multline}
It follows
\begin{multline}\label{r4-def}
\mu\epsilon\partial^\alpha \partial_x\big(\R^\F[\epsilon\zeta,w]\big)=  \mu\epsilon^2\partial_x \Big(\big(\dd_1\R_2^\F[h_2,w]+\gamma\dd_1\R_1^\F[h_1,w]\big)\partial^\alpha \zeta\Big)\\
+\mu\epsilon\partial_x \Big(\big(\dd_2\R_2^\F[h_2,w]-\gamma\dd_2\R_1^\F[h_1,w]\big)\partial^\alpha w\Big)+ r_4^{(\alpha)},
\end{multline}
where $ r_4^{(\alpha)}$ may be estimated similarly as $r_3^{(\alpha)}$ above:
\begin{equation}\label{r4-est} 
 \norm{ r_4^{(\alpha)} }_{(\Y^0)^{\star}}\leq C(\m,h_0^{-1},E^{N}(U)) \times E^{|\alpha|}(U)^{1/2}\times (\epsilon + \Upsilon_\F \norm{w}_{\Z^1}^2).
 \end{equation}
\newpage

\noindent{\em Contribution from the surface tension term,} $\frac{\gamma+\delta}{\Bo}\partial_x ^2\Big(\frac1{\sqrt{1+\mu\epsilon^2|\partial_x\zeta|^2}}\partial_x\zeta\Big)$.

Let us denote $s(\partial_x\zeta)=\frac1{\sqrt{1+\mu\epsilon^2|\partial_x\zeta|^2}}\partial_x\zeta$ and notice $\partial s= \frac1{(1+\mu\epsilon^2|\partial_x\zeta|^2)^{3/2}} \partial\partial_x\zeta$. It follows
\begin{equation}\label{r5-def}
\partial^\alpha \partial_x^2 s
= \partial_x^2\left(\frac1{(1+\mu\epsilon^2|\partial_x\zeta|^2)^{3/2}} \partial^\alpha \partial_x\zeta\right)
-\sum_{j=1}^2 3\alpha_j \partial_x^2\left(\frac{\mu\epsilon^2(\partial_x\zeta)(\partial^{\e_j}\partial_x\zeta)}{(1+\mu\epsilon^2|\partial_x\zeta|^2)^{5/2}} \partial^{\alpha-\e_j} \partial_x\zeta\right) + r_5^{(\alpha)},
\end{equation}
with 
\[ r_5^{(\alpha)}=\sum_{k=1}^{N+1} \frac{(\mu\epsilon^2)^k}{(1+\mu\epsilon^2|\partial_x\zeta|^2)^{k+3/2}}\sum_{(\beta_j)} C^{(\beta_j)} r_k^{(\beta_j)}, \qquad r_k^{(\beta_j)}\eqdef \prod_{j=1}^{2k+1} \partial^{\beta_j} \partial_x \zeta,\]
where for any $k\in \{1,\dots,N+1\}$, $(\beta_j)$ is a $2k+1$-uple such that for all $ j\in\{1,\dots,2k+1\}$, one has
\[ 0\leq |\beta_1|\leq \dots\leq |\beta_{2k+1}|\leq |\alpha|\leq N\quad \text{ and } \quad \sum_{j=1}^{2k+1} \beta_j=\alpha +(0,2),\]
and $C_{(\beta_j)}$ is a constant.

Assume first that $|\beta_{2k+1}|=N$. Then for any $j\in\{1,\dots,k\}$, $|\beta_j|\leq 2$. It follows
\[ \frac1{\Bo}\norm{r_k^{(\beta_j)}}_{L^2} \leq \frac1{\Bo^{1/2}} C(\norm{\partial_x\zeta}_{W^{2,\infty}}) \norm{\partial^{\beta_{2k+1}}\zeta}_{\X^0}\leq \frac1{\Bo^{1/2}} C(\norm{\zeta}_{H^4}) \norm{\zeta}_{\X^N} .\]

Now, if $|\beta_{2k+1}|=N-1$, then either $|\beta_{2k}|=3$ and $|\beta_j|=0$ for any $j\leq 2k-1$, or $|\beta_j|\leq 2$ for any $j\leq 2k$. The latter case is estimated as above, while in the former case, one has
\[ \frac1{\Bo}\norm{r_k^{(\beta_j)}}_{L^2} \leq \frac1{\Bo^{1/2}} C(\norm{\partial_x\zeta}_{W^{2,\infty}}) \norm{\partial^{\beta_{2k}} \zeta}_{\W^0} \norm{\partial^{\beta_{2k+1}} \partial_x \zeta}_{L^2} \leq \frac1{\Bo^{1/2}} C(\norm{\zeta}_{\X^4}) \norm{\zeta}_{\X^N}.\]

Otherwise, one has $|\beta_{j}|\leq N-2$ for any $j\in\{1,\dots,2k+1\}$, and in that case,
\[ \frac1{\Bo}\norm{r_5^{(\beta_j)}}_{L^2} \leq \frac1{\Bo} C(\norm{\zeta}_{\X^N})\norm{\zeta}_{\X^N} .\]

Altogether, this yields for $N\geq 4$
\begin{equation}\label{r5-est}
 \frac1\Bo\norm{r_5^{(\alpha)}}_{L^2}\leq \frac{\mu\epsilon^2}{\Bo^{1/2}} C(\Bo^{-1},\mu\epsilon^2,\norm{\zeta}_{\X^N}) \norm{\zeta}_{\X^N} .
 \end{equation}
\medskip

Finally, we notice that 
\[ \partial_x\check{\mathfrak{a}}_\alpha \zeta^{\langle\check\alpha\rangle}\eqdef \frac{\gamma+\delta}{\Bo}\sum_{j=1}^2 3\alpha_j \partial_x^2\left(\frac{\mu\epsilon^2(\partial_x\zeta)(\partial^{\e_j}\partial_x\zeta)}{(1+\mu\epsilon^2|\partial_x\zeta|^2)^{5/2}} \partial^{\alpha-\e_j} \partial_x\zeta\right) 
\]
may be estimated, when $1\leq|\alpha|\leq N-1$, as
\begin{equation}\label{checka-est}
\norm{\partial_x\check{\mathfrak{a}}_\alpha \zeta^{\langle\check\alpha\rangle}}_{L^2} \leq \mu\epsilon^2 C(\Bo^{-1},\mu\epsilon^2,\norm{\zeta}_{\X^N}) \norm{\zeta}_{\X^{|\alpha|}} \qquad ( 1\leq|\alpha|\leq N-1).
\end{equation}

The definition of the operators $\mathfrak{a},\mathfrak{b},\mathfrak{c},\check{\mathfrak{a}}_\alpha^{(\alpha)},r^{(\alpha)}$ and estimate~\eqref{r-est} follows from~\eqref{r1-def}-\eqref{r1-est}, \eqref{r2-def}-\eqref{r2-est}, \eqref{r3-def}-\eqref{r3-est}, \eqref{r4-def}-\eqref{r4-est}, \eqref{r5-def}-\eqref{r5-est} as well as~\eqref{checka-est} when $1\leq|\alpha|\leq N-1$.

Estimate~\eqref{rdiff-est} is obtained identically, using in particular the trivial estimates
\[ \norm{H(\epsilon\zeta_1)-H(\epsilon\zeta_2)}_{L^\infty} \leq H_{1,h_0} \epsilon\norm{\zeta_1-\zeta_2}_{L^\infty}, \quad \norm{H(\epsilon\zeta_1)-H(\epsilon\zeta_2)}_{L^2} \leq H_{1,h_0} \epsilon\norm{\zeta_1-\zeta_2}_{L^2},\]
and
\[ \norm{\widehat{H(\epsilon\zeta_1)}-\widehat{H(\epsilon\zeta_2)}}_{L^1} \leq \norm{H(\epsilon\zeta_1)-H(\epsilon\zeta_2)}_{H^1}\leq C(H_{2,h_0},\norm{\epsilon\zeta_1}_{W^{1,\infty}}, \norm{\epsilon\zeta_2}_{W^{1,\infty}}) \epsilon\norm{\zeta_1-\zeta_2}_{H^1},\]
where $H$ and $ H_{1,h_0} $ are as in Lemma~\ref{L.f/h}. This concludes the proof of Lemma~\ref{L.linearize}.
\end{proof}

\subsection{Preliminary results}\label{S.preliminary}
In this section, we prove that the operator $\mathfrak a[\epsilon\zeta,\epsilon w]$ (resp. $\mathfrak b[\epsilon\zeta]$), introduced in Lemma~\ref{L.linearize}, is symmetric, continuous and coercive with respect to the space $\X^0$ (resp. $\Y^0$), provided that some conditions are satisfied by $(\epsilon\zeta,\epsilon w)$. These requirements can be seen as sufficient conditions for the hyperbolicity of the system, and permit to control the energy solutions to the quasilinear system for positive times (Section~\ref{S.energy}), and eventually prove the well-posedness of our system (Section~\ref{S.WP-conclusion}).

 \begin{Lemma}\label{L.hyperbolicity} Let $(\zeta, w)^\top\in H_x^{{3}}\times \Z^1$ be such that $\epsilon\zeta$ satisfies~\eqref{C.depth} with $h_0>0$. Then one has ${\mathfrak{a}[\epsilon\zeta,\epsilon w] \in \mathcal L(\X^0;(\X^0)^\star)}$, $\mathfrak{b}[\epsilon\zeta]\in \mathcal L(\Y^0;(\Y^0)^\star)$ and $\mathfrak{c}[\epsilon\zeta,\epsilon w]\in \mathcal L(\Y^0;(\Y^0)^\star)$. 
 Moreover, there exists $K_0,K_1=C(\m ,h_0^{-1},\epsilon \norm{\zeta}_{H_x^{{3}}})$ such that
 \begin{align*}
 &\forall f,g \in \X^0, \qquad &\norm{\langle \mathfrak{a}[\epsilon\zeta,\epsilon w] f\ , \ g\rangle_{(\X^0)^{\star}}}&\leq K_1 (1+\Upsilon_\F \norm{w}_{\Z^1}^2 ) \norm{f}_{\X^0}\norm{g}_{\X^0}\\
 &\forall f,g \in \Y^0, \qquad &\norm{\langle \mathfrak{b}[\epsilon\zeta] f\ , \ g\rangle_{(\Y^0)^{\star}}}&\leq K_1\norm{f}_{\Y^0}\norm{g}_{\Y^0}\\
 &\forall f,g \in \Y^0, \qquad & \norm{\langle \mathfrak{c}[\epsilon\zeta,\epsilon w] f\ , \ g\rangle_{(\Y^0)^{\star}}}&\leq \epsilon K_1 \norm{w}_{\Z^1} \norm{f}_{\Y^0}\norm{g}_{\Y^0},\\
 &\forall f \in \Y^0, \qquad &\langle \mathfrak{b}[\epsilon\zeta] f\ , \ f\rangle_{(\Y^0)^{\star}}&\geq \frac1{K_0}\norm{f}_{\Y^0}^2.
 \end{align*}
 
 Assume additionally that there exists $k_0>0$ such that
 \begin{equation}\label{C.hyp0}
 (\gamma+\delta)-\epsilon^2\max_{x\in\RR}\left\{(h_2^{-3} +\gamma h_1^{-3}) |w|^2\right\}\geq k_0>0.
 \end{equation} 
 Then there exists $K,K_0'=C(\m ,h_0^{-1},k_0^{-1},\epsilon \norm{\zeta}_{H_x^{{3}}})$ such that if
 \begin{equation}\label{C.hyperbolicity}
\Upsilon_\F \norm{w}_{\Z^1}^2 \leq K^{-1} , 
 \end{equation}
 then 
\[
 \forall f \in \X^0, \qquad \langle \mathfrak{a}[\epsilon\zeta,\epsilon w] f\ , \ f\rangle_{(\X^0)^{\star}}\geq \frac1{K_0'} \norm{f}_{\X^0}^2.
\]
 \end{Lemma}
 \begin{proof}
 We establish each result for $f,g\in \S(\RR)$ so that all the terms are obviously well-defined and in particular the $(X^{\star}-X)$ duality product (with $X=\X^0$ or $\Y^0$) coincides with the $L^2$ scalar product; the result for $f,g\in \X^0$ or $ \Y^0$ is then obtained by density of $\S(\RR)$ in $\X^0$ and $\Y^0$, and continuous linear extension.
 
 One has, after integration by parts,
 \begin{multline} \label{b-quad} 
 \langle \mathfrak{b}[\epsilon\zeta] f\ , \ g\rangle_{(\Y^0)^{\star}} \ =\ \big( \mathfrak{b}[\epsilon\zeta]f , g \big)_{L^2} \\
 =\ \int_\RR \frac{h_1+\gamma h_2}{h_1 h_2}f g +\frac\mu3 h_2^3 (\partial_x \F_2\{h_2^{-1}f\} )(\partial_x \F_2\{h_2^{-1}g\})+\frac{\mu\gamma}{3} h_1^3 (\partial_x \F_1\{h_1^{-1} f\})(\partial_x \F_1\{h_1^{-1} g\} )\ \dd x.
\end{multline}
 It follows easily 
 \[ \norm{\langle \mathfrak{b}[\epsilon\zeta] f\ , \ g\rangle_{(\Y^0)^{\star}}}\leq K_1\norm{f}_{\Y^0}\norm{g}_{\Y^0}.\]

We write again for the coercivity inequality, 
 \[ \langle \mathfrak{b}[\epsilon\zeta] f\ , \ f\rangle_{(\Y^0)^{\star}} = \int_\RR \frac{h_1+\gamma h_2}{h_1 h_2}|f|^2 +\frac\mu3 h_2^3 |\partial_x \F_2\{h_2^{-1}f\} |^2+\frac{\mu\gamma}{3} h_1^3 |\partial_x \F_1\{h_1^{-1} f\} |^2\ \dd x.\]
 It follows immediately, since $\epsilon\zeta$ satisfies~\eqref{C.depth},
 \[ \langle \mathfrak{b}[\epsilon\zeta] f\ , \ f\rangle_{(\Y^0)^{\star}}\geq \frac{1+\gamma}{1+\delta^{-1}} \norm{f}_{L^2}^2 + \frac{\mu h_0^3}{3 }\norm{\partial_x \F_2 \{h_2^{-1}f\} }_{L^2}^2 + \frac{\mu\gamma h_0^3 }{3}\norm{\partial_x \F_1\{h_1^{-1}f\} }_{L^2}^2 .\]
 Now, by Lemma~\ref{L.produit}, one has
\[\norm{\partial_x \F_1f }_{L^2} = \norm{\partial_x \F_1 \{(1+\epsilon\zeta) h_1^{-1} f\} }_{L^2}
 \leq (1+\epsilon\norm{\zeta}_{\Z^0})\norm{\partial_x \F_1 \{h_1^{-1} f\} }_{L^2} ,
\]
 and similarly for $\norm{\partial_x \F_2 \{h_2^{-1}f\} }_{L^2}^2$. We conclude
 \[ \norm{f}_{\Y^0}^2 \ \leq \ C(\m,h_0^{-1},\epsilon\norm{\zeta}_{\Z^0}) \times \langle \mathfrak{b}[\epsilon\zeta] f\ , \ f\rangle_{(\Y^0)^{\star}}.\]
 
 By similar argumentation, one easily shows that the operator $\mathfrak{c}[\epsilon\zeta,\epsilon w]$ is well-defined and continuous from $\Y^0$ to $(\Y^0)^{\star}$, and satisfies the third estimate of the statement.
 \medskip
 
 We show now the coercivity of $\mathfrak{a}[\epsilon\zeta,\epsilon w]$ under additional assumption~\eqref{C.hyp0}. We write
 \begin{multline*}
 \langle \mathfrak{a}[\epsilon\zeta,\epsilon w] f\ , \ f\rangle_{(\X^0)^{\star}}\ =\ \big( \mathfrak{a}[\epsilon\zeta,\epsilon w] f , f\big)_{L^2} \\
 = \int_\RR \Big((\gamma+\delta)-\epsilon^2\dfrac{h_1^3 +\gamma h_2^3 }{(h_1 h_2)^3} |w|^2 \Big)|f|^2 + \frac{\gamma+\delta}\Bo \frac{|\partial_x f|^2}{(1+\mu\epsilon|\partial_x\zeta|^2)^{3/2}} \dd x+\mu\epsilon^2 (R_2-\gamma R_1),
 \end{multline*}
 with
 \begin{multline*}R_i = \left(\Big(h_i\big(\partial_x \F_i^\mu\{h_i^{-1}w\}\big)^2 -\frac23 w h_i^{-3} \partial_x \F_i^\mu\big\{h_i^3\partial_x \F_i^\mu \{h_i^{-1}w\}\big\} \Big)\times f,f \right)_{L^2}\\
 +\frac13 \Big( h_i^3\partial_x \F_i^\mu \{h_i^{-2}w f \},\partial_x \F_i^\mu \{h_i^{-2}w f \}\Big)_{L^2}
 -2\Big( \partial_x \F_i^\mu\big\{(h_i^{-2}w) \times f\big\} \ , \ \big(h_i^2\partial_x \F_i^\mu\{h_i^{-1}w\}\big)\times f\Big)_{L^2}.
 \end{multline*}
 Using Cauchy-Schwarz inequality and Lemmata~\ref{L.produit}, and~\ref{L.f/h}, one has the following estimate
 \begin{align*}\mu\epsilon^2 |R_2-\gamma R_1|&\leq \epsilon^2 \norm{w}_{\Z^1}^2 C(\m ,h_0^{-1},\epsilon \norm{\zeta}_{H_x^{{3}}}) \norm{f}_{\Y^0}^2\\
 &\leq \Upsilon_\F\norm{w}_{\Z^1}^2 C(\m ,h_0^{-1},\epsilon \norm{\zeta}_{H_x^{{3}}}) \norm{f}_{\X^0}^2,
 \end{align*}
 where the last identity follows from Lemma~\ref{L.embedding}.
 
From~\eqref{C.hyp0}, one has immediately
 \[ \langle \mathfrak{a}[\epsilon\zeta,\epsilon w] f\ , \ f\rangle_{(\X^0)^{\star}} \ - \ \mu\epsilon^2 (R_2-\gamma R_1) \geq \min\left\{k_0, \frac{\gamma+\delta}{(1+\mu\epsilon\norm{\partial_x\zeta}_{L^\infty}^2)^{3/2}} \right\} \ \times\ \norm{f}_{\X^0}^2 .\]
 The existence of $K_0',K$ such that~\eqref{C.hyperbolicity} implies 
 \[ \norm{f}_{\X^0}^2 \leq K_0' \langle \mathfrak{a}[\epsilon\zeta,\epsilon w] f\ , \ f\rangle_{(\X^0)^{\star}} \]
 is now straightforward.
 
One shows similarly that $\mathfrak{a}[\epsilon\zeta,\epsilon w]:\X^0\to (\X^0)^{\star}$ is well-defined and continuous, and satisfies the first estimate of the statement. This concludes the proof of Lemma~\ref{L.hyperbolicity}.
 \end{proof}
 
 The following Lemma is a direct consequence of Lemma~\ref{L.hyperbolicity}.
 \begin{Lemma}\label{L.invertible}
 Let $(\zeta, w)^\top\in H_x^{{3}}\times \Z^1$ be such that $\epsilon\zeta$ satisfies~\eqref{C.depth}. Then $\mathfrak{b}[\epsilon\zeta]:\Y^0\to (\Y^0)^{\star}$ is a topological isomorphism with: 
 \[ \forall f\in \Y^0, \qquad \norm{(\mathfrak{b}[\epsilon\zeta])^{-1} f}_{\Y^0} \ \leq \ K_0 \norm{f}_{(\Y^0)^{\star}},\]
 with $K_0$ as in Lemma~\ref{L.hyperbolicity}.
 
 If, additionally, $(\epsilon\zeta, \epsilon w)$ satisfies~\eqref{C.hyp0}-\eqref{C.hyperbolicity}, then $\mathfrak{a}[\epsilon\zeta,\epsilon w]:\X^0\to (\X^0)^{\star}$ is a topological isomorphism with:
 \[ \forall f\in \X^0, \qquad \norm{(\mathfrak{a}[\epsilon\zeta,\epsilon w])^{-1} f}_{\X^0} \ \leq \ K_0' \norm{f}_{(\X^0)^{\star}},\]
 with $K_0'$ as in Lemma~\ref{L.hyperbolicity}.
 \end{Lemma}
 \begin{proof} By Lemma~\ref{L.hyperbolicity}, $\mathfrak{b}[\epsilon\zeta]:\Y^0\to (\Y^0)^{\star}$ is well-defined, continuous and coercive. We then deduce by the operator version of Lax-Milgram theorem that $\mathfrak{b}[\epsilon\zeta]$ is an isomorphism from $\Y^0$ onto $(\Y^0)^{\star}.$ The continuity of the inverse follows from the continuity and coercivity of $\mathfrak{b}[\epsilon\zeta]$: 
 \[ \norm{\mathfrak{b}[\epsilon\zeta]^{-1}f}_{\Y^0}^2 \leq K_0 \langle \mathfrak{b}[\epsilon\zeta] \mathfrak{b}[\epsilon\zeta]^{-1}f\ , \ \mathfrak{b}[\epsilon\zeta]^{-1}f\rangle_{(\Y^0)^{\star}}\leq K_0\norm{f}_{(\Y^0)^{\star}} \norm{\mathfrak{b}[\epsilon\zeta]^{-1}f}_{\Y^0}.\]
 The whole discussion is identical for $\mathfrak{a}[\epsilon\zeta,\epsilon w]$, replacing $\Y^0$ with $\X^0$, and $K_0$ with $K_0'$.
 \end{proof}
 
 We conclude this section with the following result.
 \begin{Lemma}\label{L.symmetric}
 Let $(\zeta, w)^\top\in H_x^{{3}}\times \Z^1$ be such that $\epsilon\zeta$ satisfies~\eqref{C.depth}. Then the operator $\mathfrak{a}[\epsilon\zeta,\epsilon w]:\X^0\to(\X^0)^{\star}$ is symmetric: 
 \[ \forall f,g\in \X^0, \qquad \langle \mathfrak{a}[\epsilon\zeta,\epsilon w] f\ , \ g\rangle_{(\X^0)^{\star}} \ = \ \langle \mathfrak{a}[\epsilon\zeta,\epsilon w] g\ , \ f\rangle_{(\X^0)^{\star}}.\]
The same result holds true for $\mathfrak b[\epsilon\zeta]$ and $\mathfrak c[\epsilon\zeta,\epsilon w]$, replacing $\X^0$ with $\Y^0$.
 \end{Lemma}
 \begin{proof}The symmetry property for $\mathfrak b[\epsilon\zeta]$ is straightforwardly seen from~\eqref{b-quad}. The other operators require a slight rewriting. In particular, notice
 \begin{multline*}\dd_1\R_i^\F[h_i,w]\bullet = \Big(h_i\big(\partial_x \F_i^\mu\{h_i^{-1}w\}\big)^2 -\frac23 w h_i^{-3} \partial_x \F_i^\mu\big\{h_i^3\partial_x \F_i^\mu \{h_i^{-1}w\}\big\}\Big) \times \bullet \\
 + (h_i^{-2}w)\times\partial_x \F_i^\mu\big\{ \big(h_i^2 \partial_x \F_i^\mu \{h_i^{-1}w\}\big)\times\bullet\big\}-\big(h_i^2\partial_x \F_i^\mu\{h_i^{-1}w\}\big)\times \partial_x \F_i^\mu\big\{(h_i^{-2}w) \times \bullet\big\}\\
 -\frac13 (h_i^{-2} w) \partial_x \F_i^\mu\big\{h_i^3\partial_x \F_i^\mu \{(h_i^{-2}w) \bullet \}\big\}
 \end{multline*}
 and
 \begin{multline*} \big(\dd \Q_i^\F[h_i](w)+ \dd_2\R_i^\F[h_i,w]\big)\bullet =\Big(\frac23 h_i^{-2}\partial_x \F_i^\mu\big\{h_i^3 \partial_x \F_i^\mu\{h_i^{-1}w\}\big\}\Big)\times \bullet \\-h_i^{-1}\partial_x \F_i^\mu\big\{\big(h_i^2 \partial_x \F_i^\mu\{h_i^{-1}w\}\big) \times \bullet\big\} 
 +\big(h_i^2 \partial_x \F_i^\mu\{h_i^{-1}w\}\big) \times \partial_x \F_i^\mu\{ h_i^{-1}\times \bullet \} \\
 +\frac13 h_i^{-1}\times \partial_x \F_i^\mu\big\{h_i^3 \partial_x \F_i^\mu\{(h_i^{-2}w)\times\bullet \}\big\}
 +\frac13 (h_i^{-2}w)\times \partial_x \F_i^\mu\big\{h_i^3\partial_x \F_i^\mu \{h_i^{-1}\times\bullet \}\big\}
 \end{multline*}
 are obviously symmetric, since $\partial_x \F_i^\mu$ is skew-symmetric. The result is now clear.
 \end{proof}

\subsection{A priori estimates}\label{S.energy}
We now consider the quasi-linearized system arising from Lemma~\ref{L.linearize}:
\begin{equation}\label{linearized}
\left\{ \begin{array}{l}
\displaystyle\partial_{ t}\dot{\zeta} \ + \ \partial_x \dot w \ =\ r^1, \\ \\ 
\mathfrak{b} \partial_{ t} \dot w \ + \ \partial_x \mathfrak{a} \dot{\zeta} \ + \ \partial_x \check{\mathfrak{a}}_\alpha \check{\zeta} \ + \ \mathfrak{c} \partial_x \dot w\ = \ r^2,
\end{array}
\right.
\end{equation}
where we denote for conciseness $\mathfrak{a}=\mathfrak{a}[\epsilon\zeta,\epsilon w]$ (and similarly for $\check{\mathfrak{a}}_\alpha,\mathfrak b,\mathfrak c$), as defined in Lemma~\ref{L.linearize}, and $r^1,r^2$ are remainder terms to be precised. More accurately, we introduce a regularized version of~\eqref{linearized}. Denote $ {\sf J}_\nu =(1-\nu\partial_x^2)^{-1/2}$ and consider
\begin{equation}\label{linearized-mollified}
\left\{ \begin{array}{l}
\displaystyle\partial_{ t}\dot{\zeta} \ + \ {\sf J}_\nu^2 \partial_x \dot w \ =\ r^1, \\ \\ 
\mathfrak{b}\partial_{ t} \dot w \ + \ {\sf J}_\nu^2\partial_x \mathfrak{a} \dot{\zeta} \ + \ {\sf J}_\nu^2\partial_x \check{\mathfrak{a}}_{\alpha} \check{\zeta} \ + \ {\sf J}_\nu \mathfrak{c} {\sf J}_\nu \partial_x \dot w \ = \ r^2.
\end{array}
\right.
\end{equation}
We obtain below a uniform a priori control of the energy of any solution, and then, in Lemma~\ref{L.energy-estimate-diff}, a similar estimate on the difference between two solutions.
\begin{Lemma}\label{L.energy-estimate} Let ${\dot U\eqdef (\dot\zeta,\dot w)^\top,\ U\eqdef (\zeta, w)^\top\in L^\infty([0,T];\X^4\times\Y^4)}$, $\check\zeta\in L^\infty([0,T];(\X^1)^2)$ and $r=(r^1,r^2)^\top\in L^1([0,T);\X^0\times(\Y^0)^\star)$ satisfying~\eqref{linearized-mollified} with $\nu\in[0,1]$.
 Assume moreover that $U(t)$ satisfies\eqref{C.depth},\eqref{C.hyp0} and~\eqref{C.hyperbolicity} with $h_0^{-1},k_0^{-1},K^{-1}$ uniformly for $t\in[0,T]$. Then one has
\[ E^0(\dot U)^{1/2} \ \leq \ {\bf C}_0 \big(E^0(\dot U\id{t=0})^{1/2} +(\mu\epsilon^2)\Norm{\check\zeta}_{L^\infty([0,T];(\X^0)^2)} \big) e^{\lambda t} \ + \ {\bf C}_0\int_0^t f(t') e^{\lambda (t-t')}\dd t', \]
with 
\[ \lambda \ = \ {\bf C}_0 \times \Big(\epsilon+\Upsilon_\F \Norm{w}_{L^\infty([0,T];\Z^2)}^2 \Big) ,\ \quad f(t)=\norm{r}_{\X^0\times(\Y^0)^\star}+\mu\epsilon^2 \norm{\check\zeta}_{(\X^1)^2} \]
and ${\bf C}_0=C(\m,h_0^{-1},k_0^{-1},K,\Norm{U}_{L^\infty([0,T];\X^4\times\Y^4)})$.
\end{Lemma}
\begin{Remark}
The energy estimate is uniform with respect to $\nu \in[0,1]$, and holds in particular for 
solutions to the non-regularized system~\eqref{linearized}. 
\end{Remark}
\begin{proof}Since $U,\dot U\in L^\infty([0,T];\X^4\times\Y^4)$, all the components of equation~\eqref{linearized-mollified} are obviously well-defined in $L^2$.
We compute the $L^2$ inner product of the first equation with $\mathfrak{a}\dot\zeta+\check{\mathfrak{a}}_{\alpha}\check\zeta$, and add the $L^2$ inner product of the second equation with $\dot w$. Recalling that $\mathfrak{a},\mathfrak{b},\mathfrak{c}$ are symmetric (by Lemma~\ref{L.symmetric}), and since ${\sf J}_\nu$ is symmetric and $\partial_x$ is skew-symmetric, we obtain
after straightforward manipulations
\begin{multline}\label{energy-to-estimate}
\frac{\dd}{\dd t} \left(\frac12 \big(\mathfrak{a}\dot \zeta,\dot \zeta\big)_{L^2} + \big(\dot{\zeta} , \check{\mathfrak{a}}_{\alpha}\check\zeta\big)_{L^2} +\frac12\big(\mathfrak{b} \dot w,\dot w\big)_{L^2} \right) \ = \\ 
\frac12 \big(\big[\partial_t,\mathfrak{a}\big]\dot \zeta,\dot \zeta\big)_{L^2}
+\big(\dot{\zeta} , \partial_t(\check{\mathfrak{a}}_{\alpha}\check\zeta)\big)_{L^2} 
+\frac12 \big(\big[\partial_t,\mathfrak{b}\big]\dot w,\dot w\big)_{L^2}
+\frac12\big(\big[\partial_x,\mathfrak{c} \big]{\sf J}_\nu\dot w,{\sf J}_\nu\dot w\big)_{L^2}\\
+\big( r^1,\mathfrak{a}\dot\zeta+\check{\mathfrak{a}}_{\alpha}\check\zeta\big)_{L^2}+\big( r^2,\dot w\big)_{L^2} .
\end{multline}
We estimate below each of the components of the right-hand-side. These estimates follow from the product estimates of Section~\ref{S.product}, as in the proof of Lemma~\ref{L.linearize}. For the sake of conciseness, we do not detail all calculations but rather provide the precise estimates for each component.
\medskip

\noindent {\em $(I)\eqdef\big(\big[\partial_t,\mathfrak{a}\big]\dot \zeta,\dot \zeta\big)_{L^2}$.} One has, by definition,
\begin{multline*}
[\partial_t,\mathfrak{a}\big]\dot \zeta = -\epsilon^2 \dot\zeta \partial_t \big(G(\epsilon\zeta)|w|^2\big) -\mu\epsilon^2 \Big(\big[\partial_t,\dd_{1}\R_2^{\F}[h_2,w]\big]\dot{\zeta}-\gamma\big[\partial_t,\dd_{1}\R_1^{\F}[h_1,w]\big]\dot{\zeta}\Big)\\
-\frac{\gamma+\delta}{\Bo} \partial_x\Big(\partial_t\Big(\frac{1 }{(1+\mu\epsilon^2|\partial_x\zeta|^2)^{3/2}}\Big) \partial_x\dot{\zeta}\Big),
\end{multline*}
where $G(\epsilon\zeta)=\dfrac{h_1^3 +\gamma h_2^3 }{(h_1 h_2)^3}$ and $\R_i^{\F}[h_i,w]$ is defined in~\eqref{def-Ri}.

The first contribution is easily estimated: 
\[ \left\vert\big(-\epsilon^2 \dot\zeta \partial_t \big(G(\epsilon\zeta)|w|^2\big) ,\dot \zeta\big)_{L^2}\right\vert \leq C(\m,h_0^{-1},\norm{\epsilon\zeta}_{W^{1,\infty}})\times \epsilon^2 \norm{w}_{W^{1,\infty}}^2 \norm{\dot{\zeta}}_{L^2}^2,\]
The third component is estimated after one integration by parts:
\[\left\vert\left(-\frac{\gamma+\delta}{\Bo} \partial_x\Big(\partial_t\big((1+\mu\epsilon^2|\partial_x\zeta|^2)^{-3/2}\big)\partial_x\dot{\zeta}\Big),\dot \zeta\right)_{L^2}\right\vert
\leq C(\m,\norm{\zeta}_{W^{2,\infty}})\times \mu\epsilon^2 \frac1{\Bo}\norm{ \partial_x \dot{\zeta}}_{L^2}^2.\]
Treating the last contribution is more involved, as $\big[\partial_t,\dd_{1}\R_i^{\F}[h_i,w]\big]$ is the sum of many terms.
 However, all of these terms may be dealt with as in the proof of Lemma~\ref{L.linearize}: using integration by parts if necessary, one may ensure that the operator $\partial_x\F^\mu_i$ applies only once to each $\dot \zeta$ and since much regularity is assumed on $\zeta\in \X^4$, Lemmata~\ref{L.produit},~\ref{L.f/h} and then Lemma~\ref{L.embedding} yield
\[
\left\vert\mu\epsilon^2 \gamma^{2-i} \big(\big[\partial_t,\dd_{1}\R_i^{\F}[h_i,w]\big]\dot{\zeta},\dot{\zeta}\big)_{L^2} \right\vert \leq C_0\times \epsilon^2 \vert w\vert_{Z^2_{\F^{\mu}}}^2 \vert \dot{\zeta}\vert_{\Y^0}^2  \leq C_0 \Upsilon_\F \vert w\vert_{Z^2_{\F^{\mu}}}^2 \vert \dot{\zeta}\vert_{\X^0}^2,
\]
with $C_0=C(\m,h_0^{-1},\norm{\zeta}_{\X^4})$. 
Altogether, we proved
\begin{equation}\label{estI}
 |(I)|\leq C(\m,h_0^{-1},\norm{\zeta}_{\X^4})\times \left(\mu\epsilon^2+\Upsilon_\F \norm{w}_{\Z^2}^2 \right)\norm{\dot\zeta}_{\X^0}^2 .
\end{equation}
\medskip

\noindent {\em $(II)\eqdef\big(\big[\partial_t,\mathfrak{b}\big]\dot w,\dot w\big)_{L^2}$ }. One has, by definition,

\[
\big[\partial_t,\mathfrak{b}\big]\dot w= \partial_t\Big(\frac{h_1+\gamma h_2}{h_1 h_2}\Big)\dot w \ +\ \mu\epsilon \big(\dd\Q_2^\F[h_2](\dot w)-\gamma \dd\Q_1^\F[h_1](\dot w) \big)\partial_t\zeta,
\]
where $\dd\Q_i^\F$ is defined in~\eqref{def-dQi}. The first term is estimated as
 \[ \left\vert \Big(\partial_t\Big(\frac{h_1+\gamma h_2}{h_1 h_2}\Big)\dot w,\dot w\Big)_{L^2} \right\vert \leq C(\m,h_0^{-1},\norm{\partial_t\zeta}_{L^\infty})\times \epsilon  \norm{\dot w}_{L^2}^2,\]
For the second term we have after integration by parts and by triangular inequality
\begin{multline*} \left\vert \Big( \dd\Q_i^\F[h_i](\dot w) \partial_t\zeta , \dot w \Big)_{L^2} \right\vert \leq 
 \frac23 \left\vert \big( h_i^3 \partial_x \F_i^\mu\{h_i^{-1}\dot w\}\ ,\ \partial_x \F_i^\mu\{ h_i^{-2} \partial_t\zeta \dot w \}\big)_{L^2} \right\vert\\+ \left\vert \big(h_i^2 (\partial_t\zeta) \partial_x \F_i^\mu\{h_i^{-1}\dot w\} , \partial_x \F_i^\mu\{ h_i^{-1} \dot w\}\big)_{L^2}\right\vert .
\end{multline*}
By Lemmata~\ref{L.produit} and~\ref{L.f/h}, one immediately deduces
\begin{equation}\label{estII}
 |(II)|\leq  C(\m,h_0^{-1},\norm{\zeta}_{\X^4})\times \epsilon \norm{\dot w}_{\Y^0}^2 .
\end{equation}
\medskip

\noindent {\em $(III)\eqdef\big(\big[\partial_x,\mathfrak{c} \big] {\sf J}_\nu\dot w,{\sf J}_\nu\dot w\big)_{L^2}$ }. 
One may proceed similarly as above,
and one obtains without any additional difficulty 
\begin{equation}\label{estIII}
 |(III)|\leq  C(\m,h_0^{-1},\norm{\zeta}_{\X^4},\norm{w}_{\Y^4}) \times \epsilon\norm{{\sf J}_\nu \dot w}_{\Y^0}^2 \leq  C(\m,h_0^{-1},\norm{\zeta}_{\X^4},\norm{w}_{\Y^4})\times \epsilon \norm{\dot w}_{\Y^0}^2 .
\end{equation}
\medskip

\noindent {\em $(IV)\eqdef\big(\dot{\zeta} , \partial_t(\check{\mathfrak{a}}_\alpha\check\zeta)\big)_{L^2}$ }. After one integration by parts, one has
\[ (IV)=-3\mu\epsilon^2 \frac{\gamma+\delta}\Bo \sum_{j\in\{1,2\}} \alpha_j \left(\partial_x\dot\zeta , \partial_t\left(\frac{(\partial_x\partial^{\e_j}\zeta)(\partial_x\zeta)(\partial_x\check\zeta_j)}{(1+\mu\epsilon^2|\partial_x\zeta|^2)^{5/2}} \right) \right)_{L^2}.\]
Recall $\check \zeta=(\check\zeta_0,\check\zeta_1)^\top\in \X^1\times \X^1$, so we easily deduce by Cauchy-Schwarz inequality
\begin{equation}\label{estIV}
 |(IV)|\leq  C(\m,\norm{\zeta}_{W^{3,\infty}}) \times \mu\epsilon^2 \norm{\dot\zeta}_{\X^0} \norm{\check\zeta}_{(\X^1)^2} .
\end{equation}
\medskip

\noindent {\em $(V)\eqdef\big( r^1,\mathfrak{a}\dot\zeta+\check{\mathfrak{a}}_{\alpha}\check\zeta\big)_{L^2}+\big( r^2,\dot w\big)_{L^2}$ }.
By Lemma~\ref{L.hyperbolicity}, one obtains
\begin{multline}\label{estV}
 |(V)|\leq C(\m,h_0^{-1},k_0^{-1},\epsilon\norm{\zeta}_{H^{{3}}_x})\norm{r^1}_{\X^0}\norm{\dot\zeta}_{\X^0} +\mu\epsilon^2 C(\m,\norm{\zeta}_{W^{2,\infty}}) \norm{r^1}_{\X^0}\norm{\dot\zeta}_{(\X^0)^2} \\+\norm{r^2}_{(\Y^0)^\star} \norm{\dot w}_{\Y^0}.
\end{multline}
\medskip

Altogether, plugging~\eqref{estI},\eqref{estII}\eqref{estIII},\eqref{estIV},\eqref{estV} into~\eqref{energy-to-estimate} yields
\begin{equation}\label{energy-estimated}
\frac{\dd}{\dd t} \left(\frac12 \big(\mathfrak{a}\dot \zeta,\dot \zeta\big)_{L^2} + \big(\dot{\zeta} , \check{\mathfrak{a}}_\alpha\check\zeta\big)_{L^2} +\frac12\big(\mathfrak{b} \dot w,\dot w\big)_{L^2} \right) \ \leq 
 C_0 \left(\epsilon+\Upsilon_\F \norm{w}_{\Z^2}^2 \right) E^0(\dot U) + C_0 C_1 E^0(\dot U)^{1/2},
\end{equation}
with $C_0=C(\m,h_0^{-1},E^4(U))$, $C_1=\norm{r}_{\X^0\times(\Y^0)^\star}+\mu\epsilon^2 \norm{\check\zeta}_{(\X^1)^2} $, and $E^0(\dot U)=\norm{\dot\zeta}_{\X^0}^2+\norm{\dot w}_{\Y^0}^2$.

By Lemma~\ref{L.hyperbolicity}, there exists $K_0,K_1=C(\m,h_0^{-1},k_0^{-1},K,E^4(U))$ such that
\begin{equation} \label{equivalence-energie}
\frac1{K_0} E^0(\dot U)\leq \frac12 \big(\mathfrak{a}\dot \zeta,\dot \zeta\big)_{L^2} +\frac12\big(\mathfrak{b} \dot w,\dot w\big)_{L^2}\leq K_1 E^0(\dot U).
\end{equation}
Let us now estimate
\[ \left|\big(\dot{\zeta} , \check{\mathfrak{a}}_\alpha\check\zeta\big)_{L^2}\right|\leq \mu\epsilon^2 C_2 \times  \norm{\dot\zeta}_{\X^0}\norm{\check\zeta}_{\X^0}\leq \frac{1}{2}\mu\epsilon^2 C_2 \times \big( M^{-1}\norm{\dot\zeta}_{\X^0}^2+M\norm{\check\zeta}_{\X^0}^2\big),\]
with $C_2=C(\m,\norm{\partial_x\zeta}_{W^{2,\infty}})$ and arbitrary $M>0$. 
Choosing $M=\mu\epsilon^2 C_2 K_0$,~\eqref{equivalence-energie} yields
\begin{equation}\label{equivalence-energie-modifiee}
\frac1{2 K_0} E^0(\dot U)-\widetilde {\bf M} \leq \widetilde E^0(\dot U)\leq (K_1+\frac1{2K_0}) E^0(\dot U)+\widetilde {\bf M},
\end{equation}
where we denoted
\[ \widetilde {\bf M}\eqdef \max_{t\in[0,T]} \left\{\frac12 K_0(\mu\epsilon^2 C_2)^2 \norm{\check\zeta}_{\X^0}^2\right\}, \quad \widetilde E^0(\dot U)\eqdef \frac12 \big(\mathfrak{a}\dot \zeta,\dot \zeta\big)_{L^2} + \big(\dot{\zeta} , \check{\mathfrak{a}}_\alpha\check\zeta\big)_{L^2} +\frac12\big(\mathfrak{b} \dot w,\dot w\big)_{L^2}.\]
The differential inequality~\eqref{energy-estimated} may therefore be reformulated as
\[\frac{\dd}{\dd t} (\widetilde E^0(\dot U) +\widetilde {\bf M})\leq 
 2 K_0C_0 \left(\epsilon+\Upsilon_\F \norm{w}_{\Z^2}^2 \right) \left(\widetilde E^0(\dot U) +\widetilde {\bf M}\right)
 + \sqrt{2K_0}C_0 C_1 \left(\widetilde E^0(\dot U)+\widetilde {\bf M}\right)^{1/2}.\]
We deduce
\[ \big(\widetilde E^0(\dot U)+\widetilde{\bf M}\big)^{1/2} \leq \big(\widetilde E^0(\dot U\id{t=0})+\widetilde{\bf M} \big)^{1/2}e^{\lambda t} + {\bf C}_0 \int_0^t C_1(t') e^{\lambda (t-t')}\dd t' ,\]
where $\lambda,{\bf C}_0$ are as in the statement of the Lemma. Using~\eqref{equivalence-energie-modifiee} and augmenting ${\bf C}_0$ if necessary, the energy estimate is now straightforward.
\end{proof}

\begin{Lemma}\label{L.energy-estimate-diff} Define two tuple of solutions to~\eqref{linearized}, $(\dot U_1,U_1,r_1)$ and $(\dot U_2,U_2,r_2)$, satisfying the same properties as in Lemma~\ref{L.energy-estimate} (with $\check\zeta_1=\check\zeta_2=0$).
Then one has
\[ E^0(\dot U_1-\dot U_2)^{1/2} \ \leq \ {\bf C}_0\ E^0(\dot U_1\id{t=0}-\dot U_2\id{t=0})^{1/2}e^{\lambda t} \ + \ {\bf C}_0\int_0^t f(t') e^{\lambda (t-t')}\dd t', \]
with 
\[ \lambda \ = \ {\bf C}_0 \times \Big(\epsilon+\Upsilon_\F \Norm{w_1}_{L^\infty([0,T];\Z^2)}^2 \Big) ,\ \quad f(t)=\norm{r_1-r_2}_{\X^0\times (\Y^0)^\star}+ \epsilon\norm{\dot U_2}_{(W^{3,\infty}_x)^2} \norm{U_1-U_2}_{\X^2\times\Y^2} \]
and ${\bf C}_0=C(\m,h_0^{-1},k_0^{-1},K,\Norm{U_1}_{L^\infty([0,T];\X^4\times\Y^4)},\Norm{U_2}_{L^\infty([0,T];\X^4\times\Y^4)})$.
\end{Lemma}
\begin{proof}
The difference between the two solutions satisfies the system
\[
\left\{ \begin{array}{l}
\displaystyle\partial_{ t}(\dot{\zeta}_1-\dot{\zeta}_2) \ + \ \partial_x (\dot w_1-\dot w_2) \ =\ r^1_1-r^1_2, \\ \\ 
\mathfrak{b}_1 \partial_{ t} (\dot w_1-\dot w_2) \ + \ \partial_x \mathfrak{a}_1(\dot{\zeta}_1-\dot{\zeta}_2) \ + \ \mathfrak{c}_1 \partial_x (\dot w_1-\dot w_2)\ = \ r^2_1-r^2_2 + r_{\rm diff} ,
\end{array}
\right.
\]
where we denote $\mathfrak{a}_i=\mathfrak{a}[\epsilon\zeta_i,\epsilon w_i]$ (and similarly for $\mathfrak b_i,\mathfrak c_i$), and
\[ r_{\rm diff}\eqdef (\mathfrak{b}_2-\mathfrak{b}_1) \partial_{ t} \dot w_2+(\partial_x\mathfrak{a}_2-\partial_x\mathfrak{a}_1) \dot \zeta_2+(\mathfrak{c}_2-\mathfrak{c}_1) \partial_{ x} \dot w_2 \eqdef \sum_{i=1}^3r_{\rm diff}^{(i)}.\]
The Lemma is a straightforward consequence of Lemma~\ref{L.energy-estimate} (with $\nu=0$), once $r_{\rm diff}$ is estimated.
We focus on the most difficult term, namely $r_{\rm diff}^{(2)}=(\partial_x\mathfrak{a}_2-\partial_x\mathfrak{a}_1) \dot \zeta_2$.

Let $f\in \Y^0$. One has
\begin{multline*}\big(r_{\rm diff}^{(2)},f\big)_{L^2}=\int_{\RR} -\epsilon^2 f \partial_x \left(\Big(G(\epsilon\zeta_2)|w_2|^2-G(\epsilon\zeta_1) |w_1|^2 \Big)\dot\zeta_2\right)\\
-\mu\epsilon^2 f \partial_x\left( \big(\dd_1\R_2^\F[\epsilon\zeta_2,w_2]-\dd_1\R_2^\F[\epsilon\zeta_1,w_1]+\gamma\dd_1\R_1^\F[\epsilon\zeta_2,w_2]-\gamma\dd_1\R_1^\F[\epsilon\zeta_1,w_1]\big)\dot\zeta_2\right) \\
 +\frac{\gamma+\delta}\Bo f\partial_x^2\left( \left(\frac{ 1}{(1+\mu\epsilon^2|\partial_x\zeta_2|^2)^{3/2}} -\frac{ 1}{(1+\mu\epsilon^2|\partial_x\zeta_1|^2)^{3/2}}\right) \partial_x\dot\zeta_2\right) \end{multline*}
where $G(\epsilon\zeta)\eqdef\dfrac{h_1^3 +\gamma h_2^3 }{(h_1 h_2)^3}$.

Since $\norm{G(\epsilon\zeta_2-\epsilon\zeta_1)}_{L^2}\leq \epsilon\norm{\zeta_2-\zeta_1}_{L^2}\times \sup_{y\in[\epsilon\zeta_2,\epsilon\zeta_1]} G'(y)$, it is straightforward that 
\[ \norm{ \partial_x \left(\big(G(\epsilon\zeta_2)|w_2|^2-G(\epsilon\zeta_1) |w_1|^2 \big)\dot\zeta_2\right)}_{L^2}\leq C_0\big(\epsilon\norm{\zeta_1-\zeta_2}_{H^1}+\norm{w_1-w_2}_{H^1}\big)\norm{\dot\zeta_2}_{W^{1,\infty}_x} ,
\]
with $C_0=C(\m,h_0^{-1},\norm{\zeta_1}_{L^{\infty}},\norm{\zeta_2}_{L^{\infty}},\norm{w_1}_{L^{\infty}},\norm{w_2}_{L^{\infty}})$.

Similarly,
\begin{multline*} \frac1{\Bo^{-1}}\norm{\partial_x^2\left( \big((1+\mu\epsilon^2|\partial_x\zeta_2|^2)^{-3/2} -(1+\mu\epsilon^2|\partial_x\zeta_1|^2)^{-3/2}\big) \partial_x\dot\zeta_2\right) }_{L^2} \\
\leq \mu\epsilon^2 C_0\norm{\partial_x\zeta_1-\partial_x\zeta_2}_{\X^1}\norm{\partial_x\dot\zeta_2}_{W^{2,\infty}_x},
\end{multline*}
with $C_0=C(\m, \norm{\partial_x\zeta_1}_{W^{1,\infty}_x},\norm{\partial_x\zeta_2}_{W^{1,\infty}_x})$.

As for the last component, recall $\dd_1\R_i^\F$ is defined in~\eqref{def-Ri}. Proceeding as in the proof of Lemma~\ref{L.linearize}, we obtain
\begin{multline*} 
\gamma^{2-i} \mu\epsilon^2 \big\vert \big( \dd_1\R_i^\F[\epsilon\zeta_2,w_2]\dot\zeta_2-\dd_1\R_i^\F[\epsilon\zeta_1,w_1]\dot\zeta_2,f\big)_{L^2} \big\vert \\
 \leq \epsilon^2 C_0 \times \big(\norm{w_1-w_2}_{\Y^1}+\epsilon \norm{\zeta_1-\zeta_2}_{\Y^1}\big)\norm{\dot\zeta_2}_{\Z^1}\norm{f}_{\Y^0}
 \end{multline*}
 with $C_0=C(\m,h_0^{-1}, \norm{\zeta_1}_{H_x^{{3}}},\norm{\zeta_2}_{H_x^{{3}}},\norm{w_1}_{\Z^1},\norm{w_2}_{\Z^1})$.
 
 Altogether, we find
\[\norm{r_{\rm diff}^{(2)}}_{(\Y^0)^\star} \leq \epsilon C_0 \norm{\dot\zeta_2}_{W^{3,\infty}_x} \norm{U_2-U_1}_{\X^2\times\Y^2},\]
with $C_0=C(\m,h_0^{-1},k_0^{-1},K,\norm{U_1}_{\X^4\times\Y^4},\norm{U_2}_{\X^4\times\Y^4})$.

All the other terms in $r_{\rm diff}$ are estimated in the same way, and Lemma~\ref{L.energy-estimate-diff} now directly follows from Lemma~\ref{L.energy-estimate}.
\end{proof}

\subsection{Well-posedness results; proof of Theorem~\ref{T.WP}}\label{S.WP-conclusion}
In this section we conclude the proof of the main result of the paper, Theorem~\ref{T.WP}, namely the well-posedness of the Cauchy problem for system~\eqref{GN-w}. We first prove in Lemma~\ref{L.WP-linear} the existence and uniqueness of solutions of the linearized system~\eqref{linearized} for smooth data, and provide a uniform energy estimate. A solution of the nonlinear system~\eqref{GN-w} is then constructed using a Picard iteration scheme. Uniqueness, and continuous dependence with respect to the initial data follow from Lemma~\ref{L.energy-estimate-diff}.

\begin{Lemma}\label{L.WP-linear}
Let $\zeta,w,\check\zeta,r^1,r^2\in H^\infty([0,T]\times\RR)$ be such that~\eqref{C.depth},\eqref{C.hyp0},\eqref{C.hyperbolicity} hold. Then for any $\dot U^0\eqdef (\dot\zeta^0,\dot w^0)^\top\in H^\infty_x(\RR)^2$, there exists a unique solution $\dot U\eqdef (\dot\zeta,\dot w)^\top\in H^\infty([0,T]\times\RR)^2$ satisfying~\eqref{linearized} and $\dot U\id{t=0}=\dot U^0$.
\end{Lemma}
\begin{Remark}
One could assume only continuity in time and finite (but large enough) regularity in space on $\zeta,w,\check\zeta,r$.
\end{Remark}
\begin{proof}
We first consider the regularized system introduced in~\eqref{linearized-mollified} and that we 
rewrite (recall that, by Lemma~\ref{L.invertible}, $\mathfrak{b}^{-1}:(\Y^0)^\star\to\Y^0$ is well-defined and continuous) as
\begin{equation}\label{linearized-mollified-bnu}
\left\{ \begin{array}{l}
\displaystyle\partial_{ t}\dot{\zeta}_\nu \ + \ {\sf J}_\nu^2 \partial_x \dot w_\nu \ =\ r^1, \\ \\ 
\partial_{ t} \dot w_\nu \ + \ \mathfrak{b}^{-1}{\sf J}_\nu^2\partial_x \mathfrak{a} \dot{\zeta} \ + \ \mathfrak{b}^{-1}{\sf J}_\nu^2\partial_x \check{\mathfrak{a}}_{\alpha} \check{\zeta} \ + \ \mathfrak{b}^{-1}{\sf J}_\nu \mathfrak{c} {\sf J}_\nu \partial_x \dot w_\nu\ = \ \mathfrak{b}^{-1} r^2.
\end{array}
\right.
\end{equation}
Since ${\sf J}_\nu\eqdef (1-\nu\partial_x^2)^{-1/2}$  is of order $-1$,~\eqref{linearized-mollified-bnu} is a system of ordinary differential equations on $\X^0\times \Y^0$, which is solved uniquely by Cauchy-Lipschitz theorem. More precisely, for any $\nu>0$ and $r=(r^1,r^2)\in C^0([0,T];\X^0\times(\Y^0)^\star)$, $\check{\zeta}\in C^0([0,T];(\X^1)^2)$ and $\dot U^0\in \X^0\times \Y^0$, there exists a unique $\dot U_\nu\eqdef (\dot{\zeta}_\nu,\dot{w}_\nu)^\top\in C^1([0,T];\X^0\times \Y^0)$, solution to~\eqref{linearized-mollified-bnu} with initial data $\dot U\id{t=0}=(\dot\zeta^0,\dot w^0)^\top$. 

Differentiating $N$ times~\eqref{linearized-mollified} and proceeding as in the proof of Lemma~\ref{L.linearize}, one can check that $\partial_x^N \dot U_\nu$ satisfies~\eqref{linearized-mollified} with obvious modifications to $r^1,r^2$ and $\check{\zeta}$. Thus, by the above argument, $\partial_x^N \dot U_\nu\in C^1([0,T];\X^0\times \Y^0)$, and it follows (since $N$ may be chosen arbitrarily large) that $\dot U_\nu\in C^1([0,T];H^\infty_x(\RR))$. In particular, $\partial_t U_\nu\id{t=0}\in H^\infty_x$. 

Applying the above argument to $\partial_t U_\nu$ after differentiating~\eqref{linearized-mollified} with respect to time, one deduces $\partial_t \dot U_\nu\in C^1([0,T];H^\infty_x(\RR))$, and by induction $U_\nu\in H^\infty([0,T]\times\RR)$.

Applying the estimate of Lemma~\ref{L.energy-estimate} to $\partial_x^N U_\nu$ with $N\in\NN$ given, one has
\[ E^0(\partial_x^N \dot U_\nu)\ \leq \ M, \]
with $M= C(\m,h_0^{-1},k_0^{-1},K, T, E^0(\partial_x^N \dot U^0) , \Norm{(\zeta,w,\check\zeta,r)}_{H^\infty([0,T]\times\RR)^6} \big) $, uniform with respect to ${\nu>0}$. 

Let us now consider $V_{\nu,\nu'}=\dot U_\nu-\dot U_{\nu'}$. $V_{\nu,\nu'}$ satisfies~\eqref{linearized-mollified} with $\check\zeta=0,V_{\nu,\nu'}\id{t=0}=0$ and
\[ r_{\nu,\nu'}^1=({\sf J}_\nu^2-{\sf J}_{\nu'}^2) \partial_x\dot w_{\nu'},\quad r_{\nu,\nu'}^2=({\sf J}_\nu^2-{\sf J}_{\nu'}^2)\partial_x \mathfrak{a} \dot{\zeta}_{\nu'} \ + \ ({\sf J}_\nu^2-{\sf J}_{\nu'}^2)\partial_x \check{\mathfrak{a}}_{\alpha} \check{\zeta} \ + \ ({\sf J}_\nu \mathfrak{c} {\sf J}_\nu-{\sf J}_{\nu'} \mathfrak{c} {\sf J}_{\nu'}) \partial_x \dot w_{\nu'}.\]
Since for any $s\in\RR$, $\Norm{{\sf J}_{\nu}}_{H^s_x\to H^s_x}=1$ and $\Norm{{\sf J}_\nu-{\sf J}_{\nu'}}_{H^s_x\to H^s_x} \to 0\ (\nu\to\nu')$ and thanks to the above energy estimates, one has $\norm{r_{\nu,\nu'}}_{\X^0\times(\Y^0)^\star}\to 0 (\nu\to\nu')$.
By Lemma~\ref{L.energy-estimate}, one deduces that $\dot U_\nu$ is a Cauchy sequence of $C^0([0,T];\X^0\times \Y^0)$. Therefore there exists a limit that we denote $\dot U\in C^0([0,T];\X^0\times \Y^0)$, which satisfies the non-regularized (\ie $\nu=0$) system, namely~\eqref{linearized}.

The above energy estimates on $\partial_x^N \dot U_\nu$ being uniform with respect to $\nu$, one has $\dot U\in L^\infty([0,T]; H^\infty_x)$. By~\eqref{linearized}, we deduce $\partial_t \dot U\in L^\infty([0,T]; H^\infty_x)$, and by induction $\dot U\in H^\infty([0,T]\times \RR)$.

Uniqueness of the solution follows when applying the energy estimate of Lemma~\ref{L.energy-estimate} to the difference between two solutions.
\end{proof}

We can now conclude with the proof of our main result, Theorem~\ref{T.WP}.
\begin{proof}[Proof of Theorem~\ref{T.WP}] 
Let us define Friedrichs mollifiers, ${\sf j_{\varkappa}}={\bf 1}(|D|\leq \varkappa)$ and
\[{{\bf U}_n}\id{t=0}={\bf U}_n^0\eqdef\{(\partial^\alpha {\sf j_{2^{n}}} \zeta^0, \partial^\alpha {\sf j_{2^{n}}}w^0)\}_{|\alpha|\leq N}.\] For each $n\geq1$, we define, thanks to Lemma~\ref{L.WP-linear}, ${\bf U}_n\eqdef \{(\zeta^{(\alpha)}_n, w^{(\alpha)}_n)\}_{|\alpha|\leq N}$ as the unique solution to ${{\bf U}_n}\id{t=0}={\bf U}_n^0$ as well as~\eqref{linearized}, where (using the notations and definitions of Lemma~\ref{L.linearize}) $\mathfrak{a}=\mathfrak{a}[\epsilon\zeta_{n-1},\epsilon w_{n-1}]$ and similarly for $\mathfrak{b},\mathfrak{c},r=r^{(\alpha)}$; $\check{\mathfrak{a}}\check\zeta=0$ if $|\alpha|\leq N-1$ and $\check{\mathfrak{a}}\check\zeta=\check{\mathfrak{a}}_\alpha\check \zeta_{n-1}^{\langle\check\alpha\rangle}$ otherwise. Our iteration scheme is initialized with smooth and time-constant ${\bf U}_0={\bf U}_0^0$.

Lemma~\ref{L.WP-linear} defines at each step ${\bf U}_n\in C([0,T_n];H^\infty_x)$, where
\begin{multline*} T_n(h_0',k_0',K',M')\eqdef \max\Big\{T\geq 0,\ \text{ such that } \quad E^N(\zeta_n,w_n)^{1/2}\leq M' E^N(U^0)^{1/2}\\
\text{ and } (\zeta_{n},w_{n})\text{ satisfies~\eqref{C.depth},\eqref{C.hyp0},\eqref{C.hyperbolicity} with $h_0',k_0',K'$} \Big\}.
\end{multline*}
One has $T_n>0$ as soon as $h_0'<h_0$, $k_0'<k_0$, $K'>K$ and $M'>1$, by standard continuity arguments. Let us prove that $T_n$ can be bounded from below, uniformly with respect to $n\in\NN$.

By Lemma~\ref{L.energy-estimate}, we have the energy estimate for $U^{(\alpha)}_n\eqdef (\zeta^{(\alpha)}_n,w^{(\alpha)}_n)^\top$ with any $|\alpha|\leq N$:
\[ E^0(U^{(\alpha)}_n)^{1/2} \ \leq \ {\bf C}_0 \big( E^0(U^{(\alpha)}_n\id{t=0})^{1/2} +\mu\epsilon^2 M' E^N(U^0)^{1/2}\big)e^{\lambda t} \ + \ {\bf C}_0\int_0^t f(t')e^{\lambda(t-t')}\dd t', \]
for any $t\in [0, T_{n-1}(h_0',k_0',K',M')]$ and with 
\[ \lambda \ = \ {\bf C}_0 \times \Big(\epsilon+\Upsilon_\F \Norm{w_{n-1}}_{L^\infty([0,T);\Z^2)}^2 \Big) , \qquad f(t)= \norm{r^{(\alpha)}}_{(\Y^0)^\star}+\mu\epsilon^2 M' E^N(U^0)^{1/2},\]
and where ${\bf C}_0=C(\m,(h_0')^{-1},(k_0')^{-1},K',M')$.

Notice that $\partial^\alpha U_n\neq U^{(\alpha)}_n$ but one can check (differentiating the equations satisfied by $U_n$) that $\partial^\alpha U_n$ satisfies~\eqref{linearized} with a remainder term $\t r^{(\alpha)}[\epsilon\zeta_n,\epsilon w_n,\epsilon\zeta_{n-1},\epsilon w_{n-1}]$ which is estimated identically as in Lemma~\ref{L.linearize}. This yields, for any $t\in [0,\min\{T_{n-1},T_n\}]$,
\[ E^N(U_n)^{1/2}\ \leq \ {\bf C}_0 e^{\lambda t} E^N(U^0)^{1/2} \Big( 1 + {\bf C}_0' M' t \times \big(\epsilon+\Upsilon_\F^{1/2} \norm{w_{n-1}}_{\Z^1}+\Upsilon_\F \norm{w_{n-1}}_{\Z^1}^2\big)\Big), \]
with $\lambda, {\bf C}_0$ as above, and ${\bf C}_0'=C(\m,(h_0')^{-1},M',E^N(U^0))$.

We deduce that there exists $M^\star,\frac1{T^\star}=C(\m,h_0^{-1},k_0^{-1},K,E^N(U^0))$, independent of $n$, such that \[ T_n(h_0/2,k_0/2,2K,M^\star)\geq T^\star/\lambda', \quad \lambda'\eqdef \epsilon+\Upsilon_\F^{1/2} \norm{w^0}_{\Z^1}+\Upsilon_\F \norm{w^0}_{\Z^2}^2,\]
and that for any $t\in [0,T^\star/\lambda']$, one has
\begin{equation}\label{energy-n}
E^N(U_n)^{1/2}\ \leq \ M^\star \ E^N(U^0)^{1/2} .
\end{equation}

Let us now consider $V_n=U_n-U_{n-1}$. Notice first that 
\[ E^j(V_n\id{t=0})=E^j((U_n-U_{n-1})\id{t=0})\lesssim 2^{-2n(N-j)} E^N(U^0).\]
One can control $E^0(V_n)$ from Lemma~\ref{L.energy-estimate-diff}, using the above, the estimate on $r^{(\alpha)}_n-r^{(\alpha)}_{n-1}$ given by Lemma~\ref{L.linearize} as well as the energy estimate~\eqref{energy-n}. Similar estimates on $\partial^\alpha V^n$ for $0\leq |\alpha|\leq 2$ yield
\[ E^2(V_n)^{1/2} \leq {\bf C}_0 2^{-n(N-2)} e^{\lambda' t} + {\bf C}_0 \lambda'\int_0^t E^2(V_{n-1})^{1/2} e^{\lambda' (t-t')}\dd t',\]
with ${\bf C}_0,\lambda'$ as above. Therefore, restricting $T^\sharp\leq T^\star$ if necessary, the sequence $ U_n=U^0+\sum_{j=1}^n V_j$ converges in $C^0([0,T^\sharp/\lambda'];\X^2\times\Y^2)$.

Using that $U_n$ is uniformly bounded in $C^0([0,T^\sharp/\lambda'];\X^N\times\Y^N)$, the logarithmic convexity of Sobolev norms yields that $U_n $ converges strongly in $C^0([0,T^\sharp/\lambda'];\X^{N-1}\times\Y^{N-1})$. The limit $U=\lim_{n\to\infty}U_n$ belongs to $L^{\infty}([0,T^\sharp/\lambda'];\X^N\times\Y^N)\cap C^0([0,T^\sharp/\lambda'];\X^{N-1}\times\Y^{N-1})$ and then by classical argument belong to $ C^0_{\text{w}}([0,T^\sharp/\lambda'];\X^{N}\times\Y^{N}).$ It is now straightforward to check that $U$ satisfies system~\eqref{linearized}, and therefore (by Lemma~\ref{L.linearize})~\eqref{GN-w}. 

By passing to the limit the energy estimate~\eqref{energy-n}, one deduces the energy estimate of the statement. The uniqueness of the solution is a consequence of Lemma~\ref{L.energy-estimate-diff}, applied to the difference between two solutions (see also Proposition~\ref{P.stability}). Theorem~\ref{T.WP} is proved.
\end{proof}

\def\cprime{$'$}

\end{document}